\documentclass[11pt]{amsart}
\usepackage[top=30mm,right=30mm,bottom=30mm,left=20mm]{geometry}
\usepackage{amssymb}
\usepackage{mathrsfs}
\usepackage{enumitem}
\usepackage{hyperref}
\usepackage[normalem]{ulem}
\usepackage{tikz-cd}
\usepackage[only,llbracket,rrbracket]{stmaryrd}

\usepackage[capitalize]{cleveref}

\newtheorem{theorem}{Theorem}[section]
\newtheorem{thm}[theorem]{Theorem}

\newtheorem{prop}[theorem]{Proposition}
\newtheorem*{proposition*}{Proposition}
\newtheorem{lem}[theorem]{Lemma}
\newtheorem{cor}[theorem]{Corollary}
\newtheorem{exa}[theorem]{Example}
\newtheorem{rem}[theorem]{Remark}

\crefname{lem}{Lemma}{Lemmas}

\DeclareMathOperator{\Class}{Class}
\DeclareMathOperator{\Res}{Res}
\DeclareMathOperator{\Ind}{Ind}
\DeclareMathOperator{\pr}{pr}
\DeclareMathOperator{\Pow}{Pow}
\DeclareMathOperator{\Sim}{Sim}
\DeclareMathOperator{\Def}{def}
\DeclareMathOperator{\Irr}{Irr}
\DeclareMathOperator{\Aut}{Aut}
\DeclareMathOperator{\Hom}{Hom}
\DeclareMathOperator{\Cl}{Cl}
\DeclareMathOperator{\diag}{diag}

\DeclareMathOperator{\rk}{rk}

\newcommand{\SL}{\mathrm{SL}}
\newcommand{\GL}{\mathrm{GL}}
\newcommand{\SU}{\mathrm{SU}}
\newcommand{\GU}{\mathrm{GU}}
\newcommand{\Sp}{\mathrm{Sp}}
\newcommand{\SO}{\mathrm{SO}}
\newcommand{\Spin}{\mathrm{Spin}}

\newcommand{\abs}[1]{\lvert #1 \rvert}
\newcommand{\aff}{\mathrm{aff}}
\newcommand{\eps}{\varepsilon}

\newcommand{\St}{\mathsf{St}}
\newcommand{\supp}{\mathsf{supp}}
\newcommand{\ua}{\underline{a}}
\newcommand{\ub}{\underline{b}}
\renewcommand{\setminus}{\smallsetminus}
\newcommand{\tw}[1]{{}^#1\!}

\newcommand{\osymb}[1]{\left[ #1 \right]}
\newcommand{\symb}[1]{\llbracket #1 \rrbracket}
\newcommand{\negsymb}[1]{\langle\langle #1 \rangle\rangle}
\newcommand{\bset}[1]{\left[ #1 \right]}

\newcommand{\dd}{\mathrm{d}}
\newcommand{\der}{\mathrm{der}}
\newcommand{\odd}{\mathrm{od}}
\newcommand{\ev}{\mathrm{ev}}
\newcommand{\spec}{\mathrm{sp}}
\newcommand{\op}{\mathrm{op}}
\newcommand{\Hk}{\mathscr{H}}
\newcommand{\Sym}{\mathfrak{S}}

\newcommand\sd{
	\mathchoice{\mkern1.5mu}{\mkern1.5mu}{}{}%
	{:}%
	\mathchoice{\mkern1.5mu}{\mkern1.5mu}{}{}
}

\newcommand{\bB}{\mathbf{B}}
\newcommand{\bC}{\mathbf{C}}

\newcommand{\bG}{\mathbf{G}}

\newcommand{\bL}{\mathbf{L}}
\newcommand{\bN}{\mathbf{N}}

\newcommand{\bT}{\mathbf{T}}

\newcommand{\bZ}{\mathbf{Z}}

\newcommand{\FF}{\mathbb{F}}
\newcommand{\GG}{\mathbb{G}}
\newcommand{\NN}{\mathbb{N}}
\newcommand{\QQ}{\mathbb{Q}}
\newcommand{\CC}{\mathbb{C}}
\newcommand{\ZZ}{\mathbb{Z}}

\newcommand{\cB}{\mathcal{B}}
\newcommand{\cD}{\mathcal{D}}
\newcommand{\cE}{\mathcal{E}}
\newcommand{\cF}{\mathcal{F}}
\newcommand{\cH}{\mathcal{H}}
\newcommand{\cI}{\mathcal{I}}
\newcommand{\cO}{\mathcal{O}}
\newcommand{\cP}{\mathcal{P}}
\newcommand{\cQ}{\mathcal{Q}}
\newcommand{\cR}{\mathcal{R}}
\newcommand{\cS}{\mathcal{S}}
\newcommand{\cX}{\mathcal{X}}

\newcommand{\sA}{\mathsf{A}}
\newcommand{\sB}{\mathsf{B}}
\newcommand{\sC}{\mathsf{C}}
\newcommand{\sD}{\mathsf{D}}
\newcommand{\sS}{\mathsf{S}}

\newcommand{\rN}{\mathrm{N}}

\newcommand{\rZ}{\mathbf{Z}}

\DeclareFontFamily{U}{mathx}{\hyphenchar\font45}
\DeclareFontShape{U}{mathx}{m}{n}{
      <5> <6> <7> <8> <9> <10>
      <10.95> <12> <14.4> <17.28> <20.74> <24.88>
      mathx10
      }{}
\DeclareSymbolFont{mathx}{U}{mathx}{m}{n}
\DeclareFontSubstitution{U}{mathx}{m}{n}
\DeclareMathAccent{\wc}{0}{mathx}{"71}

\numberwithin{equation}{section}

\begin{document}

\title[Character bounds and Thompson's conjecture]
{Character bounds for regular semisimple elements and asymptotic results on Thompson's conjecture}

\author{Michael Larsen}
\email{mjlarsen@indiana.edu}
\address{Department of Mathematics\\
    Indiana University\\
    Bloomington, IN 47405\\
    U.S.A.}

\author{Jay Taylor}
\email{jay.taylor@manchester.ac.uk}
\address{
Department of Mathematics\\
The University of Manchester\\
Oxford Road\\
Manchester, M13 9PL\\
U.K.}

\author{Pham Huu Tiep}
\email{pht19@math.rutgers.edu}
\address{Department of Mathematics\\
    Rutgers University\\
    Piscataway, NJ 08854\\
    U.S.A.}
\thanks{The first author was partially supported by the NSF 
grant DMS-2001349. The third author gratefully acknowledges the support of the NSF (grants
DMS-1840702 and DMS-2200850), the Joshua Barlaz Chair in Mathematics, and the Charles Simonyi Endowment at the 
Institute for Advanced Study (Princeton).}

\begin{abstract}
For every integer $k$ there exists a bound $B=B(k)$ such that if the characteristic polynomial of $g\in \SL_n(q)$ is the product of $\le k$ pairwise distinct monic irreducible polynomials over $\FF_q$,
then every element $x$ of $\SL_n(q)$ of support at least $B$ is the product of two conjugates of $g$.  We prove this and analogous results for the other classical groups over finite fields; in the orthogonal and symplectic cases,
the result is slightly weaker.  With finitely many exceptions $(p,q)$, in the special case that $n=p$ is prime, if $g$ has order $\frac{q^p-1}{q-1}$,
then every non-scalar element $x \in \SL_p(q)$ is the product of two conjugates of $g$. The proofs use the Frobenius formula together with upper bounds for values of unipotent and quadratic unipotent characters in finite classical groups.
\end{abstract}

\maketitle

\tableofcontents

\section{Introduction}
A conjecture of Thompson states that each finite simple group $G$ contains a conjugacy class $C \subseteq G$ such that $C^2 = G$. Inspired by this, we would like to study an asymptotic version of Thompson's conjecture when $G$ is one of the finite classical groups $\SL_n(q)$, $\SU_n(q)$, $\Sp_{2n}(q)$, $\SO_{2n+1}(q)$, and $\SO_{2n}^{\pm}(q)$, which are all closely related to simple groups.  
This asymptotic version treats target elements of sufficiently large support. We prove that regular semisimple conjugacy classes $C=g^G$ satisfy our asymptotic version of Thompson's conjecture whenever the characteristic polynomial of $g$ is close to being irreducible.

If $G = \mathrm{Cl}(V)$ is a finite classical group, with natural module $V = \FF_q^n$, we define the \emph{support} $\supp(x)$ of an element $x \in G$ to be the codimension of the largest eigenspace of $x$ on $V \otimes _{\FF_q}\overline{\FF_q}$. The following is one of our main results and generalizes \cite[Theorem 7.8]{LT}.

\begin{thm}\label{main1}
For all integers $k \in \ZZ_{\geqslant 1}$ there exists an explicit constant $B=B(k) > 0$ such that for all $n \in \ZZ_{\geqslant 1}$ and all prime powers $q$ the following statement holds. Suppose $G$ is one of $\SL_n(q)$, $\SU_n(q)$, $\Sp_{2n}(q)$, $\SO_{2n+1}(q)$, and $\SO_{2n}^{\pm}(q)$, and $g \in G$ is a regular semisimple element whose characteristic polynomial on the natural module is a product of $k$ pairwise distinct irreducible polynomials, of pairwise distinct degrees if $G$ is of type $\Sp$ or $\SO$. Then $g^G\cdot g^G$ contains every element $x \in [G,G]$ with $\supp(x) \geqslant B$.
\end{thm}

In fact, in the $\Sp$ and $\SO$ cases, we prove a slightly stronger result, see Theorem \ref{almost-thompson-BCD}. 
We also note that the assumption $x \in [G,G]$ is superfluous in the $\SL$, $\SU$, and $\Sp$ cases (since $G=[G,G]$ in these
cases, aside from known exceptions with $n \leq 3$), but necessary in the $\SO$ case (since in this case $[G,G]$ has index $2$ in $G$ and so $g^G \cdot g^G \subseteq [G,G]$). 

In a special family of particularly favorable cases, Theorem \ref{sl-coxeter} shows that all non-central elements of $G$ lie in $C^2$. 

\smallskip
If $\Irr(G)$ denotes the set of the complex irreducible characters of $G$, then the well-known formula of Frobenius states that $x \in G$ is contained in $g^G\cdot g^G$ if and only if
\begin{equation}\label{frob}
\sum_{\chi \in \Irr(G)} \frac{\chi(g)^2\chi(x^{-1})}{\chi(1)} \neq 0
\end{equation}
To show that this is the case we need sufficiently good upper bounds on $\abs{\chi(g)}$. To get these we realise our group as the fixed point subgroup $\bG^F$ of a Frobenius endomorphism $F : \bG \to \bG$ on a connected reductive algebraic group $\bG$ and use the Deligne--Lusztig theory \cite{DL}.

To illustrate our techniques suppose $\bG^F = \Sp_{2n}(q)$ or $\SO_{2n+1}(q)$. To each element $w$ of  the Weyl group 
$W \cong C_2\wr\Sym_n$ of $\bG$, Deligne and Lusztig have associated a virtual character $R_w$ of $\bG^F$ whose irreducible constituents are called {\it unipotent} characters. The subspace $\Class_0(\bG^F) \subseteq \Class(\bG^F)$ of all $\CC$-valued class functions spanned by $\{R_w \mid w \in W\}$ is the space of uniform unipotent class functions.

If $\chi$ is a unipotent character then the (uniform) projection of $\chi$ onto $\Class_0(\bG^F)$ is known to have the form
\begin{equation*}
\cR_{f_{\chi}} = \frac{1}{\abs{W}} \sum_{w \in W} f_{\chi}(w)R_w
\end{equation*}
for some class function $f_{\chi} \in \Class(W)$, which is not irreducible in general. If $g \in \bG^F$ is semisimple then its characteristic function is {\it uniform}, which means $\chi(g) = \cR_{f_{\chi}}(g)$ for all $\chi$.

Our first step towards understanding $\cR_{f_{\chi}}(g)$ is to show that $f_{\chi}$ satisfies a version of the {\it recursive Murnaghan--Nakayama rule} (or MN-rule), see \cref{prop:phi-MN-rule}. This is a consequence of a fundamental combinatorial result of Asai \cite{ASO}, \cite{ASp} that relates the classical MN-rule for the irreducible characters of $W$ and Lusztig's Fourier transform, whose proof we give in \cref{sec:asai-result}. If $g \in \bG^F$ is a regular semisimple element such that $\bC_{\bG}^{\circ}(g)$ is a torus of type $w$ then $\abs{\chi(g)} = \abs{\cR_{f_{\chi}}(g)} = \abs{f_{\chi}(w)}$ and we recover the MN-rule of L\"ubeck--Malle \cite[Thm.~3.3]{LuMa}.

By working with uniform projections we may apply these results to non-unipotent characters, and we do so to obtain bounds on $\abs{\chi(g)}$ whenever $\chi$ is a {\it quadratic unipotent character} and the {\it cycle type} of $g \in \bG^F$ is a product of $k \geqslant 1$ pairwise distinct cycles (see \S\ref{quadratic-unipotent} for precise definitions). In fact, following an argument of Larsen--Shalev \cite{LaSh2} we obtain a bound on $\abs{\chi(g)}$ that depends only on $k$, see Corollary \ref{bcd-bound} in the quadratic unipotent case. Bounds for arbitrary characters, involving $k$ and $n$, are given in \cref{glu-bound2} and \cref{bcd-bound-distinct}. Using the Mackey formula for tori we also give a bound in \cref{prop:mackey-bound} that works for any irreducible character of any finite reductive algebraic group.

Now treating all characters $\chi$ in \eqref{frob} involves a reduction to Levi subgroups using Deligne--Lusztig induction. The characters that contribute to the sum the most have a heavily restricted form. Our character bounds allow us to obtain sufficiently good bounds on the sum. Aside from 
these immediate applications, we believe our character bounds for regular semisimple elements will be useful in other situations as well. 

\section{Combinatorics}\label{sec:combinatorics}
For any set $X$ we will denote by $\Pow(X)$ the set of all subsets of $X$ of \emph{finite cardinality}. This is naturally an $\mathbb{F}_2$-vector space under symmetric difference, which we denote by $A\ominus B = (A\cup B) - (A \cap B)$ for any $A,B \in \Pow(X)$. Moreover, it is equipped with a nondegenerate symmetric bilinear form $\langle -,-\rangle : \Pow(X) \times \Pow(X) \to \ZZ/2\ZZ$ given by $\langle A,B\rangle = \abs{A \cap B} \pmod{2}$. If $e \in \ZZ$ then we let $\Pow_e(X) = \{A \in \Pow(X) \mid \abs{A} \equiv e \pmod{2}\}$.

We set $X^{(2)} = X \times \ZZ/2\ZZ$. If $X \subseteq Y$ then $X^{(2)} \subseteq Y^{(2)}$. Elements of $X^{(2)}$ will be identified with their representatives in $X \times \{0,1\}$. We denote by $\delta : X^{(2)} \to \ZZ/2\ZZ$ the projection onto the second factor.

Set $\NN = \{1,2,3,\dots\}$ and $\NN_0 = \NN \cup \{0\}$. If $A \in \Pow(\NN_0)$ then we define the \emph{rank} of $A$ to be $\rho(A) = \sum_{a \in A} a - \binom{\abs{A}}{2}$. For each $k \in \mathbb{N}_0$ we define a map $(-)^{\rightarrow k} : \Pow(\NN_0) \to \Pow(\NN_0)$ by setting $A^{\rightarrow k} = \{0,\dots,k-1\} \sqcup \{a+k \mid a \in A\}$. This gives an equivalence relation ${\sim}$ by setting $A\sim B$ if $A = B^{\rightarrow k}$ or $B = A^{\rightarrow k}$ for some $k \in \NN_0$.

We denote by $[A]$ the equivalence class containing $A$ and $\cB = \Pow(\NN_0)/{\sim}$ the set of all equivalence classes. These are called $\beta$-sets. The rank $\rho([A]) = \rho(A)$ of $[A] \in \cB$ is well defined. If $n \in \NN_0$ then $\cB_n \subseteq \cB$ denotes all $\beta$-sets of rank $n$.

\subsection{Arrays}\label{sec:arrays}
The elements of $\Pow(\ZZ^{(2)})$ will be called \emph{arrays}. They will be identified with their images under the natural bijection $\Pow(\ZZ^{(2)}) \to \Pow(\ZZ)\times \Pow(\ZZ)$ given by $X \mapsto (X^0,X^1)$, where $X^i = \{x \in \ZZ \mid (x,i) \in X\}$. We say $X^0$ is the \emph{top row} of $X$ and $X^1$ the \emph{bottom row} of $X$.

Following Lusztig \cite{Lus}, and modifying the notation of \cite{W}, we consider elements of $\tilde{\cP} = \Pow(\NN_0^{(2)}) \subseteq \Pow(\ZZ^{(2)})$. Recall the \emph{rank} of $X \in \tilde{\cP}$ is defined to be
\begin{equation}\label{eq:rank-definition}
\rk(X) = \sum_{x^0\in X^0}x^0 + \sum_{x^1\in X^1}x^1 - \left\lfloor \left(\frac{\abs{X}-1}{2}\right)^2 \right\rfloor = \rho(X^0) + \rho(X^1) + \left\lfloor \left(\frac{\Def{X}}{2}\right)^2 \right\rfloor
\end{equation}
where $\Def(\Lambda) = \abs{X^0} - \abs{X^1}$ is the \emph{defect} of $\Lambda$.

For $d \in \ZZ$ we let $\tilde{\cP}^d \subseteq \tilde{\cP}$ be the set of arrays of defect $d$. We set $\tilde{\cP}^{\odd} = \bigsqcup_{d \in \ZZ} \tilde{\cP}^{2d+1}$ and $\tilde{\cP}^{\ev} = \bigsqcup_{d \in \ZZ} \tilde{\cP}^{2d}$ so that $\tilde{\cP} = \tilde{\cP}^{\odd}\sqcup \tilde{\cP}^{\ev}$. For $n \in \NN_0$ we let $\tilde{\cP}_n \subseteq \tilde{\cP}$ be the set of arrays of rank $n$. We then set $\tilde{\cP}_n^d = \tilde{\cP}_n \cap \tilde{\cP}^d$, $\tilde{\cP}_n^{\odd} = \tilde{\cP}_n \cap \tilde{\cP}^{\odd}$, and $\tilde{\cP}_n^{\ev} = \tilde{\cP}_n \cap \tilde{\cP}^{\ev}$.

Each $X \in \tilde{\cP}$ gives rise to the following subsets of $\NN_0$: $X^{\cup} := X^0 \cup X^1$, $X^{\cap} = X^0 \cap X^1$, $X^{\ominus} = X^0 \ominus X^1$. We let $\Sim(X) = \{Y \in \tilde{\cP} \mid Y^{\cup} = X^{\cup}$ and $Y^{\cap} = X^{\cap}\}$ be the \emph{similarity class} of $X$. All elements of $\Sim(X)$ have the same rank but different defects.

\subsection{Fourier transform}
We associate to each $X \in \tilde{\cP}$ an associated \emph{special array}
\begin{equation}\label{eq:spec-id}
X_{\spec} := \{(y,0),(y,1) \mid y \in X^{\cap}\} \cup \{(x,\abs{X^{\ominus}}+\langle X^{\ominus}, \{0,\dots,x\} \rangle) \mid x \in X^{\ominus}\}.
\end{equation}
Note that if $Y \in \Sim(X)$ then $Y_{\spec} = X_{\spec}$. Moreover, if $x \in X^{\ominus}$ then we have
\begin{equation}\label{eq:in-spec-lower}
\langle X_{\spec}, \{x\}\rangle \equiv \abs{X^{\ominus}}+\langle X^{\ominus}, \{0,\dots,x\} \rangle.
\end{equation}
The defect $\Def(X_{\spec}) \in \{0,1\}$ of this array satisfies $\Def(X_{\spec}) \equiv \abs{X^{\ominus}} \pmod{2}$. Thus we have an integer
\begin{equation*}
s(X) = 2^{(\abs{X^{\ominus}}-\Def(X_{\spec}))/2} \in \NN.
\end{equation*}

We have a map ${}^{\sharp} : \tilde{\cP} \to \Pow(\NN_0)$ given by $Y^{\sharp} = Y^1\ominus Y_{\spec}^1 \subseteq Y^{\ominus}$. This restricts to a bijection ${}^{\sharp} : \Sim(X) \to \Pow(X^{\ominus})$ for any $X \in \tilde{\cP}$. With this we define a $\CC$-linear map $\tilde{\cR} : \CC[\tilde{\cP}] \to \CC[\tilde{\cP}]$ by setting
\begin{equation*}
\tilde{\cR}(X) = \frac{1}{s(X)}\sum_{Y \in \Sim(X)} (-1)^{\langle X^{\sharp},Y^{\sharp}\rangle}Y.
\end{equation*}
Here $\CC[\tilde{\cP}]$ denotes the free $\CC$-module with basis $\tilde{\cP}$ and $\langle-,-\rangle$ is the symmetric $\mathbb{F}_2$-bilinear form defined above. Up to scaling this is the Fourier transform of the abelian group $\Pow(X^{\ominus})$.

\begin{rem}
{\em We briefly make a few comments regarding the conventions and definitions in \cite{Lus}. If $M = X^1 \cap X^{\ominus}$ and $M_0 = X_{\spec}^1 \cap X^{\ominus}$ then we have $X^{\sharp} = M \ominus M_0$, which is denoted by $M^{\sharp}$ in \cite[\S4.5, \S4.6]{Lus}. We have another set $M_0' = X_{\spec}^0 \cap X^{\ominus} = X^{\ominus} - M_0$. If $\abs{X^{\ominus}}$ is even then $M_0$ and $M_0'$ are distinguished by the condition that $\sum_{x \in M_0} x < \sum_{x \in M_0'} x$.

This definite choice of $M_0$ over $M_0'$, using $X_{\spec}$, is used in \cite[\S4.18]{Lus}. Moreover, our definition of $X_{\spec}$ agrees with the convention in \cite[17.2]{LusCS}, that the smallest entry of $X^{\ominus}$ occurs in the lower row of $X_{\spec}$. This is different to the definition of a distinguished symbol given in \cite[4.4.3]{GeMa}.}
\end{rem}

\subsection{Hooks}
Each $(d,i) \in \ZZ^{(2)}$ determines an injective function $\mathcal{D}_{d,i} : \ZZ^{(2)} \to \ZZ^{(2)}$ given by $\mathcal{D}_{d,i}((x,j)) = (x-d,i+j)$. This induces a map $\mathcal{D}_{d,i} : \Pow(\ZZ^{(2)}) \to \Pow(\ZZ^{(2)})$ that is $\mathbb{F}_2$-linear. For any $X,H \in \Pow(\ZZ^{(2)})$ we define
\begin{equation*}
X\setminus_{d,i}H = X\ominus H \ominus \mathcal{D}_{d,i}(H) \in \Pow(\ZZ^{(2)}).
\end{equation*}
We write this as $X\setminus H$ when $(d,i)$ is clear from the context or by $X\setminus_{d,i} \lambda = X\setminus \lambda$ when $H = \{\lambda\}$ is a singleton. For brevity we write the map $\mathcal{D}_{0,1}$ simply as $(-)^{\op}$.

If $X \in \tilde{\cP}$ then the elements of the set
\begin{equation*}
\mathscr{H}_{d,i}(X) = \{\lambda \in X \mid \mathcal{D}_{d,i}(\lambda) \in \NN_0^{(2)} - X\}
\end{equation*}
are called the \emph{$(d,i)$-hooks} of $X$. Elsewhere in the literature $(d,0)$-hooks and $(d,1)$-hooks, with $d > 0$, are called $d$-hooks and $d$-cohooks respectively. Following \cite{W} we define the \emph{leg parity} of $\lambda = (x,j) \in \mathscr{H}_{d,i}(X)$ to be
\begin{equation*}
l_{d,i}(\lambda,X) = \langle \{0,\dots,x\},X^j\rangle + \langle \{0,\dots,x-d\},(X\setminus_{d,i}\lambda)^{i+j}\rangle.
\end{equation*}
If $d > 0$ and $i = 0$ then this has the same parity as the usual notion of the leg length of a hook. This is also easily seen to agree with the definitions in \cite[4.4.10]{GeMa}.

\begin{rem}
{\em The similarity relation can be rephrased in terms of $(0,1)$-hooks. Specifically, we have bijections $\mathscr{H}_{0,1}(X) \to X^{\ominus}$ and $\mathscr{H}_{0,1}(X) \to \Sim(X)$ given by $H \mapsto H^{\cup} = H^{\ominus}$ and $H \mapsto X\setminus_{0,1} H$ respectively. Moreover, if $Y = X\setminus_{0,1} H$ then we have $Y^j = X^j \ominus H^{\ominus}$ for any $j \in \{0,1\}$.}
\end{rem}

Recall that $\delta : X^{(2)} \to \ZZ/2\ZZ$ is the second projection map. For a pair $(d,i) \in \ZZ^{(2)}$, with $d \neq 0$, and $j \in \{0,1\}$ we define a $\CC$-linear map $\tilde{\cH}_{d,i}^j : \CC[\tilde{\cP}] \to \CC[\tilde{\cP}]$ by setting
\begin{equation*}
\tilde{\cH}_{d,i}^j(X) = \sum_{\lambda \in \mathscr{H}_{d,i}(X)} (-1)^{j\delta(\lambda) + l_{d,i}(\lambda,X)}X\setminus_{d,i}\lambda
\end{equation*}
for any $X \in \tilde{\cP}$. Note that if $Y = X\setminus_{d,i}\lambda$, with $\lambda \in \mathscr{H}_{d,i}(X)$, then $\Def(Y) = \Def(X)-2i(-1)^{\delta(\lambda)}$ so $\Def(Y) \equiv \Def(X) \equiv \abs{X^{\ominus}} \pmod{2}$.

\subsection{Symbols}
The map $(-)^{\rightarrow k} : \Pow(\NN_0) \to \Pow(\NN_0)$ defined above, for $k \in \NN_0$, extends to map $\tilde{\cP} \to \tilde{\cP}$ given by $(A,B)^{\rightarrow k} = (A^{\rightarrow k}, B^{\rightarrow k})$. As before this yields an equivalence relation on $\tilde{\cP}$. We denote by $\osymb{X}$ the equivalence class containing $X \in \tilde{\cP}$ and $\tilde{\cS}$ the set of equivalence classes. The equivalence class $\osymb{X}$ is called an (ordered) symbol.

Given $\lambda = (x,j) \in \ZZ^{(2)}$ and $k \in \ZZ$ let $\lambda+k = (x+k,j)$. If $X \in \tilde{\cP}$ and $k \in \NN_0$ then it is readily checked that:
\begin{itemize}
	\item $\Hk_{d,i}(X^{\rightarrow k}) = \{\lambda+k \mid \lambda \in \Hk_{d,i}(X)\}$,
	\item $X^{\rightarrow k}\setminus_{d,i} (\lambda+k) = (X\setminus_{d,i}\lambda)^{\rightarrow k}$
	\item $l_{d,i}(\lambda+k,X^{\rightarrow k}) = l_{d,i}(\lambda,X)$.
\end{itemize}
Thus the maps $\tilde{\cR}$ and $\tilde{\cH}_{d,i}^j$ preserve the kernel of the natural quotient map $\CC[\tilde{\cP}] \to \CC[\tilde{\cS}]$ and so factor through endomorphisms of $\CC[\tilde{\cS}]$ which we denote in the same way.

Recall that $X \in \tilde{\cP}$ is degenerate if $X^0 = X^1$. We let
\begin{equation*}
\cS = \{\symb{X} \mid X \in \tilde{\cP}\text{ and }X^0 \neq X^1\} \cup \{\symb{X}_{\pm} \mid X \in \tilde{\cP}\text{ and }X^0 = X^1\}
\end{equation*}
where $\symb{X} = \{\osymb{X},\osymb{X^{\op}}\}$. We take the rank and defect of $\symb{X} \in \cS$ to be $\rk(\symb{X}) = \rk(X)$ and $\Def(\symb{X}) = \abs{\Def(X)}$. We can then partition $\cS$ with respect to the rank and defect as in \cref{sec:arrays}.

\section{More combinatorics}\label{sec:asai-result}
\subsection{A combinatorial result of Asai}
We will now prove the following fundamental combinatorial observation of Asai that relates the maps $\tilde{\cR}$ and $\tilde{\cH}_{d,i}^j$. This was first stated by Asai in \cite[Lem.~2.8.3]{ASO} and \cite[Lem.~1.5.3]{ASp} where it is left as a ``direct calculation''. However, as pointed out by L\"ubeck--Malle \cite[\S3.4]{LuMa} a sign is missing in the statement in \cite{ASO}.

Waldspurger also states a version of this result \cite[\S2]{W}, where it is left as ``un calcul fastidieux mais \'el\'ementaire'', but the conventions of \cite{W} are different leading to a different statement. Specifically the analogue of our map $\tilde{\cR}$, denoted by $\mathcal{F}$ in \cite{W}, is not the same, as can be seen by evaluating it on the similarity class $\{(\emptyset, \{1,2\}),(\{1\}, \{2\}),(\{2\},\{1\}),(\{1,2\},\emptyset)\}$.

Aside from being an important ingredient in our work here, Asai's combinatorics form a core basis for the block theory of finite reductive groups and solutions of Lusztig's conjecture on almost characters for classical groups. In this second application the correctness of signs is crucial. In light of the importance of Asai's statements, we provide some details regarding the proof.

We note that some of the main ideas of the proof have recently appeared in \cite[Prop. 6]{Malle}, where a weaker statement, leading essentially to \cref{prop:phi-MN-rule}, is proved. Unfortunately there are several errors in the proofs of \cite{Malle} that are corrected by our arguments here.

\begin{thm}[Asai]\label{thm:asai-fourier-commutes}
For any $0 \neq d \in \ZZ$ we have the following equalities of linear endomorphisms of $\CC[\tilde{\cP}]$:
\begin{enumerate}[label={\normalfont(\roman*)}]
	\item $\tilde{\cH}_{d,0}^0\circ \tilde{\cR} = \tilde{\cR}\circ \tilde{\cH}_{d,0}^0$,
	\item $\tilde{\cH}_{d,0}^1\circ \tilde{\cR} = -\Theta\circ\tilde{\cR}\circ \tilde{\cH}_{d,1}^0$,
\end{enumerate}
where $\Theta : \CC[\tilde{\cP}] \to \CC[\tilde{\cP}]$ is the $\CC$-linear map defined by $\Theta(X) = (-1)^{\Def(X)}X$.
\end{thm}

From now on $0 \neq d \in \mathbb{Z}$ and $i \in \{0,1\}$ are fixed. Given $x \in \mathbb{N}$ we denote by $\mathscr{X}(X,x)$ the set of pairs $(H,\lambda)$ with $H \subseteq \mathscr{H}_{0,1}(X)$, and $\lambda \in \mathscr{H}_{d,0}(X\setminus H)$ is such that $\lambda = (x,j)$ for some $j \in \{0,1\}$. Correspondingly we denote by $\mathscr{Y}(X,x)$ the set of pairs $(\mu,G)$ where $G \subseteq \mathscr{H}_{0,1}(X)$, and $\mu \in \mathscr{H}_{d,i}(X\setminus G)$ satisfies $\mu = (x,j)$ for some $j \in \{0,1\}$.

Given $X \in \tilde{\cP}$ we then have
\begin{equation}\label{eq:lhs-asai}
\tilde{\cH}_{d,0}^i(\tilde{\cR}(\Lambda)) = \sum_{x \in \mathbb{N}}\sum_{(H,\lambda) \in \mathscr{X}(X,x)} \frac{1}{s(X)}(-1)^{i\delta(\lambda)+\langle X^{\sharp},(X\setminus H)^{\sharp}\rangle + l_{d,0}(\lambda,X\setminus H)}X \setminus H \setminus \lambda
\end{equation}
and
\begin{equation}\label{eq:rhs-asai}
\tilde{\cR}(\tilde{\cH}_{d,i}^0(\Lambda)) = \sum_{x \in \mathbb{N}} \sum_{(\mu,G) \in \mathscr{Y}(\Lambda,x)} \frac{1}{s(X\setminus \mu)}(-1)^{\langle (X\setminus \mu)^{\sharp},(X\setminus \mu \setminus G)^{\sharp}\rangle + l_{d,i}(\mu,X)}X \setminus \mu \setminus G.
\end{equation}
Before proving (i) of \cref{thm:asai-fourier-commutes} we start with a lemma.

\begin{lem}\label{lem:fourier-comparison}
Assume $i=0$ and $(H,\lambda) \in \mathscr{X}(X,x)$ and $(\mu,G) \in \mathscr{Y}(X,x)$ are two terms satisfying one of the following:
\begin{enumerate}[label={\normalfont(\roman*)}]
	\item $(\mu,G^{\ominus}) = (\lambda,H^{\ominus})$ and $\langle \{x,x-d\}, H^{\ominus} \rangle = 0$,
	\item $(\mu,G^{\ominus}) = (\lambda^{\op},H^{\ominus}\ominus\{x,x-d\})$ and $\langle \{x,x-d\}, H^{\ominus} \rangle = \langle \{x,x-d\}, X^{\ominus} \rangle = 1$.
\end{enumerate}
Then we have
\begin{equation*}
(-1)^{\langle X^{\sharp},(X\setminus H)^{\sharp}\rangle + l_{d,0}(\lambda,X\setminus H)} = (-1)^{\langle (X\setminus \mu)^{\sharp},(X\setminus \mu \setminus G)^{\sharp}\rangle + l_{d,0}(\mu,X)}.
\end{equation*}
\end{lem}

\begin{proof}
Let $Y = X\setminus_{0,1} H$, $U = X\setminus_{d,0} \mu$, and $V = U \setminus_{0,1} G$. We have 
$$\langle X^{\sharp},Y^{\sharp}\rangle + \langle U^{\sharp},V^{\sharp}\rangle = \abs{X^1}+\abs{U^1}+\abs{X_{\spec}^1}+\abs{U_{\spec}^1}+\langle X^{\sharp},H^{\ominus}\rangle + \langle U^{\sharp},G^{\ominus} \rangle$$ 
because $V^{\sharp} = U^{\sharp} \ominus G^{\ominus}$ and $Y^{\sharp} = X^{\sharp} \ominus H^{\ominus}$. The sum of the first four terms is $0$ because $\abs{X^1} = \abs{U^1}$, as $\mu$ is a $(d,0)$-hook, and a straightforward check shows that $\abs{U_{\spec}^1} = \abs{X_{\spec}^1}$. We thus have
\begin{equation*}
\langle X^{\sharp},Y^{\sharp}\rangle + \langle U^{\sharp},V^{\sharp}\rangle = \langle U^1, H^{\ominus}\ominus G^{\ominus}\rangle + \langle U_{\spec}^1, H^{\ominus}\ominus G^{\ominus}\rangle + \langle X^1\ominus U^1, H^{\ominus}\rangle + \langle X_{\spec}^1\ominus U_{\spec}^1, H^{\ominus}\rangle.
\end{equation*}
As $U^{\delta(\mu)} = X^{\delta(\mu)} \ominus \{x,x-d\}$ and $Y^{\delta(\lambda)} = X^{\delta(\lambda)} \ominus H^{\ominus}$, it is straightforward to see that
\begin{equation*}
l_{d,0}(\lambda,Y) + l_{d,0}(\mu,X) = \langle \{0,\dots,x\}\ominus\{0,\dots,x-d\},X^{\delta(\lambda)}\ominus X^{\delta(\mu)}\ominus H^{\ominus} \rangle.
\end{equation*}
We have to show the sum of these two expressions is $0$.

For any $z \in U^{\ominus} \cap X^{\ominus}$ it follows from \cref{eq:in-spec-lower} that
\begin{equation*}
\langle U_{\spec}^1\ominus X_{\spec}^1,\{z\}\rangle = \langle U^{\ominus}\ominus X^{\ominus},\{0,\dots,z\}\rangle = \langle \{0,\dots,x\}\ominus \{0,\dots,x-d\},\{z\}\rangle
\end{equation*}
because $U^{\ominus} \ominus X^{\ominus} = \{x-d,x\}$ so $\abs{X^{\ominus}} \equiv \abs{U^{\ominus}} \pmod{2}$. Thus for any subset $Z \subseteq U^{\ominus}\cap X^{\ominus}$ we get
\begin{equation*}
\langle U_{\spec}^1\ominus X_{\spec}^1,Z\rangle = \langle \{0,\dots,x\}\ominus \{0,\dots,x-d\},Z\rangle.
\end{equation*}

Note that $X^1\ominus U^1$ is either $\emptyset$ or $\{x,x-d\}$ depending on whether $\delta(\mu)=0$ or $1$. Hence the statement clearly follows if (i) holds. So assume (ii) holds. As $H^{\ominus} \subseteq X^{\ominus}$ we must have $H^{\ominus} \cap \{x,x-d\} = X^{\ominus} \cap \{x,x-d\}$. We will assume this is $\{x\}$ as the case where it is $\{x-d\}$ is identical. This means $x-d \not\in X^{\cup}$ because $x-d \not\in X^{\cap}$.

Clearly $x \not\in U^{\cup}$ and $x-d \in U^{\ominus}$ so
\begin{equation*}
\langle U^1,H^{\ominus}\ominus G^{\ominus}\rangle + \langle X^1\ominus U^1,H^{\ominus}\rangle = \langle U^1,\{x-d\}\rangle + \langle X^1,\{x\}\rangle = 0.
\end{equation*}
Now $\langle U_{\spec}^1,H^{\ominus}\ominus G^{\ominus}\rangle + \langle X_{\spec}^1\ominus U_{\spec}^1,\{x\}\rangle = \langle U_{\spec}^1,\{x-d\}\rangle + \langle X_{\spec}^1,\{x\}\rangle$ is the same as
\begin{equation*}
\langle \{x,x-d\},\{0,\dots,x-d\}\rangle + \langle X^{\ominus},\{0,\dots,x\}\ominus \{0,\dots,x-d\} \rangle,
\end{equation*}
and the first term is equivalent to $\langle \{x\},\{0,\dots,x\}\ominus\{0,\dots,x-d\}\rangle $. Adding this to
\begin{equation*}
\langle X_{\spec}^1\ominus U_{\spec}^1,H^{\ominus}\ominus\{x\}\rangle
\end{equation*}
gives the statement because $H^{\ominus} \ominus \{x\} \subseteq U^{\ominus}\cap X^{\ominus}$.
%
\end{proof}

\begin{proof}[Proof of \cref{thm:asai-fourier-commutes}(i)]
We can assume $x \in X^{\cup}$ since otherwise $\mathscr{X}(X,x)$ and $\mathscr{Y}(X,x)$ are empty, and there is nothing to show. Similarly, this is the case if $x-d \in X^{\cap}$. The proof divides into several cases distinguished by the distribution of $x-d$ and $x$ amongst the rows of $X$.

\underline{Case 1}: $x \in X^{\ominus}$ and $x-d \not\in X^{\cup}$. We have a bijection $\mathscr{X}(X,x) \to \mathscr{Y}(X,x)$ given by
\begin{equation*}
(H,\lambda) \mapsto (\mu,G) := \begin{cases}
(\lambda,H) &\text{if }H^{\ominus} \cap \{x,x-d\}=\emptyset\\
(\lambda^{\op},H\ominus \{\lambda^{\op},\mathcal{D}_{d,0}(\lambda^{\op})\}) &\text{if }H^{\ominus} \cap \{x,x-d\}=\{x\}
\end{cases}
\end{equation*}
such that $X\setminus H \setminus \lambda = X \setminus \mu \setminus G$. Now $(X\setminus \mu)^{\ominus} = (X^{\ominus}-\{x\}) \sqcup\{x-d\}$, so we have $s(X\setminus \mu) = s(X)$. Hence, by \cref{lem:fourier-comparison} the coefficients of $X\setminus H \setminus \lambda$ and $X \setminus \mu \setminus G$ in \cref{eq:lhs-asai,eq:rhs-asai} are the same.

\underline{Case 2}: $x \in X^{\cap}$ and $x-d \not\in X^{\cup}$. We have a fixed-point free involution ${}' : \mathscr{Y}(\Lambda,x) \to \mathscr{Y}(\Lambda,x)$ given by $(\mu,G)' = (\mu^{\op},G')$, where $G'$ is defined by
\begin{equation*}
G\ominus G' = \begin{cases}
\{\mu,\mathcal{D}_{d,0}(\mu^{\op})\} &\text{if }G^{\ominus} \cap \{x,x-d\} = \emptyset\\
\{\mu^{\op},\mathcal{D}_{d,0}(\mu^{\op})\} &\text{if }G^{\ominus} \cap \{x,x-d\} = \{x\}\\
\{\mu,\mathcal{D}_{d,0}(\mu)\} &\text{if }G^{\ominus} \cap \{x,x-d\} = \{x-d\}\\
\{\mu^{\op},\mathcal{D}_{d,0}(\mu)\} &\text{if }G^{\ominus} \cap \{x,x-d\} = \{x,x-d\}.
\end{cases}
\end{equation*}
This bijection satisfies
$$X\setminus \mu \setminus G = X\setminus \mu^{\op} \setminus G' \mbox{ and }\langle \{x,x-d\}, G^{\ominus}\rangle = \langle \{x,x-d\}, G'^{\ominus}\rangle$$ 
because $G^{\ominus}\ominus G'^{\ominus} = \{x,x-d\}$. Given $e \in \{0,1\}$ we let $\mathscr{Y}^e(\Lambda,x)$ be the set of pairs $\{(\mu,G),(\mu,G)'\}$ with $\langle \{x,x-d\}, G^{\ominus}\rangle = \bar e$, where bar indicates reduction (mod $2$).

We claim the coefficients of $V = X\setminus \mu \setminus G$ and $V' = X\setminus \mu^{\op} \setminus G'$ in \cref{eq:rhs-asai} differ by $(-1)^{\langle \{x,x-d\}, G^{\ominus}\rangle}$. If $U = X\setminus \mu$ and $U' = X \setminus \mu^{\op}$ then as in \cref{lem:fourier-comparison} we get that 
$$\langle U^{\sharp},V^{\sharp}\rangle + \langle U'^{\sharp},V'^{\sharp}\rangle = \langle U^{\sharp},G^{\ominus}\rangle + \langle U'^{\sharp},G'^{\ominus}\rangle.$$ 
As $G^{\ominus}\ominus G'^{\ominus} = \{x,x-d\}$, this term can be written as
\begin{align*}
&\quad\langle U^1,\{x,x-d\}\rangle+\langle U_{\spec}^1,\{x,x-d\}\rangle + \langle U^1\ominus U'^1, G^{\ominus}\rangle + \langle U_{\spec}^1\ominus U_{\spec}'^1, G^{\ominus}\rangle\\
&= 1 + \langle U^{\ominus}, \{0,\dots,x\}\ominus \{0,\dots,x-d\}\rangle + \langle \{x,x-d\}, G^{\ominus}\rangle+0\\
&= \langle X^{\ominus}, \{0,\dots,x\}\ominus \{0,\dots,x-d\}\rangle + \langle \{x,x-d\}, G^{\ominus}\rangle.
\end{align*}
Finally, as in \cref{lem:fourier-comparison}, 
$$l_{d,0}(\mu,X)+l_{d,0}(\mu^{\op},X) = \langle \{0,\dots,x\}\ominus \{0,\dots,x-d\},X^{\ominus}\rangle.$$

Therefore, given a pair $\{(\mu,G),(\mu^{\op},G')\} \in \mathscr{Y}^1(X,x)$, the corresponding terms $X\setminus \mu \setminus G$ and $X\setminus \mu^{\op} \setminus G'$ in \cref{eq:rhs-asai} cancel. We have a bijection $\mathscr{X}(X,x) \to \mathscr{Y}^0(\Lambda,x)$ given by $(H,\lambda)\mapsto \{(\lambda,H),(\lambda,H)'\}$ such that $X\setminus H \setminus \lambda = X\setminus \lambda \setminus H$. Now $s(X\setminus \lambda) = 2s(X)$ because $(X\setminus \lambda)^{\ominus} = X^{\ominus}\sqcup\{x,x-d\}$, but the coefficients of $X \setminus \lambda \setminus H$ and $X \setminus \lambda^{\op} \setminus H'$ combine to yield the coefficient of $X \setminus H \setminus \lambda$ by \cref{lem:fourier-comparison}.

\underline{Case 3}: $x \in X^{\ominus}$ and $x-d \in X^{\ominus}$. Identically to Case 2 we have a fixed-point free involution ${}' : \mathscr{X}(X,x) \to \mathscr{X}(X,x)$ denoted by $(H,\lambda)' = (H',\lambda^{\op})$. This bijection satisfies
$$X\setminus H' \setminus\lambda^{\op} = X\setminus H \setminus\lambda \mbox{ and }\langle \{x,x-d\}, H^{\ominus}\rangle = \langle \{x,x-d\}, H'^{\ominus}\rangle.$$ 
One can check that when the terms exist, the coefficients of $X\setminus H \setminus \lambda$ and $X\setminus H'\setminus \lambda^{\op}$ in \cref{eq:lhs-asai} differ by $(-1)^{1+\langle \{x,x-d\}, X^1\rangle}$. Given $e \in \{0,1\}$ we let $\mathscr{X}^e(X,x)$ be the set of pairs $\{(H,\lambda),(H,\lambda)'\}$ with $\langle \{x,x-d\}, H^{\ominus}\rangle = \bar e$.
%

If $x-d$ and $x$ occur in the same row of $X$ then we must have $\mathscr{X}^0(X,x) = \{\emptyset\}$ and $\mathscr{Y}(X,x) = \{\emptyset\}$. Moreover, each pair in $\mathscr{X}^1(X,x)$ gives rise to terms in \cref{eq:lhs-asai} that cancel. If $x-d$ and $x$ occur in opposite rows of $X$ then we must have $\mathscr{X}^1(X,x) = \{\emptyset\}$. Moreover, we have a bijection $\mathscr{Y}(X,x) \to \mathscr{X}^0(\Lambda,x)$ given by $(\mu,G) \mapsto \{(G,\mu),(G,\mu)'\}$ such that $X\setminus \mu \setminus G = X \setminus G \setminus \mu = X \setminus G' \setminus \mu^{\op}$. In this case $(X\setminus \mu)^{\ominus} = X^{\ominus} - \{x,x-d\}$ so $2s(X\setminus \mu) = s(X)$ but $X \setminus H \setminus\lambda$ and $X \setminus H' \setminus\lambda^{\op}$ have the same coefficient in \cref{eq:lhs-asai}. Thus, the coefficients of $X\setminus \mu \setminus G = X \setminus G \setminus \mu$ agree by \cref{lem:fourier-comparison}

\underline{Case 4}: $x \in X^{\cap}$ and $x-d \in X^{\ominus}$. As in Case 1, we have a bijection $\mathscr{X}(X,x) \to \mathscr{Y}(X,x)$ given by
\begin{equation*}
(H,\lambda) \mapsto (\mu,G) := \begin{cases}
(\lambda,H) &\text{if }H^{\ominus} \cap \{x,x-d\}=\emptyset\\
(\lambda^{\op},H\ominus \{\lambda,\mathcal{D}_{d,0}(\lambda)\}) &\text{if }H^{\ominus} \cap \{x,x-d\}=\{x-d\}
\end{cases}
\end{equation*}
such that $X\setminus H\setminus\lambda = X\setminus \mu\setminus G$. Clearly $(X\setminus \mu)^{\ominus} = (X^{\ominus}-\{x-d\})\sqcup\{x\}$ so $s(X\setminus \mu) = s(X)$. Again, by \cref{lem:fourier-comparison} the coefficients of $X\setminus H \setminus \lambda$ and $X \setminus \mu \setminus G$ in \cref{eq:lhs-asai,eq:rhs-asai} are the same.
%
%
\end{proof}

We now consider the proof of (ii) of \cref{thm:asai-fourier-commutes}. The argument is exactly the same as (i), proceeding through the same cases. The bijection in Case (ii) is defined identically simply replacing $\cD_{d,0}$ with $\cD_{d,1}$. Instead of providing a direct analogue of \cref{lem:fourier-comparison} we instead check directly in each case that the signs of the corresponding coefficients agree. As an example, we treat the analogue of Case 1, leaving the remaining cases to the reader.

\begin{proof}[Proof of \cref{thm:asai-fourier-commutes}(ii)]
Assume $x \in X^{\ominus}$ and $x-d \not\in X^{\cup}$. We have a bijection $\mathscr{X}(X,x) \to \mathscr{Y}(X,x)$ given by
\begin{equation*}
(H,\lambda) \mapsto (\mu,G) := \begin{cases}
(\lambda,H\ominus\{\mathcal{D}_{d,1}(\lambda)\}) &\text{if }H^{\ominus} \cap \{x,x-d\}=\emptyset\\
(\lambda^{\op},H\ominus \{\lambda^{\op}\}) &\text{if }H^{\ominus} \cap \{x,x-d\}=\{x\}
\end{cases}
\end{equation*}
such that $X\setminus H\setminus \lambda = X\setminus \mu\setminus G$. As in the proof of (i) of \cref{thm:asai-fourier-commutes}, we need only check that the sign of the coefficient of $Y \setminus\lambda$ in \cref{eq:lhs-asai} and $(-1)^{1+\Def(V)}$ times the sign of the coefficient of $V = U\setminus G$ in \cref{eq:rhs-asai} agree, where $Y = X\setminus H$ and $U = X\setminus \mu$. We check this directly.

As $\mu$ is a $(d,1)$-hook, we have $\abs{U^1} = \abs{X^1}\pm 1$, so arguing as in the proof of \cref{lem:fourier-comparison}, we see that $1+\Def(V)+\langle X^{\sharp},Y^{\sharp}\rangle + \langle U^{\sharp},V^{\sharp}\rangle$ is
\begin{equation*}
\Def(V)+\langle U^1, H^{\ominus}\ominus G^{\ominus}\rangle + \langle U_{\spec}^1, H^{\ominus}\ominus G^{\ominus}\rangle + \langle X^1\ominus U^1, H^{\ominus}\rangle + \langle X_{\spec}^1\ominus U_{\spec}^1, H^{\ominus}\rangle.
\end{equation*}
Moreover, this time $l_{d,0}(\lambda,Y) + l_{d,1}(\mu,X)$ is
\begin{equation*}
\langle \{0,\dots,x\}\ominus\{0,\dots,x-d\},X^{\delta(\lambda)}\ominus X^{\delta(\mu)}\ominus H^{\ominus}\rangle + \langle \{0,\dots,x-d\},X^{\ominus}\ominus\{x\}\rangle.
\end{equation*}
We consider the two cases of the bijection above separately.

Suppose first that $H^{\ominus} \cap \{x,x-d\} = \emptyset$. Clearly $\langle X^1\ominus U^1, H^{\ominus}\rangle = 0$, and as $H^{\ominus} \subseteq U^{\ominus}$, we have $\langle U_{\spec}^1 \ominus X_{\spec}^1, H^{\ominus}\rangle = \langle \{0,\dots,x-d\}\ominus\{0,\dots,x\}, H^{\ominus}\rangle$ as in \cref{lem:fourier-comparison}. Now 
$$\langle U_{\spec}^1,H^{\ominus}\ominus G^{\ominus}\rangle = \langle U_{\spec}^1,\{x-d\}\rangle = \abs{U^{\ominus}} + \langle U^{\ominus},\{0,\dots,x-d\}\rangle = \Def(V) + \langle X^{\ominus}\ominus\{x,x-d\},\{0,\dots,x-d\}\rangle$$ 
because $x \not\in U^{\cup}$. Summing the above terms it suffices to show that 
$$(-1)^{\delta(\lambda)} = (-1)^{1+\langle U^1, H^{\ominus} \ominus G^{\ominus}\rangle} = (-1)^{1+\langle U^1, \{x-d\}\rangle},$$ 
and this is straightforward.

Finally we assume that $H^{\ominus} \cap \{x,x-d\} = \{x\}$. This time $H^{\ominus} \ominus\{x\} \subseteq X^{\ominus} \cap U^{\ominus}$ so 
$$\langle U_{\spec}^1 \ominus X_{\spec}^1, H^{\ominus}\ominus\{x\}\rangle = \langle \{0,\dots,x-d\}\ominus\{0,\dots,x\}, H^{\ominus}\ominus\{x\}\rangle.$$ 
Moreover, $\langle X_{\spec}^1\ominus U_{\spec}^1,\{x\}\rangle = \langle X_{\spec}^1,\{x\}\rangle = \overline{\Def(V)} + \langle X^{\ominus},\{0,\dots,x\}\rangle$ and $\langle U^1,H^{\ominus}\ominus G^{\ominus}\rangle = \langle U_{\spec}^1,H^{\ominus}\ominus G^{\ominus}\rangle = 0$ because $H^{\ominus}\ominus G^{\ominus} = \{x\}$ and $x \not\in U^{\cup}$. As before it suffices to show that 
$$(-1)^{\delta(\lambda)} = (-1)^{1+\langle X^1\ominus U^1, H^{\ominus}\rangle} = (-1)^{1+\langle X^1\ominus U^1, \{x\}\rangle},$$ 
and again this is straightforward.
\end{proof}

\subsection{More symbols}
We have a linear map $\symb{-} : \CC[\tilde{\cP}] \to \CC[\cS]$ defined such that if $X \in \tilde{\cP}$ is nondegenerate then $\symb{X} \in \cS$ and if $X$ is degenerate then $\symb{X} = \symb{X}_+ + \symb{X}_-$. This clearly factors through a map $\CC[\tilde{\cS}] \to \CC[\cS]$. The image of $\symb{-}$ has a natural complement in $\CC[\cS]$, namely $\langle \symb{X}_+ - \symb{X}_- \mid X \in \tilde{\cP}$ is degenerate$\rangle_{\CC}$.

We wish to understand to what extent the endomorphisms $\tilde{\cR}$ and $\tilde{\cH}_{d,i}^j$ of $\CC[\tilde{\cS}]$ factor through $\symb{-}$. As the term arises frequently we let
\begin{equation*}
\dd(X) = (\Def(X) - \Def(X_{\spec}))/2 \in \NN_0
\end{equation*}
for any $X \in \tilde{\cP}$. We then define a $\CC$-linear map $\varepsilon : \CC[\tilde{\cS}] \to \CC[\tilde{\cS}]$ by setting $\varepsilon(\osymb{X}) = (-1)^{\dd(X)}\osymb{X}$. The following easy observations are stated in \cite[\S2]{W}.

\begin{lem}\label{lem:fourier-op}
For any $(d,i) \in \ZZ^{(2)}$ and $j \in \{0,1\}$ we have the following equalities of linear endomorphisms of $\CC[\tilde{\cS}]$:
\begin{enumerate}
	\item[\rm(i)] $\tilde{\cR}\circ (-)^{\op} = \varepsilon\circ\tilde{\cR}$,
	\item[\rm(ii)] $(-)^{\op}\circ\tilde{\cR} = \tilde{\cR}\circ \varepsilon$,
	\item[\rm(iii)] $(-)^{\op}\circ\tilde{\cH}_{d,i}^j = (-1)^j\tilde{\cH}_{d,i}^j \circ (-)^{\op}$,
	\item[\rm(iv)] $\varepsilon \circ (-)^{\op} = \Theta\circ (-)^{\op} \circ \varepsilon$.
\end{enumerate}
\end{lem}

\begin{proof}
(i). If $X \in \tilde{\cP}$ then $\Sim(X) = \Sim(X^{\op}) = \Sim(X)^{\op}$. As $(X^{\op})^{\sharp} = X^{\ominus}\ominus X^{\sharp}$, we see that the coefficient of $Y \in \Sim(\Lambda)$ in $\tilde{\cR}(X^{\op})$ is $(-1)^{\langle X^{\ominus},Y^{\sharp}\rangle} = (-1)^{\langle X^{\ominus},Y^1\rangle+\langle X^{\ominus},X_{\spec}^1\rangle}$ times the corresponding coefficient in $\tilde{\cR}(X)$. Now,
\begin{align*}
2\dd(X) &= (\abs{X^{\ominus}} - 2\abs{Y^1 \cap X^{\ominus}}) - (\abs{X^{\ominus}} - 2\abs{X_{\spec}^1 \cap X^{\ominus}})\\
&= 2(\abs{X_{\spec}^1 \cap X^{\ominus}} - \abs{Y^1 \cap X^{\ominus}}).
\end{align*}

(ii). This is similar to (i) using that $\langle X^{\sharp},Y^{\sharp}\ominus(Y^{\op})^{\sharp}\rangle = \langle X^{\sharp},X^{\ominus}\rangle$.

(iii). It is straightforward to check that for any $X \in \tilde{\cP}$, we have $\mathscr{H}_{d,i}(X^{\op}) = \mathscr{H}_{d,i}(X)^{\op}$ and $l_{d,i}(\lambda,X) = l_{d,i}(\lambda^{\op},X^{\op})$, which gives the statement.

(iv). As $\Def(X^{\op}) = -\Def(X)$, we have $\dd(X^{\op}) = \dd(X)-\Def(X)$.
\end{proof}

We define for each integer $e \in \ZZ$ the set
\begin{equation*}
\tilde{\cP}^{\equiv e} = \{X \in \tilde{\cP} \mid \dd(X) \equiv e\pmod{2}\}.
\end{equation*}
This gives a partition $\tilde{\cP} = \tilde{\cP}^{\equiv 0} \sqcup \tilde{\cP}^{\equiv 1}$ and partitions $\tilde{\cP}^{\odd} = \tilde{\cP}^{\odd,0}\sqcup \tilde{\cP}^{\odd,1}$ and $\tilde{\cP}^{\ev} = \tilde{\cP}^{\ev,0}\sqcup \tilde{\cP}^{\ev,1}$. Note that $(-)^{\op}$ swaps $\tilde{\cP}^{\odd,0}$ and $\tilde{\cP}^{\odd,1}$ but stabilises $\tilde{\cP}^{\ev,0}$ and $\tilde{\cP}^{\ev,1}$. Thus we get a partition $\cS^{\ev} = \cS^{\ev,0}\sqcup \cS^{\ev,1}$, where $\cS^{\ev,0}$ contains all $\symb{X}_{\pm}$ with $X \in \tilde{\cP}^{\ev,0}$ degenerate.

If $X \in \tilde{\cP}$ then we let $\Sim_e(X) = \Sim(X) \cap \tilde{\cP}^{\equiv e}$. Under the map ${}^{\sharp} : \tilde{\cP} \to \Pow(\NN_0)$, we have $X \in \tilde{\cP}^{\equiv e}$ if and only if $\abs{X^{\sharp}} \equiv e \pmod{2}$. Now we define $\tilde{\cR}_e : \CC[\tilde{\cP}] \to \CC[\tilde{\cP}]$ by setting
\begin{equation*}
\tilde{\cR}_e(X) = \frac{1}{s(X)}\sum_{Y \in \Sim_e(X)} (-1)^{\langle X^{\sharp},Y^{\sharp}\rangle}Y
\end{equation*}
so that $\tilde{\cR} = \tilde{\cR}_0 + \tilde{\cR}_1$.

\begin{lem}\label{lem:Re-op}
For any $X \in \tilde{\cP}$ and $e \in \ZZ$ we have:
\begin{enumerate}
	\item[\rm(i)] $\tilde{\cR}_e(X^{\op}) = (-1)^e\tilde{\cR}_e(X)$,
	\item[\rm(ii)] $\tilde{\cR}_e(X)^{\op} = (-1)^{\dd(X)}\tilde{\cR}_{e+\Def(X)}(X)$
\end{enumerate}
In particular, for any $X \in \tilde{\cP}^{\equiv 1}$ we have $\symb{\tilde{\cR}_e(X)} = 0$.
\end{lem}

\begin{proof}
This follows by projecting (i) and (ii) of \cref{lem:fourier-op} onto each summand of the decomposition $\CC[\tilde{\cP}] = \CC[\tilde{\cP}^{\equiv 0}] \oplus \CC[\tilde{\cP}^{\equiv 1}]$. Here we use that $\Sim_e(X)^{\op} = \Sim_{e+\Def(X)}(X)$, which follows from the above remarks. For the final statement note that (ii) shows that if $X \in \tilde{\cP}^{\equiv 1}$ then $\symb{\tilde{\cR}_e(X)} = -\symb{\tilde{\cR}_{e+\Def(X)}(X)} = -\symb{\tilde{\cR}_e(X)}$.
\end{proof}

By (i) of \cref{lem:Re-op} we have $\cR_0$ factors through $\symb{-}$ to give an endomorphism of its image. We extend this to an endomorphism of $\CC[\cS]$ by letting it fix pointwise the complement defined above (in other words, $\cR_0(\symb{X}_{\pm}) = \symb{X}_{\pm}$ for all degenerate $X \in \tilde{\cP}$). We also denote this by $\cR_0$. We consider the subspaces $\CC[\cS^{\odd}]$ and $\CC[\cS^{\ev}]$ separately. First let $\mathscr{A}^{\odd} = \mathscr{U}^{\odd} = \CC[\cS^{\odd}]$.

Note that $\cR_0(\symb{X})$ is simply the Fourier transform on the abelian group $\Pow_0(X^{\ominus})$. As such $\cR_0^2$ is the identity on $\CC[\cS^{\odd}]$. To have a compatible notation we consider $\cR_0$ as a map $\mathscr{U}^{\odd} \to \mathscr{A}^{\odd}$ and denote by $\cQ_0 : \mathscr{A}^{\odd} \to \mathscr{U}^{\odd}$ its inverse.

Assume $(d,i) \in \ZZ^{(2)}$ with $d \neq 0$. It follows from \cref{lem:fourier-op} that the restriction of the map $\tilde{\cH}_{d,i}^0$ to $\CC[\tilde{\cS}^{\odd}]$ factors through a well-defined endomorphism of $\mathscr{U}^{\odd}$ and $\mathscr{A}^{\odd}$ which we denote by $\cH_{d,i}^0$. Similarly $\tilde{\cH}_{d,i}^1 \circ \varepsilon$ factors through an endomorphism which we denote by $\cH_{d,i}^1$. The following is now simply a consequence of \cref{thm:asai-fourier-commutes}.

\begin{prop}
For any $(d,i) \in \mathbb{Z}^{(2)}$, with $d \neq 0$, we have commutative diagrams
\begin{center}
\begin{tikzcd}
\mathscr{A}^{\odd}\arrow[d,"\cH_{d,i}^0"']\arrow[r,"\cQ_0"] & \arrow[d,"\cH_{d,0}^i"] \mathscr{U}^{\odd}\\
\mathscr{A}^{\odd}\arrow[r,"\cQ_0"'] & \mathscr{U}^{\odd}
\end{tikzcd}
\qquad
\begin{tikzcd}
\mathscr{U}^{\odd}\arrow[d,"\cH_{d,i}^0"']\arrow[r,"\cR_0"] & \arrow[d,"\cH_{d,0}^i"] \mathscr{A}^{\odd}\\
\mathscr{U}^{\odd}\arrow[r,"\cR_0"'] & \mathscr{A}^{\odd}
\end{tikzcd}
\end{center}
\end{prop}

If $X \in \tilde{\cP}$ has even defect then $\Sim_0(X) = \Sim_0(X)^{\op}$. As remarked in the proof of \cref{lem:fourier-op} we have $(Y^{\op})^{\sharp} = Y^{\sharp}\ominus X^{\ominus}$ so $\langle X^{\sharp},(Y^{\op})^{\sharp}\rangle = \langle X^{\sharp},Y^{\sharp}\rangle$ for any $Y \in \Sim_0(X)$. Thus if $X \in \tilde{\cP}^{\ev}$ is nondegenerate then
\begin{equation*}
\cR_0(\symb{X}) = \frac{2}{s(X)}\sum_{Y \in \overline{\Sim}_0(X)} (-1)^{\langle X^{\sharp},Y^{\sharp}\rangle}\symb{Y}
\end{equation*}
where $\overline{\Sim}_e(X) = \{\{Y,Y^{\op}\} \mid Y \in \Sim_e(X)\}$. This is the Fourier transform on the abelian group $\overline{\Pow}_0(X^{\ominus}) = \Pow_0(X^{\ominus})/\{\emptyset,X^{\ominus}\}$.

If $e \in 2\ZZ$ is an even integer then we let $\mathscr{A}^{\ev,e} = \mathscr{U}^{\ev,e} = \CC[\cS^{\ev,0}] \subseteq \CC[\cS^{\ev}]$. The map $\cR_0$ restricts to an involution on the subspace $\CC[\cS^{\ev,0}]$. As above we consider $\cR_0$ as a map $\mathscr{A}^{\ev,e} \to \mathscr{U}^{\ev,e}$ with inverse $\cQ_0 : \mathscr{U}^{\ev,e} \to \mathscr{A}^{\ev,e}$.

Let $\tilde{\cP}^{\ev,\mathrm{nd}} \subseteq \tilde{\cP}^{\ev,0}$ be the subset of nondegenerate arrays. For any odd integer $e \in 2\ZZ+1$ we let $\mathscr{U}^{\ev,e} = \CC[\cS^{\ev,1}]$ and define the quotient space
\begin{equation*}
\mathscr{A}^{\ev,e} = \CC[\tilde{\cP}^{\ev,\mathrm{nd}}]/\langle \osymb{X} + \osymb{X^{\op}} \mid X \in \tilde{\cP}^{\ev,\mathrm{nd}}\rangle_{\CC}.
\end{equation*}
By \cref{lem:Re-op} the map $\tilde{\cR}_e$ factors through a map $\mathscr{A}^{\ev,e} \to \mathscr{U}^{\ev,e}$ which we denote by $\cR_e$. We define a right inverse $\cQ_e : \mathscr{U}^{\ev,e} \to \mathscr{A}^{\ev,e}$ of this map by setting
\begin{equation*}
\cQ_e(\symb{X}) = \frac{1}{s(X)}\sum_{Y \in \Sim_0(X)} (-1)^{\langle X^{\sharp},Y^{\sharp}\rangle}\negsymb{Y} = \frac{2}{s(X)}\sum_{\overline{Y} \in \overline{\Sim}_0(X)} (-1)^{\langle X^{\sharp},Y^{\sharp}\rangle}\negsymb{Y}
\end{equation*}
where $\negsymb{-} : \CC[\tilde{\cP}^{\ev,\mathrm{nd}}] \to \mathscr{A}^{\ev,e}$ is the natural quotient map.

It follows easily from \cref{lem:fourier-op} that the endomorphism $\tilde{\cH}_{d,0}^i$ of $\CC[\tilde{\cP}^{\ev,\mathrm{nd}}]$ factors through a well defined map $\cH_{d,0}^i : \mathscr{A}^{\ev,e} \to \mathscr{A}^{\ev,e+i}$ for each $e \in \ZZ$. Similarly we have $\cH_{d,i}^0$ factors through a map $\mathscr{U}^{\ev,e} \to \mathscr{U}^{\ev,e+i}$.

\begin{prop}
For any $(d,i) \in \ZZ^{(2)}$, with $d \neq 0$, and any $e \in \ZZ$ we have commutative diagrams
\begin{center}
\begin{tikzcd}
\mathscr{U}^{\ev,e}\arrow[d,"(-1)^i\cH_{d,i}^0"']\arrow[r,"\cQ_e"] & \arrow[d,"\cH_{d,0}^i"] \mathscr{A}^{\ev,e}\\
\mathscr{U}^{\ev,e+i}\arrow[r,"\cQ_{e+i}"'] & \mathscr{A}^{\ev,e+i}
\end{tikzcd}
\qquad
\begin{tikzcd}
\mathscr{A}^{\ev,e}\arrow[d,"\cH_{d,0}^i"']\arrow[r,"\cR_e"] & \arrow[d,"(-1)^i\cH_{d,i}^0"] \mathscr{U}^{\ev,e}\\
\mathscr{A}^{\ev,e+i}\arrow[r,"\cR_{e+i}"'] & \mathscr{U}^{\ev,e+i}
\end{tikzcd}
\end{center}
\end{prop}

\section{Hyperoctahedral groups}\label{sec:hyperoctahedral-groups}
Assume $(\cI,\prec)$ is a finite totally ordered set of cardinality $\abs{\cI} = 2n$. Denote by ${}^{\dag} : \cI \to \cI$ the unique order reversing bijection on $\cI$. We say $a \in \cI$ is positive or negative if $a \succ a^{\dag}$ or $a \prec a^{\dag}$ respectively. This gives a decomposition $\cI = \cI^+ \sqcup \cI^-$ into subsets of cardinality $n$. If $\cO \subseteq \cI$ is a subset then $(\cO,\prec)$ is also a totally ordered set.

\begin{exa}\label{exmp:totally-ordered-set}
We could take $\cI = \{-n\prec\cdots \prec -1\prec 1\prec \cdots \prec n\}$ then for any $a \in \cI$ we have $a^{\dag} = -a$ so $\cI^+ = \{1,\dots,n\}$ and $\cI^- = \{-1,\dots,-n\}$.
\end{exa}

If $\Sym_{\cI}$ is the symmetric group on $\cI$ then we define $W_{\cI} = \bC_{\Sym_{\cI}}(\sigma)$ to be the centraliser of the involution $\sigma = \prod_{a \in \cI^+} (a,a^{\dag})$. Let $\delta_{\cI} : \Sym_{\cI} \to \ZZ/2\ZZ$ be the unique non-trivial homomorphism. Given $e \in \{0,1\}$ we let $W_{\cI}^e = \{w \in W_{\cI} \mid \delta_{\cI}(w) = e\}$ so that we have a decomposition $W_{\cI} = W_{\cI}^0 \sqcup W_{\cI}^1$ into the cosets of $W_{\cI}^0 \lhd W_{\cI}$. Note we have a semidirect product decomposition $W_{\cI} = N_{\cI} \rtimes H_{\cI}$ where $N_{\cI} = \langle (a,a^{\dag}) \mid a \in \cI^+\rangle$ and $H_{\cI} = \{w \in W_{\cI} \mid {}^w\cI^+ = \cI^+\} \cong \Sym_{\cI^+}$.

For any $\sigma$-stable subset $\cO \subseteq \cI$, equivalently $\cO = \cO^{\dag}$, we have a natural injective homomorphism $W_{\cO} \to W_{\cI}$ whose image is the pointwise stabiliser of $\cI\setminus\cO$. We identify $W_{\cO}$ with its image in $W_{\cI}$.

We say $w \in W_{\cI}$ is an \emph{$\cI$-cycle} if the subgroup $\langle w,\sigma\rangle \leqslant W_{\cI}$ acts transitively on $\cI$. Thus $w = nh$, with $n \in N_{\cI}$ and $h \in H_{\cI}$ acting on $\cI^+$ as cycle of length $n$. The following is an elementary calculation.

\begin{lem}\label{lem:cent-I-cycle}
If $w \in W_{\cI}$ is an $\cI$-cycle then $\bC_{W_{\cI}}(w) = \langle w\rangle$ if $\delta_{\cI}(w)=1$ and $\bC_{W_{\cI}}(w) = \langle w\rangle\rtimes\langle \sigma\rangle$ if $\delta_{\cI}(w)=0$. In either case $\abs{\bC_{W_{\cI}}(w)} = \abs{\cI} = 2n$.
\end{lem}

Now suppose $w \in W_{\cI}$ and $\cI/\langle w,\sigma\rangle = \{\cO_1,\dots,\cO_k\}$. Then we can write $w = w_1\cdots w_k$ as a pairwise commuting product with $w_i \in W_{\cO_i}$ an $\cO_i$-cycle. Such a decomposition, which we call a \emph{cycle decomposition}, is unique up to reordering. We call $k\geqslant 1$ the \emph{cycle length}.

There is a unique ordering of the orbits such that $(-1)^{\delta_{\cO_1}(w_1)}\abs{\cO_1^+} \geqslant \cdots \geqslant (-1)^{\delta_{\cO_k}(w_k)}\abs{\cO_k^+}$. With this ordering we call $((-1)^{\delta_{\cO_1}(w_1)}\abs{\cO_1^+}, \dots ,(-1)^{\delta_{\cO_k}(w_k)}\abs{\cO_k^+})$ the signed cycle type of the element. It determines the conjugacy class $\Cl_{W_{\cI}}(w)$ uniquely. If these inequalities are all strict then we say $w$ has \emph{pairwise distinct cycles}.

\begin{lem}\label{lem:cents-pairwise-distinct}
Let $w \in W_{\cI}$ be as above and let $\cI = \cI_1\sqcup\cdots\sqcup\cI_m$ be a decomposition into $\sigma$-stable sets such that $P = W_{\cI_1}\cdots W_{\cI_m} \leqslant W_{\cI}$ contains $w$. If $w$ has pairwise distinct cycles then the following hold:
\begin{enumerate}[label={\normalfont(\roman*)}]
	\item $\bC_{W_{\cI}}(w) = \bC_{W_{\cO_1}}(w_1)\cdots \bC_{W_{\cO_k}}(w_k)$,
	\item $\abs{\bC_{W_{\cI}}(w)} \leqslant 2^k\cdot n^k$,
	\item $\bC_P(w) = \bC_W(w)$.
\end{enumerate}
\end{lem}

\begin{proof}
Part (i) is given by the uniqueness of the cycle decomposition. Part (ii) follows from (i) and \cref{lem:cent-I-cycle}. Part (iii) follows from (i) because $\langle w,\sigma\rangle$ stabilises each $\cI_j$, so they must be a union of the $\cO_i$.
\end{proof}

Recall from \cref{sec:combinatorics} that we have define the $\beta$-sets $\cB_n$. After \cite[\S6.4.1]{GP} we have a bijection $\cB_n \to \Irr(\Sym_{\cI^+})$ which we denote by $\bset{A} \mapsto \chi_{\bset{A}}$. Under the natural isomorphism $\Sym_{\cI^+} \cong H_{\cI^+}$ we get a bijection $\cB_n \to \Irr(H_{\cI})$.

If $\delta \in \{0,1\}$ then this yields a bijection $\tilde{\cS}_n^{\delta} \to \Irr(W_{\cI})$, denoted by $\osymb{X} \mapsto \rho_{\osymb{X}}$, defined as follows. First note that for any $\osymb{X} \in \tilde{\cS}_n^{\delta}$ we have $\rk(\bset{X^0}) + \rk(\bset{X^1}) = n$ by \cref{eq:rank-definition}. Now choose a $\sigma$-stable partition $\cI = \cI_0\sqcup\cI_1$ such that $\abs{\cI_j} = 2\cdot\rk(X^j)$ with $j \in \{0,1\}$ (note these subsets may be empty). We then have
\begin{equation*}
\rho_{\osymb{X}} = \Ind_{W_{\cI_0}W_{\cI_1}}^{W_{\cI}}(\tilde\chi_{\bset{X^0}} \boxtimes (\varepsilon_{\cI_1}\tilde{\chi}_{\bset{X^1}}))
\end{equation*}
where $\tilde{\chi}_{\bset{X^j}}$ is the inflation of $\chi_{\bset{X^j}} \in \Irr(H_{\cI_j})$ under the map $W_{\cI_j} \to H_{\cI_j}$. These characters satisfy the following MN-rule (or Murnaghan--Nakayama rule).

\begin{prop}\label{prop:MN-rule-irr-char}
Let $\cO \in \cI/\langle w,N_{\cI}\rangle$ be an orbit for some element $w \in W_{\cI}$. Then we have a unique decomposition $w = w_1w_2 = w_2w_1$ with $w_1 \in W_{\cO}$ and $w_2 \in W_{\mathcal{I\backslash O}}$. If $(d,j) = (\abs{\cO^+},\delta_{\cO}(w_1))$ then for any $\osymb{X} \in \tilde{\cS}_n^{\delta}$, with $\delta \in \{0,1\}$, we have
\begin{equation*}
\rho_{\osymb{X}}(w) = \sum_{\lambda \in \mathscr{H}_{d,0}(X)} (-1)^{j\delta(\lambda) + l_{d,0}(\lambda,X)}\rho_{\osymb{X\setminus \lambda}}(w_2),
\end{equation*}
\end{prop}

\begin{proof}
We refer the reader to \cite[Thm.~10.3.1]{GP}. For the correspondence between hooks of partitions and hooks of $\beta$-sets see \cite[\S I.1]{O}.
\end{proof}

In \cite[Theorem~7.2]{LaSh2}, a bound is given for the character values of the symmetric group at a given element in terms of its cycle length. The argument in \cite{LaSh2} is a consequence of the MN-rule together with analogues of the following easy observations.

\begin{lem}\label{lem:hook-inclusion}
Let $X \in \tilde{\cP}$ be an array with an $(e,j)$-hook $\lambda \in \Hk_{e,j}(X)$ for some $(e,j) \in \NN_0^{(2)}$. Then for any $(d,i) \in \NN_0^{(2)}$ we have
\begin{equation*}
\Hk_{d,i}(X) \subseteq \Hk_{d,i}(X\setminus_{e,j} \lambda) \cup \{\lambda,\cD_{e-d,i+j}(\lambda)\}.
\end{equation*}
\end{lem}

\begin{proof}
Let $\mu \in \Hk_{d,i}(X)$ so $\cD_{d,i}(\mu) \in \NN_0^{(2)}$. If $\mu \neq\lambda$ then $\mu \in Y := X\setminus_{e,j} \lambda$, and if $\mu$ is not a $(d,i)$-hook of $Y$ then clearly $\cD_{d,i}(\mu) \in Y-X = \{\cD_{e,j}(\lambda)\}$.
\end{proof}

\begin{lem}\label{lem:n-hooks}
If $(n,i) \in \NN_0^{(2)}$ then $X \in \tilde{\cP}_n$ has at most one $(n,i)$-hook.
\end{lem}

\begin{proof}
Suppose $\lambda \in \mathscr{H}_{n,i}(X)$ is such a hook. Then $X\setminus_{n,i}\lambda$ is of rank $0$ and so has no hooks. Hence, by the previous lemma the only possible $(n,i)$-hooks are $\lambda$ and $\lambda^{\op}$. But it is easily seen that if $\lambda^{\op}$ were a hook of $X$ then it would also be one of $X\setminus_{n,i}\lambda$, which is impossible.
\end{proof}

\begin{thm}\label{thm:char-bound-hyperoctahedral}
Fix an integer $1 \leqslant k \leqslant n$. Then for each element $w \in W_{\cI}$ of cycle length $k$ and each irreducible character $\chi \in \Irr(W_{\cI})$ we have
\begin{equation*}
\abs{\chi(w)} \leqslant 2^{k-1}\cdot k!.
\end{equation*}
Moreover, if $w \in W_{\cI}^0$ then for all $\chi \in \Irr(W_{\cI}^0)$ we have $\abs{\chi(w)} \leqslant (2^k+1)\cdot 2^{k-1}\cdot k! \leqslant 2^{2k} \cdot k!$.
\end{thm}

\begin{proof}
Let $\chi = \rho_{\osymb{X}}$ with $\osymb{X} \in \tilde{\cS}_n^1$. We argue by induction on $k$. Suppose $k = 1$. If $\chi(w) \neq 0$ then by \cref{prop:MN-rule-irr-char}, $[X]$ has an $(n,0)$-hook, but by \cref{lem:n-hooks}, there is at most one such hook, so $\abs{\chi(w)} \leqslant 1$. So assume $k > 1$. We have $\cI/\langle w,\sigma\rangle = \{\cO_1,\dots,\cO_k\}$ and we set $d_i = \abs{\cO_i}$ for any $1 \leqslant i \leqslant k$. We also let $w = w_1w_2$ with $w_1 \in W_{\cO_1}$ and $w_2 \in W_{\cI\setminus \cO_1}$.

Clearly we may assume that $\chi(w) \neq 0$. By \cref{prop:MN-rule-irr-char} and the induction hypothesis we see that
\begin{equation*}
\abs{\chi(w)} \leqslant \sum_{\lambda \in \Hk_{d_1,0}(X)} \abs{\rho_{\osymb{X\setminus \lambda}}(w_2)} \leqslant \abs{\Hk_{d_1,0}(X)} \cdot 2^{k-2}\cdot (k-1)!
\end{equation*}
So it suffices to show that $\abs{\Hk_{d_1,0}(X)} \leqslant 2k$.

Repeatedly applying \cref{prop:MN-rule-irr-char}, we see that there exist arrays $X = X_0,X_1,\dots,X_k \in \tilde{\cP}$ such that for $1 \leqslant i \leqslant k$ we have $X_i = X_{i-1} \setminus_{d_i,0} \lambda_i$ for some hook $\lambda_i \in \Hk_{d_i,0}(X_{i-1})$. As 
$$\rk(X_i) = \rk(X_0) - (d_1+\cdots+d_i),$$ 
we have $\rk(X_k) = 0$, so $\mathscr{H}_{d_1,0}(X_k) = \emptyset$. By \cref{lem:hook-inclusion} we have $\abs{\mathscr{H}_{d_1,0}(X_i)} \leqslant \abs{\mathscr{H}_{d_1,0}(X_{i+1})}+2$, which yields the desired bound.

Now assume $\chi \in \Irr(W_{\cI}^0)$. If $\chi$ extends to $W_{\cI}$ then we are done so assume this is not the case. Then $\chi = \Res_{W_{\cI}^0}^{W_{\cI}}(\rho_{\osymb{X}})$ for some degenerate symbol $\osymb{X} \in \tilde{\cS}_n^0$ and is the sum $\chi_+ + \chi_-$ of two distinct irreducible characters. Clearly $\chi_{\pm}(w) = \frac{1}{2}(\chi(w)+\Delta(w))$ where $\Delta = \chi_+ - \chi_-$ is the difference character.

If $\Delta(w) = 0$ then we are done, so we may assume that $\Delta(w) \neq 0$. A result of Stembridge \cite[Theorem~7.5]{S} shows that, in this case, $\abs{\Delta(w)} = 2^k\abs{\chi_{\bset{A}}(x)}$ for some character $\chi_{\bset{A}} \in \Irr(\Sym_{\cI^+})$ and some element $x \in \Sym_{\cI^+}$ which is a product of $k$ disjoint cycles. Hence, by \cite[Theorem~7.2]{LaSh2}, we have $\abs{\Delta(w)} \leqslant 2^{2k-1}\cdot k!$, which easily gives the bound.
\end{proof}

For an ordered symbol $\osymb{X} \in \tilde{\cS}_n$ of rank $n$ we define a corresponding class function
\begin{equation*}
\phi_{\osymb{X}} = \frac{1}{s(X)}\sum_{\substack{Y \in \Sim(X)\\ \Def(Y) = \Def(Y_{\spec})}} (-1)^{\langle X^{\sharp},Y^{\sharp}\rangle}\rho_{\osymb{Y}} \in \Class(W_{\cI}).
\end{equation*}
Note that $\{Y \in \Sim(X) \mid \Def(Y) = \Def(Y_{\spec})\} \subseteq \Sim_0(X)$ so this is essentially the projection of the Fourier transform $\tilde{\cR}_0(X)$ onto the subspace $\QQ[\tilde{\cS}_n^{\delta}]$, where $\delta = \Def(X_{\spec})$. Somewhat remarkably these functions also satisfy a version of the MN-rule, which is the main point of \cref{thm:asai-fourier-commutes}.

\begin{thm}\label{prop:phi-MN-rule}
Let $\cO \in \cI/\langle w,N_{\cI}\rangle$ be an orbit for some element $w \in W_{\cI}$. Then we have a unique decomposition $w = w_1w_2 = w_2w_1$ with $w_1 \in W_{\cO}$ and $w_2 \in W_{\mathcal{I\backslash O}}$. If $(d,j) = (\abs{\cO^+},\delta_{\cO}(w_1))$ then for any $\osymb{X} \in \tilde{\cS}_n$, we have
\begin{equation*}
\phi_{\osymb{X}}(w) = (-1)^{j(1+\Def(X))}\sum_{\lambda \in \mathscr{H}_{d,j}(X)} (-1)^{l_{d,j}(\lambda,X)}\phi_{\osymb{X\setminus \lambda}}(w_2).
\end{equation*}
\end{thm}

\begin{proof}
By \cref{prop:MN-rule-irr-char} we have
\begin{equation*}
\phi_{\osymb{X}}(w) = \frac{1}{s(X)}\sum_{\substack{Y \in \Sim(X)\\\Def(Y) = \Def(X_{\spec})}}\sum_{\lambda \in \mathscr{H}_{d,0}(Y)} (-1)^{j\delta(\lambda)+\langle X^{\sharp},Y^{\sharp}\rangle+l_{d,0}(\lambda,X)}\rho_{\osymb{Y\setminus \lambda}}(w_2).
\end{equation*}
Let $\delta = \Def(X_{\spec}) \in \{0,1\}$. Under the isomorphism $\QQ[\Irr(W_{\cI\setminus\cO})] \overset{\sim}{\longrightarrow} \QQ[\tilde{\cS}_{n-d}^{\delta}]$ defined above the right hand side is identified with an expression in $\QQ[\tilde{\cS}]$.

Under the decomposition $\QQ[\tilde{\cS}] = \bigoplus_{e \in \ZZ} \QQ[\tilde{\cS}^e]$ this is the projection of $\tilde{\cH}_{d,0}^j(\tilde{\cR}(\osymb{X}))$ onto the subspace $\QQ[\tilde{\cS}^{\delta}]$. If $\Def(X)$ is odd then by \cref{thm:asai-fourier-commutes} we have $\tilde{\cH}_{d,0}^j(\tilde{\cR}(\osymb{X})) = \tilde{\cR}(\tilde{\cH}_{d,j}^0(\osymb{X}))$ and projecting the right hand side of this onto $\QQ[\tilde{\cS}^{\delta}]$ gives us that
\begin{align*}
\phi_{\osymb{X}}(w) &= \sum_{\lambda \in \mathscr{H}_{d,\delta}(X)}  
\frac{1}{s(X\setminus \lambda)}
\sum_{\substack{Y \in \Sim(X\setminus \lambda)\\\Def(Y) = \Def(Y_{\spec})}}(-1)^{\langle (X\setminus \lambda)^{\sharp},Y^{\sharp}\rangle+l_{d,j}(\lambda,X)}\rho_{\osymb{Y}}(w_2)\\
&= \sum_{\lambda \in \mathscr{H}_{d,\delta}(X)}  
(-1)^{l_{d,j}(\lambda,X)}\phi_{\osymb{X\setminus \lambda}}(w_2).
\end{align*}
If $\Def(X)$ is even then the same holds. but we must multiply through by $(-1)^j$.
\end{proof}

\begin{thm}\label{length-k}
Fix an integer $1 \leqslant k \leqslant n$. Then for each element $w \in W_{\cI}$ of cycle length $k$ and each symbol $\osymb{X} \in \tilde{\cS}_n$ of rank $n$, we have
\begin{equation*}
\abs{\phi_{\osymb{X}}(w)} \leqslant 2^{k-1}\cdot k!.
\end{equation*}
\end{thm}

\begin{proof}
The proof is identical to that of \cref{thm:char-bound-hyperoctahedral}.
\end{proof}

We now associate functions to unordered symbols as follows. If $\symb{X} \in \cS_n^{\odd}$ has odd defect then we simply let $\phi_{\symb{X}} = \phi_{\osymb{X}} \in \Class(W_{\cI})$. Now fix $e \in \{0,1\}$. Then for any $\symb{X} \in \cS_n^{\ev,e}$ we let
\begin{equation*}
\phi_{\symb{X}}
= \Res_{W_{\cI}^e}^{W_{\cI}}(\phi_{\osymb{X}})
= \frac{1}{s(X)}\sum_{\substack{Y \in \Sim(X)\\ \Def(Y) = 0}} (-1)^{\langle X^{\sharp},Y^{\sharp}\rangle}\Res_{W_{\cI}^e}^{W_{\cI}}(\rho_{\osymb{Y}}) \in \Class(W_{\cI}^e).
\end{equation*}
Note that $\Res_{W_{\cI}^e}^{W_{\cI}}(\rho_{\osymb{Y^{\op}}}) = (-1)^e\Res_{W_{\cI}^e}^{W_{\cI}}(\rho_{\osymb{Y}})$ so these functions are nonzero.

\begin{rem}
{\em With some additional justifications, the statement in \cref{prop:phi-MN-rule} may now equally be seen to hold for the class functions $\phi_{\symb{X}}$. If $\symb{X}$ has odd defect then the same statement holds verbatim.

Assume now in the statement of \cref{prop:phi-MN-rule} that $\delta(w) = e$, so that $w \in W_{\cI}^e$. Then it makes sense to consider $\phi_{\symb{X}}(w)$ for any $\symb{X} \in \cS_n^{\ev,e}$. Clearly $w_2 \in W_{\cI}^{e+j}$, and if $\lambda \in \mathscr{H}_{d,j}(X)$ then $\symb{X\setminus\lambda} \in \cS_n^{\ev,e+j}$, so the term $\phi_{\symb{X\setminus \lambda}}(w_2)$ makes sense. Hence, restricting the symbols to $\cS_n^{\ev,2e}$, the statement in \cref{prop:phi-MN-rule} continues to hold.}
\end{rem}

\section{Lusztig series}
In the next few sections we consider a general connected reductive group $\bG$ defined over $\mathbb{F} = \overline{\mathbb{F}_p}$ with Frobenius endomorphism $F : \bG \to \bG$. We will follow the setup in \cite{Lus}, see also the exposition in \cite[Chap.~2]{GeMa}. This setting, whilst a little less frequently used, is more convenient as we wish to discuss character values of Deligne--Lusztig characters. In this section we just outline some notation.

Let $\bT \leqslant \bB \leqslant \bG$ be a fixed $F$-stable maximal torus and Borel subgroup of $\bG$. Let $X = X(\bT) = \Hom(\bT,\GG_m)$, which we view as a $\mathbb{Z}$-module. For any $\phi : \bT \to \bT$ a morphism of algebraic groups we denote by $\phi^* : X \to X$ the map given by $\phi^*(\chi) = \chi\circ\phi$ for all $\chi \in X$.

For each $w \in W := \rN_{\bG}(\bT)/\bT$ we fix an element $n_w \in \rN_{\bG}(\bT)$ such that $w = n_w\bT$. Let $\iota_g : \bG \to \bG$, with $g \in \bG$, be the inner automorphism given by $\iota_g(x) = gxg^{-1}$. Then $Fw := F\iota_{n_w}$ and $wF := \iota_{n_w}F$ are also Frobenius endomorphisms of $\bG$ stabilising $\bT$. We write $w^*$ instead of $(\iota_{n_w}|_{\bT})^*$ and for brevity we let ${}^w\lambda = w^{*-1}(\lambda)$ for all $w \in W$ and $\lambda \in X$.

If $\mathbb{Z}_{(p)}$ is the localisation of $\mathbb{Z}$ at the prime ideal $(p) \subseteq \mathbb{Z}$ containing $p$ then $V = \mathbb{Z}_{(p)} \otimes_{\mathbb{Z}} X$ is a free $\mathbb{Z}_{(p)}$-module of finite rank. We will identify $X$ with its image in $V$ under the canonical map $x \mapsto 1\otimes x$ and similarly any homomorphism $\gamma : X \to X$ is identified with $1\otimes\gamma : V \to V$. Note the quotient $V/X$ is a torsion $\mathbb{Z}$-module.

Abusing terminology we call $W_{\aff} = X \rtimes W$ the affine Weyl group. We have a natural action of $W_{\aff}$ on $V$ given by $(\chi,z)\cdot \lambda = \chi + {}^z\lambda$ for any $(\chi,z) \in W_{\aff}$ and $\lambda \in V$. Let $W_{\aff}(\lambda)$ be the stabiliser of $\lambda$ and let $W(\lambda) = \{w \in W \mid \lambda - {}^w\lambda \in X\}$ be the projection of $W_{\aff}(\lambda)$ onto $W$.

If $\mathbb{Z}\Phi \subseteq X$ is the $\mathbb{Z}$-submodule generated by the roots then $W_{\aff}^{\circ} = \mathbb{Z}\Phi \rtimes W \lhd W_{\aff}$ is the usual affine Weyl group. Again we let $W_{\aff}^{\circ}(\lambda)$ be the stabiliser of $\lambda$ and let $W^{\circ}(\lambda)$ be the projection of $W_{\aff}^{\circ}(\lambda) \lhd W_{\aff}(\lambda)$ on to $W$. We have $W^{\circ}(\lambda)$ is the kernel of the homomorphism $W(\lambda) \to X/\mathbb{Z}\Phi$ given by $w \mapsto (\lambda-{}^w\lambda) + \mathbb{Z}\Phi$.

Given $\lambda \in V$ and $w \in W$ we let $\lambda_w = F^*\lambda - {}^w\lambda \in V$. Note that $(\lambda+\lambda')_w = \lambda_w+\lambda_w'$ for any $\lambda,\lambda' \in V$. Moreover, for any $z \in W$ a straightforward calculation shows that
\begin{equation}\label{eq:conj-identity}
{}^{F^{-1}(z)}\lambda_w = ({}^z\lambda)_{F^{-1}(z)wz^{-1}}.
\end{equation}
The set $\rZ_W(\lambda,F) = \{w \in W \mid \lambda_w \in X\}$ is either empty or a coset of the group $W(\lambda)$.

Denote by $\mathcal{C}_W(X,F)$  (resp. $\mathcal{D}_W(X,F)$) the set of all pairs $(\lambda,w)$ (resp. $(\lambda,a)$) with $\lambda \in V$ and $w \in \rZ_W(\lambda,F)$ (resp. $a = wW^{\circ}(\lambda) \subseteq \rZ_W(\lambda,F)$). By \eqref{eq:conj-identity}, we have a natural action of $W_{\aff}$ on $\mathcal{C}_W(X,F)$ via
\begin{equation*}
(\chi,z)\cdot (\lambda,w) = (\chi+{}^z\lambda,F^{-1}(z)wz^{-1})
\end{equation*}
and a similar action on $\mathcal{D}_W(X,F)$ as $W({}^z\lambda) = {}^zW(\lambda)$ for any $z \in W$ and $\lambda \in V$.

We now fix an injective homomorphism $\kappa : \mathbb{F}^{\times} \to \mathbb{C}^{\times}$. As $\bT^{wF}$ is a $p'$-group, we have a bijection $X(\bT^{wF}) \to \Irr(\bT^{wF})$ given by $\chi \mapsto \kappa\circ\chi$. Given a pair $(\lambda,w) \in \mathcal{C}_W(X,F)$ we set $$\lambda_{wF} = \kappa\circ(\lambda_w|_{\bT^{wF}}) \in \Irr(\bT^{wF}).$$ 
The following is straightforward; see \cite[Lem.~6.2]{Lus} and \cite[Lem.~2.4.8]{GeMa}.

\begin{lem}\label{lem:char-grp-iso}
Fix $w \in W$ and let $V(w) = \{\lambda \in V \mid \lambda_w \in X\}$. Then the map $V(w) \to \Irr(\bT^{wF})$ defined by $\lambda \mapsto \lambda_{wF}$ is a surjective $\mathbb{Z}$-module homomorphism with kernel $X \subseteq V(w)$.
\end{lem}

Given $w \in W$ and $\theta \in \Irr(\bT^{wF})$ we denote by $R_w^{\bG}(\theta)$ the virtual character of $\bG^F$ defined in \cite[Def.~1.9]{DL}. As usual we extend this by linearity to a map on all class functions. Moreover, for any $(\lambda,w) \in \mathcal{C}_W(X,F)$ we set $R_w^{\bG}(\lambda) := R_w^{\bG}(\lambda_{wF})$. We then define for any pair $(\lambda,a) \in \mathcal{D}_W(X,F)$ the set
\begin{equation*}
\mathcal{E}(\bG^F,\lambda,a) = \{\rho \in \Irr(\bG^F) \mid \langle R_w^{\bG}(\lambda),\rho \rangle \neq 0\text{ for some }w \in a\}.
\end{equation*}
This is a rational Lusztig series of $\bG^F$ contained in the geometric series 
$$\mathcal{E}(\bG^F,\lambda) = \{\rho \in \Irr(\bG^F) \mid \langle R_w^{\bG}(\lambda),\rho \rangle \neq 0\text{ for some }w \in \rZ_W(\lambda,F)\}$$ 
indexed by $\lambda \in V$. We have $\mathcal{E}(\bG^F,\lambda,a) = \mathcal{E}(\bG^F,\mu,b)$ if and only if $(\lambda,a)$ and $(\mu,b)$ are in the same $W_{\aff}$-orbit.

Suppose now that $\iota : \bG \to \tilde{\bG}$ is a regular embedding. Then $\tilde{\bT} = \bT\cdot \rZ(\bG)$ is an $F$-stable maximal torus of $\tilde{\bG}$ and if we let $\tilde{X} = X(\tilde{\bT})$ then we have a surjective $\ZZ$-module homomorphism $\iota^* : \tilde{X} \to X$ given by $\iota^*(\chi) = \chi\circ\iota$. Note this maps the roots of $\tilde{\bG}$ in $\tilde{X}$ bijectively onto the roots of $\bG$ in $X$. Through $\iota$ we identify $W$ with the Weyl group $\rN_{\tilde{\bG}}(\tilde{\bT})/\tilde{\bT}$ of $\tilde{\bG}$.

\begin{lem}\label{lem:res-series}
For any $(\tilde{\lambda},a) \in \cD_W(\tilde{X},F)$ we have $W(\tilde{\lambda}) = W^{\circ}(\tilde{\lambda}) = W^{\circ}(\lambda)$ and
\begin{equation*}
\mathcal{E}(\bG^F,\lambda,a) = \{\rho \in \Irr(\bG^F) \mid \langle \rho,\Res_{\bG^F}^{\tilde{\bG}^F}(\tilde{\rho})\rangle \neq 0\text{ for some }\tilde{\rho} \in \mathcal{E}(\tilde{\bG}^F,\tilde{\lambda},a)\}
\end{equation*}
where $\lambda = \iota^*(\tilde{\lambda})$.
\end{lem}

\begin{proof}
By the above remark we have $\tilde{\lambda} - {}^w\tilde{\lambda} \in \ZZ\Phi$ if and only if $\lambda - {}^w\lambda = \iota^*(\tilde{\lambda} - {}^w\tilde{\lambda}) \in \ZZ\Phi$ which shows that $W^{\circ}(\tilde{\lambda}) = W^{\circ}(\lambda)$. From the proof of \cite[Lem.~11.2.1]{DM} we see that the image of the map $W(\tilde{\lambda}) \to \tilde{X}/\ZZ\Phi$ has $p'$-order but as $\rZ(\tilde{\bG})$ is connected the quotient $\tilde{X}/\ZZ\Phi$ has trivial $p'$-torsion so $W(\tilde{\lambda}) = W^{\circ}(\tilde{\lambda})$. The second statement is \cite[Prop.~11.7]{B}.
\end{proof}

\section{A character bound from the Mackey formula}
We denote by $W\sd F$ the semidirect product of $W$ with the group $\langle F\rangle \leqslant \Aut(W)$ such that $FwF^{-1} = F(w)$ for all $w \in W$. The unique coset $WF \subseteq W\sd F$ of $W$ containing $F$ is a $W$-set under conjugation and for $w \in W$ we write $\bC_W(wF)$ for the stabiliser of $wF$ under this action.

For each $w \in W$ we choose an element $g_w \in \bG$ such that $g_w^{-1}F(g_w) = n_w \in \rN_{\bG}(\bT)$. As usual the map $wF \mapsto \bT_w := {}^{g_w}\bT$ yields a bijection between the orbits of $W$ acting on $WF$ and the $\bG^F$-classes of $F$-stable maximal tori. Note that $t\mapsto {}^{g_w}t$ gives an isomorphism $\bT^{wF}\to \bT_w^F$. In particular, we have a bijection $\Irr(\bT_w^F) \to \Irr(\bT^{wF})$ given by $\theta \mapsto {}^{g_w^{-1}}\theta = \theta\circ\iota_{g_w}$.

\begin{rem}\label{rem:conj-DL-chars}
{\em For $\theta \in \Irr(\bT_w^F)$ we have a Deligne--Lusztig character $R_{\bT_w}^{\bG}(\theta)$ defined as in \cite[1.20]{DL} which satisfies $R_{\bT_w}^{\bG}(\theta) = R_w^{\bG}({}^{g_w^{-1}}\theta)$. We will implicitly use this equality in what follows.}
\end{rem}

We have an action of $\bC_W(wF)$ on $\Irr(\bT^{wF})$ by setting ${}^z\theta = \theta\circ\iota_{n_z}^{-1}$ for any $z \in \bC_W(wF)$ and $\theta \in \Irr(\bT^{wF})$. We denote by $\bC_W(wF,\theta) \leqslant \bC_W(wF)$ the stabiliser of $\theta \in \Irr(\bT^{wF})$. Now for any $(\lambda,w) \in \mathcal{C}_W(X,F)$ and $z \in W$ it follows from \eqref{eq:conj-identity} that ${}^z(\lambda_{wF}) = (F(z)^{*-1}\lambda)_{zwFz^{-1}}$. Therefore,
\begin{equation}\label{eq:cent-comp}
F(\bC_W(wF,\lambda_{wF})) = \bC_{W(\lambda)}(Fw),
\end{equation}
where the latter stabiliser is calculated with respect to the action of $W(\lambda)$ on $F\rZ_W(\lambda,F)$.

A regular semisimple element $g \in \bG^F$ is said to be of type $wF \in WF$ if $\bC_{\bG}^{\circ}(g)$ is $\bG^F$-conjugate to $\bT_w$. Of course, the type is only determined up to $W$-conjugacy. Moreover, an element $g$ of type $wF$ is then $\bG^F$-conjugate to an element of the form ${}^{g_w}t$ with $t \in \bT^{wF}$.

As in \cite{DM}, if $H$ is a finite group then we denote by $\pi_h^H \in \Class(H)$ the function taking the value $\abs{\bC_H(h)}$ on $\Cl_H(h)$ and the value $0$ on $H - \Cl_H(h)$. For any $f \in \Class(H)$ we then have $\langle f,\pi_h^H\rangle = f(h)$. We also write $[g,h] = g^{-1}h^{-1}gh$ for the commutator of $g,h \in H$.

\begin{thm}\label{prop:mackey-bound}
Assume $g \in \bG^F$ is a regular semisimple element of type $wF$. Fix a pair $(\lambda,a) \in \mathcal{D}_W(X,F)$ and let 
$$\mathcal{X}_w(\lambda,a) := \{z \in W \mid [z,Fw]w^{-1}a = W^{\circ}(\lambda)\}.$$ 
If $\chi \in \mathcal{E}(\bG^F,\lambda,a)$ and $x \in a$ then
\begin{equation*}
\abs{\chi(g)} \leqslant \sum_{z \in \bC_W(Fw)\backslash\mathcal{X}_w(\lambda,a)/\bC_W(Fx,\lambda)} \big( \abs{\bC_W(Fw)}/\abs{\bC_{W({}^z\lambda)}(Fw)}^{\frac{1}{2}} \big) \leqslant \abs{W}\cdot \abs{\bC_W(Fw)}
\end{equation*}
\end{thm}

\begin{proof}
Assume $t \in \bT^{wF}$ is such that ${}^{g_w}t \in \Cl_{\bG^F}(g)$. By \cite[Prop.~9.18]{DL}, see also \cite[Prop.~10.3.6]{DM}, we have $\pi_g^{\bG^F} = R_w^{\bG}(\pi_t^{\bT^{wF}}) = \sum_{\theta \in \Irr(\bT^{wF})} \theta(t^{-1})R_w^{\bG}(\theta)$. For any $\theta \in \Irr(\bT^{wF})$ we have by the Mackey formula for tori, see \cite[Thm.~6.8]{DL} or \cite[Cor.~9.3.1]{DM}, that
\begin{equation*}
\sum_{\chi \in \Irr(\bG^F)} \langle \chi,R_w^{\bG}(\theta)\rangle^2 = \langle R_w^{\bG}(\theta),R_w^{\bG}(\theta)\rangle = \abs{\bC_W(wF,\theta)}.
\end{equation*}
Applying this to $\langle \chi, \pi_g^{\bG^F}\rangle$ we get
\begin{equation}\label{eq:intermediate-bound}
\abs{\chi(g)} \leqslant \sum_{\theta \in \Irr(\bT^{wF})} \abs{\langle \chi,R_w^{\bG}(\theta)\rangle} \leqslant \sum_{\theta \in \Irr(\bT^{wF})} \abs{\bC_W(wF,\theta)}^{\frac{1}{2}}.
\end{equation}

Let $V(w)$ be as in \cref{lem:char-grp-iso} so that $V(w)/X \cong \Irr(\bT^{wF})$. If $\mu \in V(w)$ is such that $\langle \chi, R_w^{\bG}(\mu)\rangle \neq 0$ then we must have $(\mu,wW^{\circ}(\mu))$ and $(\lambda,a)$ are in the same $W_{\aff}$-orbit by the disjointness of Lusztig series. This happens if and only if there exists a $z \in W$ such that $\mu - {}^z\lambda \in X$ and
\begin{equation*} 
F^{-1}(z)az^{-1} = wW^{\circ}(\mu) = wzW^{\circ}(\lambda)z^{-1}.
\end{equation*}
This last condition is equivalent to $z \in \mathcal{X}_w(\lambda,a)$.

If $z \in \mathcal{X}_w(\lambda,a)$ and $v \in \bC_W(Fx,\lambda)$ then one checks easily that $zv \in \mathcal{X}_w(\lambda,a)$, and clearly ${}^{zv}\lambda = {}^z\lambda$. Hence, the sum in \eqref{eq:intermediate-bound} can be taken over $\mathcal{X}_w(\lambda,a)/\bC_W(Fx,\lambda)$ with $\theta$ replaced by $({}^z\lambda)_{wF}$. The group $\bC_W(Fw) = F(\bC_W(wF))$ acts on $\mathcal{X}_w(\lambda,a)$ by left multiplication, and the term in the sum is constant on orbits. Hence, we can sum as in the statement once we multiply through by the size of the orbit $\abs{\bC_W(Fw)/\bC_{W({}^z\lambda)}(Fw)}$. We now cancel terms using \eqref{eq:cent-comp}.
\end{proof}

\begin{rem}
{\em The above follows the same argument as \cite[Thm.~5.4]{GM} where one finds the bound $\abs{\bC_W(Fw)}$. The above yields this bound when the sum contains only one term. However, this does not hold in general so this should be corrected as above. In the extreme cases where $W(\lambda) = W$ or $W^{\circ}(\lambda) = \{1\}$ then this does give the bound $\abs{\bC_W(Fw)}$.}
\end{rem}

\section{Further bounds from Deligne--Lusztig characters}
The following result giving the value of a Deligne--Lusztig character at a regular semisimple element is well known. We simply translate this into the setup we utilise here, see \cite[Prop.~4.5.8]{Ge} for an equivalent formulation.

\begin{prop}\label{prop:val-DL-char}
Assume $g \in \bG^F$ is a regular semisimple element of type $wF$ and let $t \in \bT^{wF}$ be such that ${}^{g_w}t \in \Cl_{\bG^F}(g)$. Then for any $(\lambda,x) \in \mathcal{C}_W(X,F)$ we have
\begin{equation*}
R_x^{\bG}(\lambda)(g) = \sum_{\substack{\mu \in V(w)/X \\ (\mu,x) \in \Cl_W(\lambda,w)}} \abs{\bC_{W(\mu)}(Fw)} \mu_{wF}(t) = \abs{\bC_{W(\lambda)}(Fx)}\sum_{\substack{z \in W/\bC_{W(\lambda)}(Fx) \\ Fw = zFxz^{-1}}} ({}^z\lambda)_{wF}(t).
\end{equation*}
In particular, we have $\abs{R_x^{\bG}(\lambda)(g)} \leqslant \abs{\bC_W(Fw)}$.
\end{prop}

\begin{proof}
As noted in the proof of \cref{prop:mackey-bound} we simply have to calculate $\langle R_x^{\bG}(\lambda), R_w^{\bG}(\pi_t^{\bT^{wF}})\rangle$. The statement now follows immediately from the inner product formula for Deligne--Lusztig characters and the identification $V(w)/X \cong \Irr(\bT^{wF})$.
\end{proof}

To make invoking Lusztig's classification results for the irreducible characters of $\bG^F$ simpler it will be beneficial to assume that $\rZ(\bG)$ is connected. As usual we invoke a regular embedding to achieve this. We note that for our purposes we do not need the significantly more difficult multiplicity freeness results obtained by Lusztig.

\begin{lem}\label{lem:ss-reg-embed}
Assume $\bG \to \tilde{\bG}$ is a regular embedding and let $\chi \in \Irr(\bG^F)$ be an irreducible constituent of $\Res_{\bG^F}^{\tilde{\bG}^F}(\tilde{\chi})$ for some $\tilde{\chi} \in \Irr(\tilde{\bG}^F)$. Then for any semisimple element $g \in \bG^F$ we have
\begin{equation*}
\abs{\chi(g)} \leqslant \abs{\tilde{\bG}^F/\tilde{\bG}_{\chi}^F}^{-1}\abs{\tilde{\chi}(g)} \leqslant \abs{\tilde{\chi}(g)},
\end{equation*}
where $\tilde{\bG}_{\chi}^F$ is the stabiliser of $\chi$ in $\tilde{\bG}^F$.
\end{lem}

\begin{proof}
Let $\bT \leqslant \bG$ be an $F$-stable maximal torus and $\theta \in \Irr(\bT^F)$. We then have $\tilde{\bT} = \bT\cdot\rZ(\tilde{\bG})$ is an $F$-stable maximal torus of $\tilde{\bG}$. Let $\tilde{\theta} \in \Irr(\tilde{\bT}^F)$ be an irreducible character such that $\Res_{\bT^F}^{\tilde{\bT}^F}(\tilde{\theta}) = \theta$. By \cite[Prop.~10.10]{B} we have $\Res_{\bG^F}^{\tilde{\bG}^F}(R_{\tilde{\bT}}^{\tilde{\bG}}(\tilde{\theta})) = R_{\bT}^{\bG}(\theta)$. Therefore, by Frobenius reciprocity,
\begin{equation*}
\langle \chi, R_{\bT}^{\bG}(\theta)\rangle = \langle \Ind_{\bG^F}^{\tilde{\bG}^F}(\chi), R_{\tilde{\bT}}^{\tilde{\bG}}(\tilde{\theta})\rangle.
\end{equation*}
This shows that for any $c \in \tilde{\bG}^F$ we have ${}^c\chi \in \Irr(\bG^F)$ and $\chi$ have the same uniform projection. As $\pi_g^{\bG^F}$ is a uniform function this means that $\chi(g) = \langle \chi,\pi_g^{\bG^F}\rangle = \langle {}^c\chi,\pi_g^{\bG^F}\rangle = {}^c\chi(g)$. Now by Clifford's Theorem we have
\begin{equation*}
\Res_{\bG^F}^{\tilde{\bG}^F}(\tilde{\chi}) = e\sum_{c \in \tilde{\bG}^F/\tilde{\bG}_{\chi}^F} {}^c\chi
\end{equation*}
for some integer $e \geqslant 1$. Hence $\tilde{\chi}(g) = e\abs{\tilde{\bG}^F/\tilde{\bG}_{\chi}^F}\chi(g)$ giving the bound.
\end{proof}

Fix a pair $(\lambda,a) \in \mathcal{D}_W(X,F)$. For a class function $f \in \Class(Fa)$ on the coset of $W^{\circ}(\lambda)$ we define a corresponding class function
\begin{equation}\label{eq:almost-char-def}
\mathcal{R}_{\lambda,a}^{\bG}(f) = \frac{1}{\abs{W^{\circ}(\lambda)}}\sum_{x \in a} f(Fx)R_x^{\bG}(\lambda).
\end{equation}
of $\bG^F$. Clearly this is contained in the subspace of all $\mathbb{C}$-class functions $\Class(\bG^F,\lambda,a)$ spanned by the Lusztig series $\mathcal{E}(\bG^F,\lambda,a)$. In fact, $\cR_{\lambda,a}^{\bG}$ gives an isomorphism $\Class(Fa) \to \Class_0(\bG^F,\lambda,a)$ onto the subspace spanned by $\{R_x^{\bG}(\lambda) \mid x \in a\}$; see the arguments in \cite[\S11.6]{DM}.

Suppose we choose a representative $a = wW^{\circ}(\lambda)$ of the coset. We may then form the semidirect product $W^{\circ}(\lambda) \sd Fw$ by the group $\langle Fw\rangle \leqslant \Aut(W^{\circ}(\lambda))$ as above. The coset of $W^{\circ}(\lambda)$ in $W^{\circ}(\lambda) \sd Fw$ containing $Fw$ can be identified, $W^{\circ}(\lambda)$-equivariantly, with the same coset in $W\sd F$.

We denote by $\Irr(Fw.W^{\circ}(\lambda))$ the set of restrictions $\Res_{Fa}^{W^{\circ}(\lambda)\sd Fw}(\tilde{\phi})$, where $\tilde{\phi} \in \Irr(W^{\circ}(\lambda)\sd Fw)$ restricts irreducibly to $W^{\circ}(\lambda)$. These functions on the coset $Fa$ depend on our choice of representative $w$, so we include a period to indicate this choice. There is, however, a natural choice $w_a \in a$ which is the element of minimal length (determined by our choice of Borel subgroup $\bB$).

\begin{lem}\label{lem:almost-char-bound}
Assume $g \in \bG^F$ is a regular semisimple element of type $wF$. For any irreducible character $\phi \in \Irr(Fx.W^{\circ}(\lambda))$, with $x \in a$, we have
\begin{equation*}
\abs{\mathcal{R}_{\lambda,a}^{\bG}(\phi)(g)} \leqslant \abs{\bC_W(Fw)}
\end{equation*}
\end{lem}

\begin{proof}
By \cref{prop:val-DL-char} we have
\begin{equation*}
\abs{\mathcal{R}_{\lambda,a}^{\bG}(\phi)(g)} \leqslant \abs{\bC_W(Fw)} \left( \frac{1}{\abs{W^{\circ}(\lambda)}} \sum_{y \in W^{\circ}(\lambda)} \abs{\phi(Fxy)}\right),
\end{equation*}
but by \cite[Lem.~8.14(c)]{I}, the sum on the right hand side is equal to $1$.
\end{proof}

\begin{lem}\label{lem:value-almost-char}
Assume $g \in \bG^F$ is a regular semisimple element of type $wF$. Then for any class function $f \in \Class(Fa)$ we have
\begin{equation*}
\mathcal{R}_{\lambda,a}^{\bG}(f)(g) = \sum_{\substack{z \in \bC_W(Fw)\backslash W/W^{\circ}(\lambda) \\ Fx = z^{-1}Fwz \in Fa}} \frac{f(Fx) \cdot R_x^{\bG}(\lambda)(g)}{\abs{\bC_{W^{\circ}(\lambda)}(Fx)}}.
\end{equation*}
\end{lem}

\begin{proof}
By \cref{prop:val-DL-char} we may restrict the sum over $x \in a$, found in the definition of $\mathcal{R}_{\lambda,a}^{\bG}(f)$, to those elements satisfying $Fx \in \Cl_W(Fw)$. Alternatively, via the bijection $\bC_W(Fw)\backslash W \to \Cl_W(Fw)$, given by $\bC_W(Fw)z \mapsto z^{-1}Fwz$, we can sum over all cosets $\bC_W(Fw)z \in \bC_W(Fw)\backslash W$ such that $z^{-1}Fwz \in Fa$. Grouping together elements in the same $W^{\circ}(\lambda)$-orbit and bringing $\abs{W^{\circ}(\lambda)}$ into the sum gives the statement.
\end{proof}

Observe that, if $W^{\circ}(\lambda) = W$ and $w \in a$ then $\mathcal{R}_{\lambda,a}^{\bG}(f)(g) = f(Fw)\lambda_{wF}(t)$, see \cite[Prop.~3.3]{LM}. Together with \cref{thm:char-bound-hyperoctahedral} and \cref{length-k}, this implies

\begin{cor}\label{unip-bound}
Suppose $\bG$ is a quasisimple group of classical type $\sA$ to $\sD$. If $g \in \bG^F$ is a regular semisimple element of cycle length $k$ and $\chi \in \Irr(\bG^F)$ is a unipotent character, then $\abs{\chi(g)} \leq 2^{2k} \cdot k!$.
\end{cor}

Now suppose $[\bG,\bG]$ is quasisimple of type $\sA_{n-1}$. In this case we have $W$ is isomorphic to $\Sym_n$ and $F$ induces an inner automorphism on $W$, which is either trivial or conjugation by the longest element. Hence, the coset $WF$ can be identified with $W$ so it makes sense to speak of the cycle type of an element of $WF$.

\begin{cor}\label{glu-bound2}
Assume all the quasi-simple components of $\bG$ are of type $\sA$. If $g \in \bG^F$ is a regular semisimple element of type $wF$ then
\begin{equation*}
\abs{\chi(g)} \leqslant \abs{\bC_W(Fw)}
\end{equation*}
for any $\chi \in \Irr(\bG^F)$. Moreover, suppose $[\bG,\bG]$ is quasisimple of type $\sA_{n-1}$ with $n \geqslant 2$. If $wF$ has cycle length $k \geqslant 1$ then $\abs{\chi(g)} \leqslant k! \cdot n^k$ for any $\chi \in \Irr(\bG^F)$.
\end{cor}

\begin{proof}
By \cref{lem:ss-reg-embed}, we can assume that $\rZ(\bG)$ is connected. In that case every irreducible character is, up to sign, of the form $\mathcal{R}_{\lambda,a}^{\bG}(\phi)$ with $\phi \in \Irr(Fw.W(\lambda))$ so this is just \cref{lem:almost-char-bound}.

For the final statement we need only show that $\abs{\bC_{\Sym_n}(w)} \leqslant k!\cdot n^k$. If all cycles have the same length, say $m \geqslant 1$, then $\bC_{\Sym_n}(w) \cong C_m \wr \Sym_k$ so $\abs{\bC_{\Sym_n}(w)} = k! \cdot m^k$. Now an arbitrary $w$ may be written as a pairwise commuting product $w = w_1\cdots w_r$ such that for each $1 \leqslant i \leqslant r$ we have $w_i$ is a product of $k_i \geqslant 1$ disjoint cycles of length $m_i \geqslant 1$ and the lengths $m_1,\dots,m_r$ are pairwise distinct. We then have 
\begin{equation*}
\bC_{\Sym_n}(w) \cong \bC_{\Sym_{k_1m_1}}(w_1) \times \cdots \times \bC_{\Sym_{k_rm_r}}(w_r),
\end{equation*}
and by the previous calculation 
\begin{equation*}
\abs{\bC_{\Sym_n}(w)} = (k_1!\cdot m_1^{k_1})\cdots (k_r!\cdot m_r^{k_r}) \leqslant k!\cdot n^{k_1+\cdots +k_r} = k! \cdot n^k.\qedhere
\end{equation*}
\end{proof}

For the next statement we wish to define an integer $r(W,F) \geqslant 0$ as follows. Let $\mathbb{S} \subseteq W$ be the set of Coxeter generators determined by our choice of Borel $\bB$. Write $W = W_1\cdots W_m$ as a direct product of its irreducible components, all of which are assumed to be of type $\sA$ through $\sD$. We then have a corresponding decomposition $\mathbb{S} = \mathbb{S}_1\sqcup \cdots \sqcup \mathbb{S}_m$. The Frobenius $F$ permutes the $W_i$. Suppose first that it does so transitively. Then we define 
\begin{equation*}
r(W,F) = \begin{cases}
0 &\text{if $W_1$ is of type $\sA_n$ with $n \geqslant 0$},\\
\abs{\mathbb{S}_1} &\text{otherwise}.
\end{cases}
\end{equation*}
Here we consider the trivial group as being of type $\sA_0$.

Now grouping together the $W_i$ we can write $W = W^{(1)}\cdots W^{(n)}$ where each $W^{(i)}$ is an $F$-stable subgroup such that $F$ permutes transitively its irreducible components. Hence, we are in the previous situation and we define $r(W,F) = r(W^{(1)},F)+\cdots+r(W^{(n)},F)$.

\begin{thm}\label{thm:main-bound-classical}
Assume $g \in \bG^F$ is a regular semisimple element of type $wF$ and every quasisimple component of $\bG$ is of classical type $\sA$ to $\sD$. If $F$ is a Frobenius endomorphism then for any irreducible character $\chi \in \mathcal{E}(\bG^F,\lambda,a)$ we have
\begin{equation*}
\abs{\chi(g)} \leqslant 2^r\cdot \abs{\bC_W(Fw)}
\end{equation*}
where $w_a \in a$ is the unique element of minimal length and $r = r(W^{\circ}(\lambda),Fw_a)$ is defined as above.
\end{thm}

\begin{proof}
Again, by \cref{lem:ss-reg-embed} we can assume $\rZ(\bG)$ is connected. By \cref{lem:res-series} this implies that $W(\lambda) = W^{\circ}(\lambda)$ and $a = \rZ_W(\lambda,F)$ so $\mathcal{E}(\bG^F,\lambda,a) = \mathcal{E}(\bG^F,\lambda)$. Recall that in \cite[Chp.~4]{Lus} Lusztig has defined a partition of $\Irr(W(\lambda))$ into families.

Denote by $w_a \in a = \rZ_W(\lambda,F)$ the unique element of minimal length. The automorphism $\gamma := Fw_a$ of $W(\lambda)$ permutes the families. Suppose $\mathcal{F} \subseteq \Irr(W(\lambda))$ is a $\gamma$-stable family. For each $\gamma$-fixed character $\phi \in \mathcal{F}^{\gamma}$ we fix an extension $\tilde{\phi} \in \Irr(W(\lambda) \sd \gamma)$ that is realisable over $\mathbb{Q}$.

Suppose $\mathcal{F} \subseteq \Irr(W(\lambda))$ is $\gamma$-stable and let 
$$\mathcal{E}(\bG^F,\lambda,\mathcal{F}) = \{\chi \in \Irr(\bG^F) \mid \langle \chi, \mathcal{R}_{\lambda,a}^{\bG}(\tilde{\phi})\rangle \neq 0 \mbox{ for some }\phi \in \mathcal{F}^{\gamma}\}.$$ 
These sets partition $\mathcal{E}(\bG^F,\lambda)$; see \cite[Thm.~6.17]{Lus}. Associated to $\mathcal{F}$ we have a corresponding finite group $\mathcal{G}_{\mathcal{F}}$.

As all factors are of classical type, we have $\mathcal{G}_{\mathcal{F}}$ is a (possibly trivial) elementary abelian $2$-group. We also have two sets $\overline{\mathcal{M}}(\mathcal{G_F},\gamma)$ and $\mathcal{M}(\mathcal{G_F},\gamma)$ and a pairing 
$$\{-,-\} : \overline{\mathcal{M}}(\mathcal{G_F},\gamma) \times \mathcal{M}(\mathcal{G_F},\gamma) \to \mathbb{C}.$$ 
From the formula for this pairing, we see that
\begin{equation*}
\abs{\{\bar{x},x\}} = \abs{\mathcal{G_F}}^{-1}
\end{equation*}
for any $\bar{x} \in \overline{\mathcal{M}}(\mathcal{G_F},\gamma)$ and $x \in \mathcal{M}(\mathcal{G_F},\gamma)$

By \cite[Thm.~4.23]{Lus}, we have a bijection $\mathcal{E}(\bG^F,\lambda,\mathcal{F}) \to \overline{\mathcal{M}}(\mathcal{G_F},\gamma)$, which we denote by $\chi \mapsto x_{\chi}$, and an injection $\mathcal{F}^{\gamma} \to \mathcal{M}(\mathcal{G_F},\gamma)$, denoted by $\phi \mapsto x_{\tilde{\phi}}$. This latter map depends on our choice of extension. Now, by \cite[4.26.1]{Lus}, if $\chi \in \mathcal{E}(\bG^F,\lambda,\mathcal{F})$ then
\begin{equation*}
\chi(g) = \pm \sum_{\phi \in \mathcal{F}^{\gamma}} \{\bar{x}_{\chi}, x_{\tilde{\phi}}\}\mathcal{R}_{\lambda,a}^{\bG}(\tilde{\phi})(g).
\end{equation*}
The group of roots of unity acts on $\mathcal{M}(\mathcal{G_F},\gamma)$, and the number of orbits is the same as $\abs{\overline{\mathcal{M}}(\mathcal{G_F},\gamma)} = \abs{\mathcal{G_F}}$. It follows from \cite[4.21.6]{Lus} that $\abs{\mathcal{F}^{\gamma}} \leqslant \abs{\mathcal{G_F}}^2$. We now use \cref{lem:almost-char-bound}.
\end{proof}

We assume $F$ is a Frobenius endomorphism. If $[\bG,\bG]$ is quasisimple of type $\sB_n$ or $\sC_n$ then $W \cong W_{\cI}$ is a hyperoctahedral group, and $F$ induces the identity on $W$. If $[\bG,\bG]$ is quasisimple of type $\sD_n$ then $W \cong W_{\mathcal{I}}^0$ and either $F$, $F^2$, or $F^3$ induces the identity on $W$. When $F^2$ induces the identity on $W$, we have an embedding $W\sd F \to W_{\cI}$. Thus, it makes sense to speak of the \emph{cycles} of an element of the coset $WF$. The following is now just a simple application of \cref{lem:cents-pairwise-distinct}.

\begin{cor}\label{bcd-bound-distinct}
Assume $[\bG,\bG]$ is quasisimple of type $\sB_n$ ($n\geqslant 2$), $\sC_n$ ($n\geqslant 2$), or $\sD_n$ ($n\geqslant 4$), and $F$ is a Frobenius endomorphism with $F^2$ inducing the identity on $W$. If $wF \in WF$ has cycle length $k \geqslant 1$ and pairwise distinct cycles, then for any regular semisimple element $g \in \bG^F$ of type $wF$ we have
\begin{equation*}
\abs{\chi(g)} \leqslant 2^{n+k}\cdot n^k
\end{equation*}
for all $\chi \in \Irr(\bG^F)$.
\end{cor}

\section{Quadratic unipotent characters}\label{quadratic-unipotent}
Recall that a bound for the values of unipotent characters at regular semisimple elements was obtained in \cref{unip-bound}.
In this section, we establish a bound for the more general class of quadratic unipotent characters. Consider a connected reductive group $\bG$ whose center $\rZ(\bG)$ is connected, of dimension $1$, and whose derived subgroup $\bG_{\der}=[\bG,\bG] \leqslant \bG$ is a symplectic or special orthogonal group. We specify $\bG$ by its root datum $(X,\Phi,\wc{X},\wc{\Phi})$. Firstly we have $X = \bigoplus_{i=0}^n\mathbb{Z}e_i$ and $\wc{X} = \bigoplus_{i=0}^n\mathbb{Z}\wc{e}_i$ with perfect pairing $\langle -,-\rangle : X\times\wc{X} \to \mathbb{Z}$ given by $\langle e_i,\wc{e}_j\rangle = \delta_{ij}$ (the Kronecker delta). We assume $n \geqslant 2$.

A set of simple roots $\alpha_1,\dots,\alpha_n$ and corresponding coroots $\wc{\alpha}_1,\dots,\wc{\alpha}_n$ are as follows. We have $(\alpha_1,\wc\alpha_1)$ is one of the pairs $(-e_1,-2\wc{e}_1)$, $(e_0-2e_1,-\wc{e}_1)$, or $(e_0-e_1-e_2,-\wc{e}_1-\wc{e}_2)$. Then for $2 \leqslant i \leqslant n$ we let $\alpha_i = e_{i-1}-e_i$ and $\wc\alpha_i = \wc{e}_{i-1}-\wc{e}_i$. The choices correspond to whether $\bG$ is of type $\sB_n$, $\sC_n$, or $\sD_n$, respectively.

One easily calculates that $(\langle \alpha_i,\wc{\alpha}_j\rangle)$ is a Cartan matrix and $X/\mathbb{Z}\Phi \cong \mathbb{Z}$ is generated by $e_0+\mathbb{Z}\Phi$. Let $X_{\der} = X/(\mathbb{Z}e_0)$ and $\wc{X}_{\der} = \bigoplus_{i=1}^n\wc{e}_i$. Denote by $\overline{\phantom{\alpha}} : X \to X_{\der}$ the natural quotient map. We then have $(X_{\der},\overline{\Phi},\wc{X}_{\der},\wc{\Phi})$ is the root datum of $\bG_{\der}$.

Fix a prime power $q = p^a$. We describe $F$ by defining $F^*$ as an endomorphism of $X$. For any $0 \leqslant i \leqslant n$ with $i \neq 1$ we have $F^*e_i = qe_i$. We then have $F^*e_1$ is either $qe_1$ or $q(e_0-e_1)$ with this latter case occurring only when $\bG^F$ is of type ${}^2\sD_n(q)$.

If $V_{\der} = \mathbb{Z}_{(p)} \otimes_{\mathbb{Z}} X_{\der}$ then we have $V_{\der} \cong \bigoplus_{i=1}^n \mathbb{Z}_{(p)}\overline{e}_i$ and $V \cong \bigoplus_{i=1}^{n+1} \mathbb{Z}_{(p)}e_i$. The natural quotient map $X \to X_{\der}$ extends to a surjective $\mathbb{Z}_{(p)}$-module homomorphism $\overline{\phantom{\lambda}} : V \to V_{\der}$. Given $\lambda \in V$ we have $W(\lambda) = W^{\circ}(\lambda) = W^{\circ}(\overline{\lambda})$ and if $\chi \in \mathcal{E}(\bG^F,\lambda,a)$ then all irreducible constituents of $\chi$ are contained in $\mathcal{E}(\bG_{\der}^F,\overline{\lambda},a)$, see \cref{lem:res-series}.

Consider the totally ordered set $\mathcal{I} = \{-\overline{e}_n \prec \cdots \prec -\overline{e}_1\prec \overline{e}_1 \prec \cdots \prec \overline{e}_n\}$. Letting $F$ act on $X$ via $q^{-1}F^*$ we get an action of $W\sd F$ on $X_{\der}$ that gives an injective homomorphism $W\sd F \to \Sym_{\mathcal{I}}$. We implicitly identify $W\sd F$, hence also $W$, with its image which is contained in $W_{\mathcal{I}}$. Via this identification we can speak of the {\it signed cycle type} of any element of $W\sd F$.

\begin{thm}\label{quadratic}
Let $(\lambda,a) \in \mathcal{D}_W(X,F)$ be such that $2\overline{\lambda} \in X_{\der}$ and let $g \in \bG^F$ be a regular semisimple element of type $wF$. If $wF$ has cycle length $k \geqslant 1$ and pairwise distinct cycles then
\begin{equation*}
\abs{\chi(g)} \leqslant 2^{3k+4}\cdot k!
\end{equation*}
for any irreducible character $\chi \in \cE(\bG^F,\lambda,a)$.
\end{thm}

\begin{proof}
Recall that we have an isomorphism $\cR_{\lambda,a}^{\bG} : \Class(Fa) \to \Class_0(\bG^F,\lambda,a)$. Thus, if $\pr : \Class(\bG^F) \to \Class(\bG^F)$ is the projection onto the subspace of uniform functions then there exists a unique class function $f_{\chi} \in \Class(Fa)$ such that $\cR_{\lambda,a}^{\bG}(f_{\chi}) = \pr(\chi)$.

Now $\pr(\chi)$ and $\chi$ have the same value at $g$, so it suffices to bound the value of $\cR_{\lambda,a}^{\bG}(f_{\chi})$ at $g$. By \cref{prop:val-DL-char,lem:value-almost-char}, we have
\begin{equation*}
\abs{\mathcal{R}_{\lambda,a}^{\bG}(f_{\chi})(g)} \leqslant \sum_{\substack{z \in \bC_W(Fw)\backslash W/W(\lambda) \\ Fx = z^{-1}Fwz \in Fa}}  \frac{\abs{\bC_W(Fx)}}{\abs{\bC_{W(\lambda)}(Fx)}} \cdot \abs{f_{\chi}(Fx)}.
\end{equation*}
We can assume $\mathcal{R}_{\lambda,a}^{\bG}(f_{\chi})(g) \neq 0$ and hence assume that $Fx = z^{-1}Fwz \in Fa$ is a conjugate of $Fw$.

We now bound: $\abs{\bC_W(Fx)}/\abs{\bC_{W(\lambda)}(Fx)}$, the number of terms in the sum, and finally $\abs{f_{\chi}(Fx)}$. Let us note that the number of terms in the sum is precisely the number of $W(\lambda)$-orbits on $Fa$ that meet the centraliser $\bC_W(Fw)$.

We will take this case by case. First let us note that as $wF$ has pairwise distinct cycles, so does its conjugate $Fx = zF(wF)F^{-1}z^{-1}$. Now, by replacing $\lambda$ with an element in the same $W_{\aff}$-orbit, we can assume that $\overline{\lambda} = \frac{1}{2}(\overline{e}_1+\cdots+\overline{e}_m)$ for some $0 \leqslant m \leqslant n$ where $\overline{\lambda} = 0$ when $m = 0$.

We set $\mathcal{I}_1 = \{\pm\overline{e}_i \mid 1 \leqslant i \leqslant m\}$ and $\mathcal{I}_0 = \mathcal{I}\smallsetminus\mathcal{I}_1$ and let $H = W_{\mathcal{I}_0}W_{\mathcal{I}_1} \leqslant W_{\mathcal{I}}$. For convenience, we let $\pi_i = (\overline{e}_i,-\overline{e}_i) \in W_{\mathcal{I}}$ for any $1 \leqslant i \leqslant n$.

\underline{Type $\sB_n$}. We have $W(\lambda) = Fa = H$. Lusztig has shown that there is an isomorphism 
$$\mathscr{U}_m^{\odd}\otimes_{\mathbb{C}} \mathscr{U}_{n-m}^{\odd} \to \Class(\bG^F,\lambda,a)$$
such that the natural basis $\{\symb{X}\otimes\symb{Y} \mid \symb{X} \in \cS_m^{\odd}$ and $\symb{Y} \in \cS_{n-m}^{\odd}\}$ maps onto the series $\cE(\bG^F,\lambda)$.

The images of the Fourier transforms $\cR_0(\symb{X})\otimes \cR_0(\symb{Y})$ are Lusztig's almost characters. By \cite[4.23]{Lus}, this bijection may be chosen such that if $\osymb{X} \in \tilde{\cS}_m^1$ and $\osymb{Y} \in \tilde{\cS}_{n-m}^1$ have defect $1$ then
\begin{equation*}
\cR_0(\symb{X})\otimes \cR_0(\symb{Y}) \mapsto \cR_{\lambda,a}^{\bG}(\rho_{\osymb{X}} \boxtimes \rho_{\osymb{Y}}).
\end{equation*}
As $\symb{X}\otimes \symb{Y} = \cR_0\cQ_0(\symb{X})\otimes \cR_0\cQ_0(\symb{Y})$, we see that if $\chi$ is the image of $\symb{X}\otimes \symb{Y}$ then $f_{\chi} = \phi_{\symb{X}} \boxtimes \phi_{\symb{Y}}$.

By (iii) of \cref{lem:cents-pairwise-distinct}, we have $\bC_W(Fx) = \bC_H(Fx)$. Write $Fx = x_0x_1$ with $x_i \in W_{\mathcal{I}_i}$. If $k_i\geqslant 1$ is the cycle length of $x_i$ then $k = k_0+k_1$. Using \cref{thm:char-bound-hyperoctahedral}, we thus get the following bound on the character value
\begin{equation*}
\abs{f_{\chi}(Fx)} = \abs{\phi_{\symb{X}}(x_0)} \cdot \abs{\phi_{\symb{Y}}(x_1)} \leqslant (2^{k_0-1}\cdot k_0!) \cdot (2^{k_1-1}\cdot k_1!) \leqslant 2^{k-2} \cdot k!.
\end{equation*}

If $\alpha = (\alpha_1,\dots,\alpha_{k_0})$ and $\beta = (\beta_1,\dots,\beta_{k_1})$ are the signed cycle types of $x_0$ and $x_1$ respectively then, up to reordering the entries, $\alpha\cup\beta = (\alpha_1,\dots,\alpha_{k_0},\beta_1,\dots,\beta_{k_1})$ is the signed cycle type of $Fx$. Thus, there can certainly be at most $\sum_{c=0}^k \binom{k}{c} = 2^k$ terms in the above sum. Putting things together we get the bound $2^{2k-2}\cdot k!$ in this case.

\underline{Type $\sC_n$}. We have $W(\lambda) = W_{\mathcal{I}_0}^0W_{\mathcal{I}_1}$ and $Fa = W_{\mathcal{I}_0}^eW_{\mathcal{I}_1}$ for some $e \in \{0,1\}$. In this case, we have an isomorphism 
$$\mathscr{U}_m^{\ev,e}\otimes_{\CC} \mathscr{U}_{n-m}^{\odd} \to \Class(\bG^F,\lambda,a).$$
The natural basis $\{\symb{X}\otimes\symb{Y} \mid \symb{X} \in \cS_m^{\ev,e}$ and $\symb{Y} \in \cS_{n-m}^{\odd}\}$ maps onto the series $\cE(\bG^F,\lambda)$.

By \cite[4.23]{Lus}, this bijection may be chosen such that for any $\osymb{X} \in \tilde{\cS}_m^0$ and $\osymb{Y} \in \tilde{\cS}_{n-m}^1$, we have
\begin{equation*}
\cR_e(\symb{X})\otimes \cR_0(\symb{Y}) \mapsto \cR_{\lambda,a}^{\bG}(\Res_{W_{\cI_0}^e}^{W_{\cI_0}}(\rho_{\osymb{X}}) \boxtimes \rho_{\osymb{Y}}).
\end{equation*}
As $\symb{X}\otimes \symb{Y} = \cQ_e\cR_e(\symb{X})\otimes \cQ_0\cR_0(\symb{Y})$ we see that if $\chi$ is the image of $\symb{X}\otimes \symb{Y}$ then $f_{\chi} = \phi_{\symb{X}} \boxtimes \phi_{\symb{Y}}$.

By (iii) of \cref{lem:cents-pairwise-distinct}, we have $\bC_H(Fx) = \bC_W(Fx)$ and $\abs{\bC_H(Fx)}/\abs{\bC_{W(\lambda)}(Fx)} \leqslant \abs{H/W(\lambda)} \leqslant 2$.
Appealing to \cref{thm:char-bound-hyperoctahedral}, when $X$ is degenerate, we find, as above, that $\abs{f_{\chi}(Fx)} \leqslant 2^{2k-1}\cdot k!$. As we have $\abs{\bC_{W}(Fw)\backslash W/W(\lambda)} \leqslant 2\abs{\bC_{W}(Fw)\backslash W/H}$, there are at most $2^{k+1}$ terms in the above sum. Putting things together gives the bound $2^{3k+1}\cdot k!$ in this case.

\underline{Type $\sD_n$}. We have $W(\lambda) = W_{\mathcal{I}_0}^0W_{\mathcal{I}_1}^0$. If $\pi_0 = (\bar{e}_1,-\bar{e}_1)$ and $\pi_1 = (\bar{e}_{m+1},-\bar{e}_{m+1})$ then the coset $Fa \subseteq H$ is either: $W(\lambda)$, $\pi_0W(\lambda) = W_{\cI_0}^1W_{\cI_1}^0$, $\pi_1W(\lambda) = W_{\cI_0}^0W_{\cI_1}^1$, or $\pi_0\pi_1 W(\lambda)$. These cases are similar to the above. We have $\bC_H(Fx) = \bC_W(Fx)$ and $\abs{\bC_H(Fx)}/\abs{\bC_{W(\lambda)}(Fx)} \leqslant \abs{H/W(\lambda)} \leqslant 4$. Arguing similarly, we get that the above sum has at most $2^{k+2}$ terms, and $\abs{f_{\chi}(Fx)} \leqslant 2^{2k}\cdot k!$. Putting things together gives the bound $2^{3k+4}\cdot k!$.

We end with a comment about the final coset $\pi_0\pi_1 W(\lambda)$. In this case we have an isomorphism $\mathscr{U}_m^{\ev,1}\otimes_{\CC} \mathscr{U}_{n-m}^{\ev,1} \to \Class(\bG^F,\lambda,a)$. Lusztig's Theorem in this case says that this isomorphism can be chosen such that for any $\osymb{X} \in \tilde{\cS}_m^0$ and $\osymb{Y} \in \tilde{\cS}_{n-m}^0$ we have
\begin{equation*}
\cR_1(\symb{X})\otimes \cR_1(\symb{Y}) \mapsto \cR_{\lambda,a}^{\bG}(\Res_{\pi_1\pi_2 W(\lambda)}^{W_{\cI_0}W_{\cI_1}}(\rho_{\osymb{X}} \boxtimes \rho_{\osymb{Y}}));
\end{equation*}
see the discussion in \cite[\S 4.21]{Lus}. One readily checks that if $\chi$ is the image of $\symb{X} \otimes \symb{Y}$ then $f_{\chi} = \Res_{\pi_1\pi_2 W(\lambda)}^{W_{\cI_0}W_{\cI_1}}(\phi_{\symb{X}} \boxtimes \phi_{\symb{Y}})$.
\end{proof}

\begin{cor}\label{bcd-bound}
Assume $\bG$ is a symplectic or special orthogonal group and $\chi \in \Irr(\bG^F)$ is a quadratic unipotent character. Furthermore, let $wF \in WF$ have cycle length $k \geqslant 1$ and pairwise distinct cycles. Then $\abs{\chi(g)} \leqslant 2^{3k+4}\cdot k!$ for any regular semisimple element $g \in \bG^F$ of type $wF$.
\end{cor}

\begin{proof}
Recall that being quadratic unipotent means that $\chi$ lies in a series $\mathcal{E}(\bG^F,\lambda,a)$ with $2\lambda \in X$. The statement is thus an immediate consequence of \cref{quadratic} and \cref{lem:res-series,lem:ss-reg-embed}.
\end{proof}

\section{Character degrees}

In this section, we prove the following theorem.

\begin{thm}
\label{Degrees}
There exists an absolute constant $C>0$ such that for every finite quasisimple group $G$ of Lie type of rank $r$ and every positive integer $D$, the number of 
irreducible characters of $G$ of degree $\le D$ is at most $D^{C/r}$.
\end{thm}

Taking $C$ large enough, we can ignore any finite number of quasisimple groups $G$, and thus we may assume that
$G = \bG^F$ for a simple, simply connected algebraic group $\bG$ of rank $r$ and a Frobenius endomorphism
$F: \bG \to \bG$.  The Landazuri-Seitz bound \cite{LaSe} implies that the minimal non-trivial character of $G$ has degree at least $\abs{G}^{\epsilon}$,
where $\epsilon$ depends only on the rank of $G$.  Therefore, we are justified in assuming that $r$ is as large as we wish, so, in particular, $G$ is of classical type
and $F$ is an endomorphism of Steinberg type.

Our proof closely follows the character degree estimates of Liebeck and Shalev \cite{LiSh}.   Liebeck and Shalev prove a more precise result \cite[Theorem~1.1]{LiSh} than Theorem~\ref{Degrees}
when $q$ is sufficiently large in terms of $r$ and a weaker result \cite[Theorem~1.2]{LiSh} for general $q$.  

What is needed to obtain good bounds in high rank for small $q$ is an estimate for the number
of unipotent characters of $G$ and certain related groups of bounded degree.  This follows in principle from the degree formulas for
unipotent characters of classical groups in \cite{Lus}.  We begin with these computations.

\begin{prop}
\label{Unipotent Degrees}
There exists an absolute constant $C'$ such that for every finite quasisimple group $G$ of classical Lie type of rank $r$ and every positive integer $D$, the number of unipotent characters of $G$ of degree $\le D$ is at most $D^{C'/r}$.
\end{prop}

\begin{proof}
For every prime power $q$, we have by \cite[Lemma 4.1(i), (iii)]{LMT} that
\begin{equation}\label{d10}
  \prod_{n=1}^\infty (1-q^{-1}) > \frac 14\ge q^{-2}
\end{equation}  
and 
\begin{equation}\label{d11}
  \prod_{n=1}^\infty (1+q^{-1}) < 2.4 < q^2.
\end{equation}  

\smallskip
(i) If $G$ is of type $A_r$ or $^2A_r$, then the unipotent characters of $G$ are indexed by sets $A$
of positive integers such that $\rho(A)=r+1$ in the notation of \S2.  Denoting the elements of $A$ by $\lambda_1<\cdots<\lambda_m$, we have
$$r+1 = \sum_{i=1}^{m} (\lambda_i+1-i).$$
The terms $\mu_i:=\lambda_i+1-i$ in this sum give a partition of $r+1$, so in particular, $m\le r+1$.

Here, $G=\SL_n^{\eps}(q)$ with $n = r+1$.
The degree $d_A$ of the character with given set $A$  is the absolute value of 
\begin{equation}
\label{A-dim}
\frac{\prod_{1\le j < i \le m}((\eps q)^{\lambda_i}-(\eps q)^{\lambda_j})\prod_{i=1}^r ((\eps q)^i-1)}
{\prod_{i=1}^{m}\prod_{j=1}^{\lambda_i} ((\eps q)^j-1)\prod_{k=2}^{m-1} q^{\binom k2}}.
\end{equation}
For any fixed $j$, we have from \eqref{d10} that
$$\prod_{i=j+1}^{m} \abs{(\eps q)^{\lambda_i}-(\eps q)^{\lambda_j}} > q^{-2}\cdot q^{\sum_{i=j+1}^{m} \lambda_i},$$
so
$$\prod_{j=1}^{m-1} \prod_{i=j+1}^{m} \abs{(\eps q)^{\lambda_i}-(\eps q)^{\lambda_j}} > q^{2-2m}\cdot q^{\sum_{1\le j < i\le m}\lambda_i}.$$
Treating the other factors of \eqref{A-dim} in the same way (and using \eqref{d11} for the products in the denominator), we obtain
$$d_A > q^{-4m}\cdot q^{\sum_{1\le j < i\le m}\lambda_i + \binom{r+1}2 -\sum_{i=1}^{m} \binom{\lambda_i+1}2 - \sum_{k=2}^{m-1} \binom k2}.$$
As is well-known (see e.g. the proof of \cite[Lemma 5.3]{GLT}, the exponent of the second factor on the right-hand side is 
$$\frac 12 (n^2-\sum^m_{i=1}\mu_i^2) \geq \frac 12(n^2-\mu_m\sum^m_{i=1}\mu_i) = n(n-\mu_m)/2,$$
and so
$$d_A \geq q^{n(n-\mu_m)/2-4m} \geq q^{n(n-\mu_m)/2-4n}$$

Therefore, 
$$n - \mu_m \le  \Bigl(2\frac{\log d_A}{\log q^n}+8\Bigr).$$
As $\mu_m=\lambda_{m} + 1-m$ is the largest part of the partition of $n$ associated to $A$, the number of possibilities for $A$ such that $d_A \le D$ is at most
$$\sum_{i=1}^{\lfloor 2\frac{\log D}{\log q^n}+8\rfloor} p(i),$$
where $p$ denotes the partition function.  As $p(i)$  is sub-exponential in $i$, when $D \geq q^{n/3}$ this number is 
$e^{O(\log D/\log q^n)} = D^{O(1)/n}$,
yielding a uniform upper bound of the form $D^{C'/r}$  for the number of unipotent characters of degree $\leq D$.
However, by \cite{LaSe}, for $D \le q^{n/3}$, there are no non-trivial irreducible characters of degree
$\leq D$, and in particular no such unipotent characters.

\smallskip
(ii) The proof for the remaining classical groups follows the same pattern.
Let $[X]$ denote an equivalence class of ordered symbols, and let $X=(X^0,X^1)$ be the representative such that $0\not\in X^\cap = X^0\cap X^1$.
Let the elements of $X^0$ and $X^1$ respectively form
the increasing sequences of non-negative integers $\lambda^0_1<\cdots<\lambda^0_{\abs{X^0}}$ and $\lambda^1_1<\cdots<\lambda^1_{\abs{X^1}}$, so assuming $X^0$ and $X^1$ are both non-empty, we have
$\lambda^0_1+\lambda^1_1>0$.  We note that $X^0$ and $X^1$ determine (possibly improper) partitions $\{\lambda^0_i+1-i\}_i$ and $\{\lambda^1_j+1-j\}_i$ of the ranks $\rho(X^0)$ and $\rho(X^1)$ respectively. More precisely, if $\lambda^j_1 > 0$ then  $\{\lambda^j_i+1-i\}_i$ is a partition of $\rho(X^j)$, but
if $\lambda^j_1=0$ then $\{\lambda^j_i+1-i\}_i$ is a sequence of initial zeroes concatenated with a partition of $\rho(X^j)$.
Let $n := \abs{X^0}+\abs{X^1}$, and let $\nu_1\le \nu_2\le \cdots\le\nu_n$ be the sequence obtained by first merging $X^0$ and $X^1$ and then sorting, without eliminating repetitions.
Thus 
\begin{equation}
\label{odd conditions}
\nu_1 < \nu_3 < \nu_5 <\cdots,
\end{equation}
and 
\begin{equation}
\label{even conditions}
0 < \nu_2 < \nu_4 < \cdots.
\end{equation}

By \eqref{eq:rank-definition}, the rank $r$ of the symbol $X$ is given by
\begin{equation}\label{d12}
r = \sum_i \lambda^0_i + \sum_j \lambda^1_j - \Bigl\lfloor\frac{(\abs{X^0}+\abs{X^1}-1)^2}4\Bigr\rfloor = \sum_k \nu_k - \Bigl\lfloor\frac{(n-1)^2}4\Bigr\rfloor
= \sum_{k=1}^n (\nu_k - \lfloor(k-1)/2\rfloor).
\end{equation}
We have $\nu_k \ge \lfloor(k-1)/2\rfloor$ for all $k$, with strict equality when $k$ is even, so $r \ge \lfloor n/2\rfloor$. 
Thus,
\begin{equation}\label{d12b}
  \begin{aligned}r & = \rho(X^0)+\binom{\abs{X^0}}2 + \rho(X^1)+\binom{\abs{X^1}}2 - \Bigl\lfloor\frac{(\abs{X^0}+\abs{X^1}-1)^2}4\Bigr\rfloor\\
  & = \rho(X^0)+\rho(X^1) + \left\{ \begin{array}{lr} \bigl(\Def(X)^2-1\bigr)/4, & 2 \nmid n\\ \Def(X)^2/4, & 2|n \end{array} \right.\\
  &  \geq \rho(X^0)+\rho(X^1).\end{aligned}
\end{equation}  

For every prime power $q$ and every  symbol $X$, there is at least one associated unipotent character of at least one classical group $G$ of rank $r$ over the field $\FF_q$.  If $\Def(X)$ is odd, we obtain characters of $G=\Sp_{2r}(\FF_q)$ and of $G=\Spin_{2r+1}(\FF_q)$ in this way.  If it is divisible by $4$, there is a character of $G=\Spin^+_{2r}(\FF_q)$; otherwise, there is a character of $G=\Spin^-_{2r}(\FF_q)$.   If $X^0=X^1$, then there are two unipotent characters for $\Spin^+_{2r}(\FF_q)$ associated to $X$; otherwise, there is only one  for each possible $G$.  Moreover, all unipotent characters for groups of type $B$, $C$, and $D$ arise in this way for a unique equivalence class of unordered symbols.

The degree $d_X$ of the unipotent character of $G$ associated to the symbol $X$ is (at least)
$$\frac{\abs{G}_{p'}
\prod_{1\le j < i\le n^0} (q^{\lambda^0_i}-q^{\lambda^0_j})
\prod_{1\le j < i\le n^1} (q^{\lambda^1_i}-q^{\lambda^1_j})
\prod_{i=1}^{n^0}\prod_{j=1}^{n^1} (q^{\lambda^0_i}+q^{\lambda^1_j})}
{2^{\lfloor n/2\rfloor}
\prod_{i=1}^{n^0} \prod_{j=1}^{\lambda^0_i} (q^{2j}-1)
\prod_{i=1}^{n^1} \prod_{j=1}^{\lambda^1_j} (q^{2j}-1)
\prod_{k=1}^{\lfloor (n^0+n^1-2)/2\rfloor}q^{\binom{n^0+n^1-2k}2}
}$$
with $n^j := \abs{X^j}$.
In terms of the sequence $\nu_i$, this takes the form
$$d_X\geq \frac{\abs{G}_{p'}
\prod_{1\le j < i\le n} (q^{\nu_i}\pm q^{\nu_j})}
{2^{\lfloor n/2\rfloor}
\prod_{i=1}^n \prod_{j=1}^{\nu_i} (q^{2j}-1)
\prod_{k=1}^{\lfloor (n-2)/2\rfloor}q^{\binom{n-2k}2}
}.$$
Note that
$$\abs{G}_{p'} = \begin{cases}
\prod_{i=1}^r (q^{2i}-1)&\text{if $\Def(X)\equiv 1 \pmod2$},\\
(q^r-1)\prod_{i=1}^{r-1} (q^{2i}-1)&\text{if $\Def(X)\equiv 0 \pmod4$},\\
(q^r+1)\prod_{i=1}^{r-1} (q^{2i}-1)&\text{if $\Def(X)\equiv 2 \pmod4$},\
\end{cases}$$
so $\abs{G}_{p'} \ge q^{-2}\cdot q^{r^2+r}$ if $n$ is odd and $\abs{G}_{p'}\ge q^{-2} q^{r^2}$ if $n$ is even, again by \eqref{d10}.

Reasoning as in case $A$, for cases $B$ and $C$ (i.e., $n=2m+1$ is odd) and using the fact that $n\le 2r+1$, we have
\begin{align*}
d_X&\ge q^{-2+2-2n-n/2}\cdot q^{r^2+r+ \sum_{1\le j<i\le n} \nu_i - \sum_{i=1}^n \nu_i(\nu_i+1) - \sum_{k=1}^{\frac{n-1}2} \binom{n-2k}2} \\
&\geq q^{-15r/2}\cdot q^{r^2+r+ \sum_{1\le j<i\le n} \nu_i - \sum_{i=1}^n \nu_i(\nu_i+1) - \sum_{k=1}^{\frac{n-1}2} \binom{n-2k}2}.
\end{align*}
We define
$$Y:=\sum^n_{i=1}\Bigl( (i-1)-2\lfloor \frac{i-1}{2} \rfloor \Bigr)\nu_i = \nu_2 + \nu_4 + \ldots + \nu_{2m} \geq \sum^m_{i=1}i = \frac{m(m+1)}2,$$ 
By \eqref{d12}, the exponent of the second term on the right-hand side is
\begin{align*}
&r^2+\sum_{i=1}^n \nu_i-m^2+
\sum_{i=1}^n (i-1)\nu_i - \sum_{i=1}^n \nu_i(\nu_i+1) -\sum_{k=1}^{m} \binom{n-2k}2 \\
& = r^2-m^2+
\sum_{i=1}^n (i-1)\nu_i - \sum_{i=1}^n \nu_i^2 -\sum_{k=1}^{m} \binom{n-2k}2 \\
&=r^2-m^2 +Y - \sum_{i=1}^n \Bigl(\nu_i-\bigl{\lfloor} \frac{i-1}2 \bigr{\rfloor} \Bigr)^2 + \sum_{i=1}^n \Bigl(\bigl{\lfloor} \frac{i-1}2 \bigr{\rfloor}\Bigr)^2 -\sum_{k=1}^{m} \binom{n-2k}2 \\
&\geq  r^2 -  \sum_{i=1}^n \Bigl(\nu_i-\bigl{\lfloor} \frac{i-1}2 \bigr{\rfloor} \Bigr)^2 + m\\
&\geq  r^2 -  \sum_{i=1}^n \Bigl(\nu_i-\bigl{\lfloor} \frac{i-1}2 \bigr{\rfloor} \Bigr) \cdot \max_i \Bigl(\nu_i-\bigl{\lfloor} \frac{i-1}2 \bigr{\rfloor} \Bigr)\\
& = r^2-r\cdot \max_i \Bigl(\nu_i-\bigl{\lfloor} \frac{i-1}2 \bigr{\rfloor} \Bigr).
\end{align*}
Here we have used the identity
$$\sum^{2m+1}_{i=1}\Bigl(\bigl{\lfloor} \frac{i-1}2 \bigr{\rfloor}\Bigr)^2 -\sum_{k=1}^{m} \binom{2m+1-2k}2 = 2\sum^{m-1}_{i=1}i^2+m^2 - \sum^{m-1}_{i=1}i(2i+1)= 
    \frac{m(m+1)}{2}.$$
Thus,
\begin{equation}\label{d13a}
  \max_i \Bigl(\nu_i - \bigl{\lfloor} \frac{i-1}2 \bigr{\rfloor}\Bigr) \ge r - \frac{\log d_X}{\log q^r} - 7.5.
\end{equation}  

Since $\abs{\Irr(G)} \leq q^{C_1r}$ for some absolute constant $C_1$ by \cite{FG}, by enlarging $C'$ (which then covers all small ranks), we
may assume $\log d_X/\log q^r \leq r/2-8.5$, and so 
\begin{equation}\label{d13}
   \max_i \Bigl(\nu_i - \bigl{\lfloor} \frac{i-1}2 \bigr{\rfloor}\Bigr) \ge \frac{r}{2}+1.
\end{equation}
By \eqref{odd conditions} and \eqref{even conditions} and the integrality of the $\nu_i$, we have
$$\nu_{i+2} - \bigl{\lfloor} \frac{(i+2)-1}2 \bigr{\rfloor} \geq \nu_i - \bigl{\lfloor} \frac{i-1}2 \bigr{\rfloor},$$ 
and so 
$$\max_i \Bigl(\nu_i- \lfloor \frac{i-1}2\rfloor \Bigl) = \max\Bigl(\nu_{n-1}-\lfloor \frac{n-2}{2} \rfloor,\nu_n-\lfloor \frac{n-1}{2} \rfloor\Bigr).$$
Now, if $\max_i\bigl(\nu_i- \bigl{\lfloor} \frac{i-1}2\bigr{\rfloor}\bigr)$ is attained at $i=n-1$, then 
$$\nu_{n-1} - \bigl\lfloor \frac{n-2}2 \bigr\rfloor- 1 \geq r/2+1$$ and 
$$\nu_{n} - \bigl\lfloor \frac{n-1}2 \bigr\rfloor \geq \nu_{n-1} -\bigl\lfloor \frac{n-2}2 \bigr\rfloor- 1 \geq r/2$$ by \eqref{d13}, and this violates \eqref{d12}.
Hence, $\nu_i-\lfloor \frac{i-1}2\rfloor$ achieves its maximum at only $i=n$, and $\nu_n>\nu_{n-1}$, again by \eqref{d13};
also, 
$$r \geq \nu_n-\bigl\lfloor \frac{n-1}2\big\rfloor \geq r-\frac{\log d_X}{\log q^r} - 7.5.$$

Applying the Landazuri-Seitz bound as before, we may assume that $d_X \geq q^r$. 
The rank $r'$ of the symbol, obtained from $X$ by deleting the largest single term $\nu_n$, call it $X'$, 
is bounded above by $\log d_X/\log q^r+7.5$, by \eqref{d12} and \eqref{d13a}. Applying \eqref{d12b} to
$X' = ((X')^0,(X')^1)$, we see that 
\begin{equation}\label{d14}
  \rho((X')^0)+\rho((X')^1) \leq \log d_X/\log q^r+7.5 \leq x+7.5
\end{equation}  
if $d_X\leq q^{rx}$ with $x \geq 0$. In fact, we can show that such symbols $X$ satisfy
\begin{equation}\label{d15}
  \abs{X^0}+\abs{X^1} < 3x+22.5.
\end{equation}  
Indeed, without loss we may assume that $\mu_1 \geq 1$, so the sequence $\{\mu_j+1-j\}$ of 
$\abs{T'}$ integers is a proper partition of $\rho((X')^1)$, and so $\abs{(X')^1} \leq \rho((X')^1) \leq x+7.5$ by \eqref{d14}. Applying \eqref{d12b} to
$X'$ we have $(\abs{(X')^0}-\abs{(X')^1})^2 \leq 4r'+1$, and so
$$\abs{(X')^0}-\abs{(X')^1} \leq \sqrt{4r'+1} \leq \sqrt{4(x+7.5)+1} < x+6.5.$$
Hence $\abs{(X')^0} < 2x+14$, and $\abs{X^0}+\abs{X^1} = \abs{(X')^0}+\abs{(X')^1}+1 < 3x + 22.5$, as stated. 
 
By \eqref{d15}, even when $\lambda^j_1=0$, the number of zero entries in the sequence $\{\lambda^j_i+1-i\}$
is at most $\abs{(X')^j} < 3x+22.5$. Now, counting the number of (possibly improper)
partitions $\{\lambda^j_i+1-i\}$ of $\rho((X')^j)$ and using \eqref{d14}, we see that
the number of possibilities for the symbol $X$ 
with $d_X < q^{rx}$ is bounded above by 
$$(3x+22.5)\sum^{\lfloor x +7.5\rfloor}_{i=1}p(i)^2,$$
an exponential in $x$ for $x\ge 1$,
proving the proposition for types $B_r$ and $C_r$. 

\smallskip
(iii) For types $D_r$ and $\tw2 D_r$, that is, when $n=2m$ is even, we have $n\le 2r$, and
\begin{align*}
d_X &\ge q^{-2+2-2n-n/2}\cdot q^{r^2+ \sum_{1\le j<i\le n} \nu_i - \sum_{i=1}^n \nu_i(\nu_i+1) - \sum_{k=1}^{\frac{n-2}2} 
\binom{(n-2k}{2}} \\
&\geq q^{-5r}\cdot q^{r^2+\sum_{1\le j<i\le n} \nu_i - \sum_{i=1}^n \nu_i(\nu_i+1) - \sum_{k=1}^{\frac{n-2}2} \binom{n-2k}{2}}.
\end{align*}
The exponent of the second term on the right hand side is therefore
\begin{align*}
&r^2+
\sum_{i=1}^n (i-1)\nu_i - \sum_{i=1}^n \nu_i(\nu_i+1) -\sum_{k=1}^{m-1} \binom{n-2k}2 \\
&=r^2+
\sum_{i=1}^n (i-2)\nu_i - \sum_{i=1}^n \nu_i^2 -\sum_{k=1}^{m-1} \binom{n-2k}2 \\
&=r^2-Y - \sum_{i=1}^n \Bigl(\nu_i-\bigl{\lfloor} \frac{i-1}2 \bigr{\rfloor} \Bigr)^2 + \sum_{i=1}^n \Bigl(\bigl{\lfloor} \frac{i-1}2 \bigr{\rfloor}\Bigr)^2 -\sum_{k=1}^{m-1} \binom{n-2k}2 \\
&\geq  r^2 -  \sum_{i=1}^n \Bigl(\nu_i-\bigl{\lfloor} \frac{i-1}2 \bigr{\rfloor} \Bigr)^2 - r\\
&\geq  r^2 -  \sum_{i=1}^n \Bigl(\nu_i-\bigl{\lfloor} \frac{i-1}2 \bigr{\rfloor} \Bigr) \cdot \max_i \Bigl(\nu_i-\bigl{\lfloor} \frac{i-1}2 \bigr{\rfloor} \Bigr)-r\\
& = r^2-r\cdot \max_i \Bigl(\nu_i-\bigl{\lfloor} \frac{i-1}2 \bigr{\rfloor} \Bigr)-r.
\end{align*}
Here we have used the inequality 
$$\begin{array}{lll}Y&:=\sum^n_{i=1}\Bigl(2\Bigl{\lfloor} \frac{i-1}{2} \Bigr{\rfloor}-(i-2) \Bigr)\nu_i
  & = \nu_1 + \nu_3 + \ldots + \nu_{2m-1}\\
 & = \sum^n_{i=1,~2 \nmid i}\Bigl(\nu_i - \bigl{\lfloor} \frac{i-1}2 \bigr{\rfloor} \Bigr)+ \binom{m}{2} & \leq r + \binom{m}{2}\end{array}$$ 
and the identity
$$\sum^{2m}_{i=1}\Bigl(\bigl{\lfloor} \frac{i-1}2 \bigr{\rfloor}\Bigr)^2 -\sum_{k=1}^{m-1} \binom{2m-2k}2 = 2\sum^{m-1}_{i=1}i^2- \sum^{m-1}_{i=1}i(2i-1)= 
    \binom{m}{2}.$$
The argument finishes as before.
\end{proof}

\begin{proof}[Proof of Theorem~\ref{Degrees}]
Let $G^* = (\bG^*)^{F^*}$ denote the dual group of $G$.  We partition the irreducible characters of $G$ into rational Lusztig series $\cE(s)$, indexed
by conjugacy classes of semisimple conjugacy classes $(s)$ of $G^*$.  There is a bijection between the elements of $\cE(s)$ and unipotent characters of $\bC_{G^*}(s)$; this correspondence multiplies degrees by $\abs{G^*}_{p'}/\abs{\bC_{G^*}(s)}_{p'}$.

We restate \cite[Lemma~3.2]{LiSh} in a form more convenient for our purposes.  Since the ratio $n/r$ between the dimension $n$ of the natural module and the rank of $G$ is bounded
between $1$ and $3$, and since the constants $d$ and $d'$ in \cite{LiSh} are absolute, if $\delta$ is greater than some absolute constant
$\delta_0$, then the number of semisimple conjugacy classes $(s)$ with $\abs{G^*}_{p'}/\abs{\bC_{G^*}(s)}_{p'} \le q^{\delta r}$ is less than $q^{\delta A}$ for some absolute constant $A$.  Moreover, 
$\bC_{G^*}(s)$ contains a factor, the \emph{large factor}, which is classical of rank $r' \ge r-B\delta$ for some absolute constant $B$, and this large factor is 
$A_{r'}(q)$ or $\tw2 A_{r'}(q)$ when $\bG$ is of type $A$.

For $D <q^{r/3}$, there is only one irreducible character, by \cite{LaSe}. Hence we may assume $D \geq q^{r/3}$. 
Enlarging $C$ if necessary, we may assume that $D = q^{\delta r}$ with $\delta \geq \max(\delta_0,1/2)$.  By \cite{FG}, $\abs{\Irr(G)} \leq q^{C_1r}$
for some absolute constant $C_1$, so the result follows if $\delta C \geq C_1r$. Again enlarging $C$ if necessary, 
we may assume without loss of generality that $\delta < r/2B$.

If $\chi$ is an irreducible character of degree $\le D$, then it belongs to the Lusztig series $\cE(s)$ for some $s$ with $\abs{G}_{p'}/\abs{\bC_{G^*}(s)}_{p'}\le D$.
The number of such semisimple classes $s$ is bounded above by $q^{\delta A}$.  
Following the proof of \cite[Lemma 3.4]{LiSh}, note that 
for each $s$,
$\bC_{G^*}(s)$ contains a subgroup $\bC_{G^*}(s)^\circ$ which is the group of $F^*$-fixed points of the connected reductive 
algebraic group $\bC_{\bG^*}(s)^\circ$.
If $\bG$ is of type $B$, $C$, or $D$, the quotient group $\bC_{G^*}(s)/\bC_{G^*}(s)^\circ$ has order $\le 4$.
Suppose $\bG$ is of type $A$ and $\bC_{G^*}(s)^\circ$ is a proper subgroup of $\bC_{G^*}(s)$. Lifting $s$ to an element 
$\hat{s}$ of $\GL_n^{\eps}(q)$, we see that every eigenvalue of $\hat{s}$ has multiplicity $\le n/2$, but this 
contradicts the existence of the large factor of $\bC_{G^*}(s)$ which is of type $A_{r'}(q)$ or $\tw2 A_{r'}(q)$ with $r' > r/2$.
So we have $\bC_{G^*}(s) = \bC_{G^*}(s)^\circ$ for type $A$.
Thus, there are at most $4$ unipotent characters of $\bC_{G^*}(s)$ of degree $\le D$ for each unipotent character of $\bC_{G^*}(s)^\circ$ of degree $\le D$.
Taking $F^*$-fixed points of the derived group of $\bC_{\bG^*}(s)^\circ$ we obtain a subgroup whose unipotent characters correspond to those of $\bC_{G^*}(s)$,
and this subgroup is a product of classical groups whose ranks sum up to at most $r$ and at least one of which, the large factor,  
has rank $r'$ at least $r-B\delta \ge r/2$.
The total number of unipotent characters of the product of all the factors other than the large factor can therefore be bounded above by $q^{C_1B\delta}$ by \cite{FG}.
The number of unipotent characters of the large factor of degree $\le D=q^{\delta r}$ is bounded above by 
$D^{C'/r'} \leq q^{2\delta C'}$, by Proposition~\ref{Unipotent Degrees}.
Hence, the number of unipotent characters of degree $\le D$ of $\bC_{G^*}(s)^\circ$ is bounded by 
$q^{(BC_1+2C')\delta}$, and the number for $\bC_{G^*}(s)$ is likewise
bounded by an exponential in $\delta$. Thus the number of characters of degree at most $D$ is bounded by 
$q^{C_2\delta}$ for some absolute constant $C_2$, and the theorem follows by taking $C \geq C_2$.
\end{proof}

For later use, we prove the following related statement:

\begin{prop}\label{j-red}
Let $n \geq 4$ and $j \geq 1$ be integers and let $q$ be any prime power. Let 
$G = \bG^F$ be one of the groups $\Sp_{2n}(q)$, $\SO_{2n+1}(q)$, or $\SO^\eps_{2n}(q)$, and suppose that $\chi \in \Irr(G)$ has 
degree 
$$\chi(1) \leq \min\bigl(q^{nj},q^{(n^2-n)/2-4}\bigr).$$ 
Then there is an $F$-stable Levi subgroup $\bL= \bL_1 \times \bL_2$ of $\bG$, possibly equal to
$\bG$, such that the following statements hold:
\begin{enumerate}
\item[\rm(i)] $\bL_1$ is $F$-stable, of the same type as $\bG$, and of rank $n-m \geq n-2j$.
\item[\rm(ii)] $\bL_2$ is $F$-stable, of type $\GL_m$ with $m \leq 2j$.
\item[\rm(iii)] $\chi = \pm R^\bG_\bL(\varphi_1 \boxtimes \varphi_2)$, where $\varphi_1 \in \Irr(\bL_1^F)$ is a unipotent or quadratic unipotent character,
and $\varphi_2 \in \Irr(\bL_2^F)$. Moreover, if $j \leq n/2$, then $\varphi_1(1) \leq q^{(n-m)j}$.
\end{enumerate} 
\end{prop}

\begin{proof}
View $G = \Sp(V)$ or $\SO(V)$, where $V = \FF_q^{2n}$, $\FF_q^{2n+1}$, or $\FF_q^{2n}$. Then we can identify the dual group
$G^*= (\bG^*)^{F^*}$ with $\SO(V^*) \cong \SO_{2n+1}(q)$, $\Sp(V^*) \cong \Sp_{2n}(q)$, or $\SO(V^*) \cong \SO(V) \cong \SO_{2n}(q)$, 
respectively. Let $\cE(G,(s))$ be the rational Lusztig series that contains $\chi$, where $s \in G^*$ is semisimple. If 
$s^2=\mathrm{id}_{V^*}$, then we are done by choosing $\bL=\bG$. Otherwise we can decompose 
$V^* = V^*_1 \oplus V^*_2$, where $s^2-\mathrm{id}_{V^*}$ is zero on $V_1^*$, and invertible on its orthogonal complement
$V_2^* \neq 0$. Since $s$ is semisimple, we see that $\bC_{\bG^*}(s)$ is contained in a proper $F^*$-stable Levi subgroup
$\bL^* = \bL^*_1 \times \bL_2^*$ of $\bG^*$, where $\bL_2^* \cong \GL_m$, with $m := \dim(V^*_2)/2 \leq n$, 
and $\bL^*_1$, of the same type 
as of $\bG^*$, are both $F^*$-stable. Let $\bL = \bL_1 \times \bL_2$ denote the Levi subgroup of $\bG$ dual to $\bL^*$, where 
$\bL_1$ is $F$-stable and of the same type as of $\bG$, and $\bL_2 \cong \GL_m$. 
By \cite[Thm. 11.4.3]{DM}, $\varepsilon_\bG\varepsilon_\bL R^\bG_\bL$  yields a bijection between 
$\cE(\bL^F,(s))$ and $\cE(G,(s))$, which implies (iii); in particular, 
\begin{equation}\label{j10}
  \chi(1)=\frac{\abs{G}_{p'}}{\abs{\bL^F}_{p'}}\varphi_1(1)\varphi_2(1),
\end{equation}   
if $p$ denotes the unique prime divisor of $q$. 

Using \eqref{d10} and \eqref{d11}, one readily checks that 
$$\frac{\abs{G}_{p'}}{\abs{\bL^F}_{p'}} > \left\{ \begin{array}{ll}q^{(4nm-3m^2-m-8)/2}, & \bG = \SO_{2n},\\
    q^{(4nm-3m^2+m-8)/2}, & \bG = \Sp_{2n},~\SO_{2n+1}.\end{array}\right.$$ 
In particular, if $m=n$ then $\chi(1) > q^{(n^2-n)/2-4}$ by \eqref{j10}, a contradiction. 
Assume now that $m \leq n-1$, but $m \geq 2j+1$. Then $4n-3m-1 \geq n+2$, whence
$$4nm-3m^2-m-8=m(4n-3m-1)-8 \geq (2j+1)(n+2)-8 = 2nj+n+2(2j+1)-8 \geq 2nj+1,$$
and so $\chi(1) > q^{nj}$, again a contradiction. Thus $m \leq 2j$. 

To show $\varphi_1(1) \leq q^{(n-m)j}$, it suffices by \eqref{j10} to check that $\abs{G}_{p'}/\abs{\bL^F}_{p'} \geq q^{mj}$. This is obvious if
$m=0$. If $1 \leq m \leq 3$, then $j \leq n/2 \leq 2n-6$ implies that $4nm-3m^2-m-8 \leq 2mj$. If $m \geq 4$, then 
$$4nm-3m^2-m-8=m(4n-3m-1)-8 \geq m(n+2)-8 \geq mn \geq 2mj,$$
and so we are done.
\end{proof}

\section{Applications to asymptotic variants of Thompson's conjecture}
\subsection{Type $A$}
Recall $\SL^\eps_n(q)$ denotes $\SL_n(q)$ when $\eps=+$, and $\SU_n(\FF_{q^2})$ when $\eps=-$, and 
similarly for $\GL^\eps_n(q)$. 

\begin{thm}
\label{almost-thompson-SLU}
For all $k \in \ZZ_{\geq 1}$, there exists an explicit constant $B=B(k) >0$ such that the following statement holds for all $n \in \ZZ_{\geq 1}$ and all prime powers 
$q$. Suppose $G = \SL^\eps_n(q)$ for some $\eps=\pm$ and $g \in G$ 
is a regular semisimple element whose characteristic polynomial on the natural module of $G$ is a 
product of $k$ pairwise distinct irreducible polynomials.
Then $g^{G}\cdot g^{G}$ contains every element $x \in G$ of support $\ge B$.  
\end{thm}

\begin{proof}
(a) Embed $G$ in $\tilde G:=\GL^\eps_n(q)$. Since the support of an element of $\tilde G$ is at most $n$, 
by enlarging $B$, we are free to make $n \geq k$ as large as we wish. 

Note that the element $g$ is regular semisimple, and $T:=\bC_{\tilde G}(g)$ is a maximal torus, so of order at most $(q+1)^n$. Moreover, the image of $T$ under the determinant map is 
the same as that of $\tilde G$. Hence the conjugacy class of $g$ in $G$ is the same as its class in $\tilde G$.  Let $x\in G$.  To show that $x\in g^{G}\cdot g^{G}$, it suffices to prove that
$$\sum_{\chi \in \Irr(\tilde G)} \frac{\chi(g)^2\bar\chi(x)}{\chi(1)} \neq 0.$$
As $\det(g)=\det(x)=1$, for every character $\chi$ of degree $1$ we have $\chi(g) = \chi(g)^2\bar\chi(x) = 1$. 
Therefore, it suffices to prove that
\begin{equation}
\label{abs bound}
\sum_{\{\chi \in \Irr(\tilde G) \mid \chi(1)>1\}} \frac{\abs{\chi(g)}^2\abs{\chi(x)}}{\chi(1)} < q-\eps.
\end{equation}

\smallskip
(b) For any fixed $\epsilon > 0$, choosing $B$ sufficiently large, the contribution of characters $\chi \in \Irr(\tilde G)$ 
satisfying $\chi(1)\ge q^{\epsilon n^2}$ to \eqref{abs bound} is $o(1)$.  Indeed, consider any such character $\chi$ and any irreducible constituent $\psi$ of $\chi|_G$. 
Since $\tilde G/G \cong C_{q-\eps}$, by Clifford's theorem we have $\chi|_G = \psi_1 + \ldots +\psi_t$, where
$\psi_1=\psi, \ldots,\psi_t$ are distinct $\tilde G$-conjugates of $\psi$, and $t|(q-\eps)$. By \cite[Theorem 5.5]{LT}, 
$$\abs{\psi_i(x)} \leq \psi_i(1)^{1-\sigma B/n} = (\chi(1)/t)^{1-\sigma B/n}$$
for some absolute constant $\sigma > 0$, and so $\abs{\chi(x)} \leq t(\chi(1)/t)^{1-\sigma B/n}$. As $\chi(1) \geq (q+1)^2 \geq t^2$, we obtain 
$$\abs{\chi(x)/\chi(1)} \leq \chi(1)^{-\sigma B/2n} \leq q^{-\eps\sigma Bn/2}.$$
Since $\abs{T} \leq (q+1)^n < q^{2n}$, it follows that the contribution of all these characters to \eqref{abs bound} is at most
$$q^{-\eps\sigma Bn/2}\sum_{\chi}\abs{\chi(g)}^2 \leq q^{-\eps\sigma Bn/2}\abs{T} <  q^{2n(1-\eps\sigma B/4)}$$
which is $o(1)$ when $B$ is large enough.

\smallskip
(c) Now let $j$ denote the {\it level} of $\chi \in \Irr(\tilde G)$, as defined in \cite{GLT}. Assuming $\chi(1) > 1$, we have $j > 0$. 
If $j \geq n/2$, then $\chi(1) \ge q^{n^2/4-2}$ by \cite[Theorem 1.2(ii)]{GLT}, and, as shown in (b), the contribution of all such 
characters to the left hand side of \eqref{abs bound} is $o(1)$. Hence it remains to consider the characters $\chi$ with 
$$j < n/2;$$
any such character is irreducible over $G$, see \cite[Corollary 8.6]{GLT}. Up to a linear factor, we may assume that 
$\chi$ has {\it true level} $j$. By \cite[Theorem 3.9]{GLT}, any such character $\chi$ is of the form
$$\chi = \pm R^{\tilde G}_L(\varphi \boxtimes \psi),$$
where $L = \GL^\eps_{n-m}(q) \times \GL^\eps_m(q)$ is a (possibly non-proper) Levi subgroup of $\tilde G$ with $0 \leq m < n$, 
$\varphi=\varphi^{\lambda}$ is the unipotent character of $\GL^\eps_{n-m}(q)$ labeled by a partition $\lambda \vdash (n-m)$ with 
largest part $\lambda_1=n-j$, so, in particular, 
\begin{equation}\label{bd20}
  m \leq j,
\end{equation} 
and $\psi \in \Irr(\GL^\eps_m(q))$ when $m > 0$. Moreover, the total number of characters of $\tilde G$ of true level $j$ is 
$\abs{\Irr(\GL^\eps_j(q))}$, which is  shown in \cite[Propositions 3.5, 3.9]{FG} to be at most $9q^j$. 

\smallskip 
Since $\chi$ has level $j$, 
$\chi(1) \geq q^{j(n-j)-1} > q^{nj/3}$ by \cite[Theorem 1.2(i)]{GLT}.
For these characters $\chi$, $\supp(x) \ge B$ implies by \cite[Theorem 5.5]{LT} that
\begin{equation}\label{bd21}
  \abs{\chi(x)}/\chi(1) < q^{-\sigma Bj/3}.
\end{equation}  

\smallskip
As $g$ is regular semisimple, the Steinberg character
${\mathsf{St}}_{G}$ of $G$ takes value $\pm 1$ at $x$. Applying \cite[Cor. 10.2.10]{DM} we have 
\begin{equation}\label{bd2}
  \chi(g) = \pm ({\mathsf{St}}_{G}\cdot\chi)(g) = \pm \mathrm{Ind}^G_L({\mathsf{St}}_L\cdot \varphi)(g).
\end{equation}
Note that if $V=\FF_q^n$ denotes the natural module of $\tilde G$ (endowed
with a Hermitian form when $\eps=-$), then the $L$-module $V$ is a direct (orthogonal when $\eps=-$) sum of two 
non-isomorphic irreducible modules $V_1:=\FF_q^{n-m}$ and $V_2:=\FF_q^m$, with $m \leq j < n/2$, see \eqref{bd20}. 
In particular, if $y \in \bN_G(L)$, then $y$ preserves each of $V_1$ and $V_2$, and thus 
$\bN_G(L)=L$.

Now we count the number $N$ of elements 
$y \in G$ such that $y^{-1}gy \in L$, i.e. $g \in yLy^{-1}$.  Then $g$ acts on each of the subspaces
$yV_1$ and $yV_2$. On the other hand, the decomposition $p_V(g) = \prod^k_{i=1}f_i(X)$ leads to a decomposition
$V = \oplus^k_{i=1}U_i$, where $p_{U_i}(g)=f_i(X)$, and each $U_i$ is a minimal $\langle g \rangle$-invariant, non-degenerate if $\eps=-$,
subspace. Moreover, the $\langle g \rangle$-modules $U_i$ are pairwise non-isomorphic. Hence $(yV_1,yV_2)$ is uniquely
determined by choosing a subset of $\{U_1, \ldots,U_k\}$ (so that $yV_1$ is the sum over this subset and $yV_2$ is the sum over the 
complement). Thus the total number of possibilities for $yLy^{-1} = \bN_{G}(yV_1,yV_2)$ is at most $2^k$. 
On the other hand, $yLy^{-1} = y'Ly'^{-1}$ if and only if $y^{-1}y' \in \bN_G(L)=L$. Hence
\begin{equation}\label{bd22}
  N \leq 2^k\abs{L}.
\end{equation}    

Suppose $y^{-1}gy = \diag(g_1,g_2) \in L$, with $g_1 \in L_1:=\GL^\eps_{n-m}(q)$ and $g_2 \in L_2:=\GL^\eps_m(q)$. Let $k_i$ denote
the number of irreducible factors of the characteristic polynomial of $g_i$ on the natural module for $L_i$.
Then $k_1 + k_2 \leq k$. Since $\varphi$ is unipotent and $g_1$ is regular semisimple,
we have 
$$\abs{\varphi(g_1)} \leq 2^{k_1-1} \cdot k_1!$$
by \cref{unip-bound}. On the other hand, when $m > 0$, \eqref{bd20} and Corollary \ref{glu-bound2} show that
$$\abs{\psi(g_2)} \leq k_2! \cdot j^{k_2}.$$
As $y^{-1}gy$ is regular semisimple in $L$, $\abs{\St_L(y^{-1}gy)}=1$. Hence
$$\bigl{|}\bigl(\St_L \cdot (\varphi \boxtimes \psi)\bigr)(y^{-1}gy)\bigr{|} \leq  2^{k-1} \cdot k! \cdot j^{k-1}.$$ 
It now follows from \eqref{bd2} and \eqref{bd22} that
$$\abs{\chi(g)} = \biggl{|}\frac{1}{\abs{L}}\sum_{y \in G:y^{-1}gy \in L}\pm (\St_L\cdot\varphi)(y^{-1}gy)\biggr{|}\leq 2^{2k-1} \cdot k! \cdot j^{k-1}.$$

With \eqref{bd21}, this shows that the total contribution of characters of a fixed true level $1 \leq j < n/2$ 
to \eqref{abs bound} is at most
$$9q^j \cdot q^{-\sigma Bj/3} \cdot (2^{2k-1} \cdot k! \cdot j^{k-1})^2 < A(k)/q_1^j,$$
where $A(k) = 9 \cdot (2^{2k-1} \cdot k!)^2$ and $q_1:= q^{\sigma B/3-k}$ (here we use the estimate $q^j \geq 2^j \geq j^2$). 
Note that 
\begin{equation}\label{bd23}
  \sum^\infty_{j=1}1/q_1^j = 1/(q_1-1)
\end{equation}  
is $o(1)$ when $B$ is large enough.
Hence, the total contribution of characters $\chi$, of level at least $1$ and less than $n/2$, to \eqref{abs bound}, is less than 
$(q-\eps)A(k)/(q_1-1)=o(q-\eps)$, and the theorem follows.
\end{proof}

\subsection{Types $BCD$}
Recall the involution $f \mapsto f^\checkmark$ on the set $\cF_q^*$ of monic irreducible polynomial $f \in \FF_q[t]$ with $f(0) \neq 0$:
if $\deg(f)=m$ then $f^\checkmark(t)=t^mf(1/t)/f(0)$ (equivalently, 
$\lambda \in \overline{\FF_q}$ is a root of $f^\checkmark$ if and only if
$1/\lambda$ is a root of $f$). In the following theorem, the condition that the regular semisimple element $g \in G$ has pairwise distinct cycles implies that the characteristic polynomial of $g$ on the natural $\FF_qG$-module is of the form 
$$(t-1)^a(t+1)^b\prod^c_{i=1}f_if_i^\checkmark\prod^d_{j=1}h_j,$$ 
where $a,b,c,d \geq 0$, $a,b \leq 2$, $f_i^\checkmark \neq f_i \in \cF_q^*$, $h_j=h_j^\checkmark \in \cF_q^*$,
$\deg(f_i)=m_i$, $\deg(h_j) = n_j$,  
$m_1 > \ldots >m_c$, and $n_1 > \ldots > n_d$ (and $c+d \leq k \leq c+d+2$ for the cycle length $k$). 
Note that if the irreducible factors of the characteristic polynomial of $g$ have
pairwise distinct degrees, then $g$ has pairwise distinct cycles. 

\begin{thm}
\label{almost-thompson-BCD}
For all $k \in \ZZ_{\geq 1}$, there exists $B>0$ such that the following statement holds for all $n \in \ZZ_{\geq 1}$ and all prime powers 
$q$. If $G = \Sp_{2n}(q)$, $\SO_{2n+1}(q)$, or $\SO^\pm_{2n}(q)$, and $g \in G$ is a regular semisimple element with cycle length 
$k$ and pairwise distinct cycles,
then $g^{G}\cdot g^{G}$ contains every element $x \in [G,G]$ of support $\ge B$.  
\end{thm}

\begin{proof}
(a) Enlarging $B$, we are free to make $n \geq \max(k,5)$ as large as we wish. Write $G = \bG^F$ for a corresponding simple algebraic 
group of type $\Sp$ or $\SO$. Then $\bC_{\bG}(g)$ is a maximal torus, so, using the well-known structure of centralizers of semisimple
elements in the finite group $G$, we see that $T:=\bC_G(g)$ has order at most $2(q+1)^n$. 

Let $x\in [G,G]$. Since the linear characters of $G$ take value $1$ at $x$, to show that $x\in g^{G}\cdot g^{G}$, it suffices to prove that
\begin{equation}
\label{abs bound2}
\sum_{\{\chi \in \Irr(G) \mid \chi(1)>1\}} \frac{\abs{\chi(g)}^2\abs{\chi(x)}}{\chi(1)} < 1 \leq \abs{G/[G,G]}.
\end{equation}

\smallskip
(b) For any fixed $\epsilon > 0$, choosing $B$ sufficiently large, the contribution of characters $\chi \in \Irr(G)$ 
satisfying $\chi(1)\ge q^{\epsilon n^2}$ to \eqref{abs bound2} is $o(1)$.  Indeed, for any such character $\chi$, by \cite[Theorem 5.5]{LT} we have 
$$\abs{\chi(x)/\chi(1)} \leq \chi(1)^{-\sigma B/n} \leq q^{-\eps\sigma Bn}$$
for some absolute constant $\sigma > 0$.
Since $\abs{T} \leq 2(q+1)^n < q^{2n}$, it follows that the contribution of all these characters to \eqref{abs bound2} is at most
$$q^{-\eps\sigma Bn}\sum_{\chi}\abs{\chi(g)}^2 \leq q^{-\eps\sigma Bn}\abs{T} <  q^{2n(1-\eps\sigma B/2)}$$
which is $o(1)$ when $B$ is large enough.

\smallskip
(c) Now we consider any $\chi \in \Irr(G)$ with $1 < \chi(1) < q^{(n^2-4n)/4}$. By the Landazuri-Seitz bound 
\cite{LaSe}, we have $\chi(1) > q^{n/2}$. Let $j \in \ZZ_{\geq 1}$ be the unique integer such that  $q^{n(j-1)} \leq \chi(1) < q^{nj}$, and note
that 
$$q^{nj/2} \leq \chi(1) < \min\bigl(q^{nj},q^{(n^2-n)/2-4}\bigr),~~j < n/4.$$
By Proposition \ref{j-red}, there is an $F$-stable Levi subgroup $\bL= \bL_1 \times \bL_2$ of $\bG$, possibly equal to
$\bG$, such that the following statements hold:
\begin{enumerate}
\item[\rm $(\alpha)$] $\bL_1$ is $F$-stable, of the same type as of $\bG$, and of rank $n-m \geq n-2j$.
\item[\rm $(\beta)$] $\bL_2$ is $F$-stable, of type $\GL_m$ with $m \leq 2j$.
\item[\rm $(\gamma)$] $\chi = \pm R^\bG_\bL(\varphi_1 \boxtimes \varphi_2)$, where $\varphi_1 \in \Irr(\bL_1^F)$ is a unipotent or quadratic unipotent character, and $\varphi_2 \in \Irr(\bL_2^F)$.
\end{enumerate}
Moreover, by Theorem \ref{Degrees}, there is an absolute constant $C$ such that the total number $N_j$ of characters of $G$ of degree 
$\leq q^{nj}$ is at most 
\begin{equation}\label{bd30}
  N_j \leq q^{Cj}.
\end{equation}
Since $\chi(1) \geq q^{nj/2}$, $\supp(x) \ge B$ implies by \cite[Theorem 5.5]{LT} that
\begin{equation}\label{bd31}
  \abs{\chi(x)}/\chi(1) < q^{-\sigma Bj/2}.
\end{equation}  

\smallskip
Next we bound $\abs{\chi(g)}$, again using \eqref{bd2}. Let $V=\FF_q^d$ denote the natural module of $G$ endowed
with a symplectic or quadratic form, $d = 2n$ or $2n+1$, and let $L:=\bL^F$. Then the $L$-module $V$ is an orthogonal sum of two 
non-degenerate $L$-invariant subspaces $V_1:=\FF_q^{d-2m}$ and $V_2:=\FF_q^{2m}$, with 
$$2m \leq 4j < n \leq d-2m.$$ 
Furthermore,
$V_1$ is the natural, irreducible module of dimension $d-2m$ for $L_1:=\bL_1^F$, with $\bL_1$ of the same type as of $\bG$, and 
$L_1$ acts trivially on $V_2$. Next, $L_2:=\bL_2^F \cong \GL_m(q)$ or $\GU_m(q)$, with $V_1$ a minimal $L_2$-invariant 
non-degenerate subspace, and $L_2$ acts trivially on $V_1$.  
In particular, if $y \in \bN_G(L)$, then $y$ preserves each of $V_1$ and $V_2$, and thus 
$\bN_G(L)=L$.

Now we count the number $N$ of elements $y \in G$ such that $y^{-1}gy \in L$, i.e. $g \in yLy^{-1}$.  Then $g$ acts on each of the subspaces
$yV_1$ and $yV_2$. On the other hand, since $g$ has cycle length $k$ with pairwise distinct cycles, $V$ admits an orthogonal 
decomposition $V = \oplus^k_{i=1}U_i$, where each $U_i$ is a minimal $\langle g \rangle$-invariant non-degenerate 
subspace. Moreover, the $\langle g \rangle$-modules $U_i$ are pairwise non-isomorphic. Hence $(yV_1,yV_2)$ is uniquely
determined by choosing a subset of $\{U_1, \ldots,U_k\}$ (so that $yV_1$ is the sum over this subset and $yV_2$ is the sum over the 
complement). Thus the total number of possibilities for $yLy^{-1} = \bN_{G}(yV_1,yV_2)$ is at most $2^k$. 
On the other hand, $yLy^{-1} = y'Ly'^{-1}$ if and only if $y^{-1}y' \in \bN_G(L)=L$. Hence
\begin{equation}\label{bd32}
  N \leq 2^k\abs{L}.
\end{equation}    

Suppose $y^{-1}gy = \diag(g_1,g_2) \in L$, with $g_1 \in L_1$ and $g_2 \in L_2 =\GL^\pm_m(q)$. Let $k_i$ denote the cycle length of
$g_i$. Then $k_1 + k_2 \leq k$. Since $\varphi_1$ is quadratic unipotent and $g_1$ is regular semisimple,
we have 
$$\abs{\varphi_1(g_1)} \leq 2^{3k_1+4} \cdot k_1!$$
by Corollary \ref{bcd-bound}. On the other hand, when $m>0$, the statement $(\beta)$ and Corollary \ref{glu-bound2} show that
$$\abs{\varphi_2(g_2)} \leq k_2! \cdot (2j)^{k_2}.$$
As $y^{-1}gy$ is regular semisimple in $L$, $\abs{\St_L(y^{-1}gy)}=1$. Hence
$$\bigl{|}\bigl(\St_L \cdot (\varphi_1 \boxtimes \varphi_2)\bigr)(y^{-1}gy)\bigr{|} \leq  2^{3k+4} \cdot k! \cdot j^{k-1}.$$ 
It now follows from \eqref{bd2} and \eqref{bd32} that
$$\abs{\chi(g)} = \biggl{|}\frac{1}{\abs{L}}\sum_{y \in G:y^{-1}gy \in L}\pm (\St_L\cdot\varphi)(y^{-1}gy)\biggr{|}\leq 2^{4k+4} \cdot k! \cdot j^{k-1}.$$

With \eqref{bd31}, this shows that the total contribution of characters of degree satisfying 
$q^{n(j-1)} \leq \chi(1) < q^{nj}$ 
to \eqref{abs bound2} is at most
$$q^{Cj} \cdot q^{-\sigma Bj/2} \cdot (2^{4k+4} \cdot k! \cdot j^{k-1})^2 < A(k)/q_1^j,$$
where $A(k) = (2^{4k+4} \cdot k!)^2$ and $q_1:= q^{\sigma B/2-C-k+1}$ (here we again use $q^j \geq j^2$). Recalling 
\eqref{bd23},
we conclude that the total contribution to \eqref{abs bound2} of characters $\chi$, of degree at least $2$ and less than $q^{(n^2-4n)/4}$,   is less than 
$A(k)/(q_1-1)=o(1)$, and hence the theorem follows.
\end{proof}

\subsection{Another result for $\SL$}
For any positive interger $k$, let $\ua$ denote a fixed increasing sequence $a_1<\cdots<a_k$ of positive integers.  By an \emph{$\ua$-flag} in an $\FF_q$-vector space $V$, we mean a flag
$$V_1\subset\cdots\subset V_k\subset V$$
of $\FF_q$-subspaces such that $\dim V_i = a_i$ for $1\le i\le k$.  The number of $\ua$-flags is
$$F_{\ua}(\dim V) := \frac{\prod_{j=1}^{a_{k+1}}(q^j-1)}{\prod_{i=0}^k \prod_{j=1}^{a_{i+1}-a_i} (q^j-1)},$$
where we define $a_0:=0$ and $a_{k+1} := \dim V$.
As 
$$\prod_{j=1}^\infty (1-q^{-j}) > \frac 14$$
by \cite[Lemma 4.1]{LMT}, we have
\begin{equation}
\label{F-sub-a}
\frac{q^{d_{\ua}(N)}}4 <F_{\ua}(N) < 4^k q^{d_{\ua}(N)},
\end{equation}
where
$$d_{\ua}(N):= \sum_{0\le i < j\le k} (a_{i+1}-a_i)(a_{j+1}-a_j) = a_k(N-a_k) +  \sum_{0\le i < j\le k-1} (a_{i+1}-a_i)(a_{j+1}-a_j).$$
In particular, as $N$ goes to infinity,
\begin{equation}
\label{d-sub-a}
d_{\ua} = a_k N + O(1),
\end{equation}
where the implicit constant depends on $\ua$.
Moreover,
\begin{equation}
\label{ratios}
\frac{F_{\ua}(N+1)}{F_{\ua}(N)} = \frac{q^{N+1}-1}{q^{N+1-a_k}-1} = q^{a_k}+q^{-N+O(1)},
\end{equation}
so
\begin{equation}
\label{ratio-limit}
\lim_{N\to \infty} \frac{F_{\ua}(N+1)}{F_{\ua}(N)}  = q^{a_k}.
\end{equation}

\begin{lem}
\label{stable flags}
Let $k$ and $m$ be positive integers, and let $\ua$ be an increasing sequence of $k$ positive integers.  If $N$ is sufficiently large in terms of $m$ and $\ua$, then for all $g\in \GL_N(q)$
with $\supp(g) = m$,  the number of $g$-stable $\ua$-flags in $\FF_q^N$ can be written
$$q^{-a_k m}(1+\epsilon)F_{\ua}(N),$$
with $\abs{\epsilon} < q^{-N/2}$.
\end{lem}

\begin{proof}
We may assume $N > 3m$, so the eigenvalue $\lambda$ of multiplicity $N-m$ is unique and therefore lies in $\FF_q$.  Let $W_\lambda \subset \FF_q^N$ denote the generalized $\lambda$-eigenspace of $g$ and $W^\lambda$ the direct sum of the
generalized eigenspaces of $g$ for all eigenvalues other than $\lambda$.  Thus $\dim W^\lambda < N/3$.

If $V_1\subset\cdots \subset V_k$ is a $g$-stable $\ua$-flag, $V_k$ is determined by the decomposition 
$$V_k = (V_k\cap W_\lambda)\oplus (V_k\cap W^\lambda).$$
If $\dim V_k\cap W_\lambda < a_k$, then by applying \eqref{F-sub-a} to sequences of length $1$, we see that
 the number of possibilities for $V_k$ is less than 
$$\sum_{i=0}^{a_k-1} F_i(\dim W_\lambda)F_{a_k-i}(\dim W^\lambda) < 4^2\sum_{i=0}^{a_k-1} q^{i(\dim W_\lambda-i)}q^{(a_k-i)(\dim W^\lambda-a_k+i)} = q^{(a_k-1)N-O(1)},$$
so the total number of possibilities for the whole flag is less than $q^{(a_k-1)N-O(1)}$.

If $\dim V_k\cap W_\lambda = a_k$, then $V_k\subset W_\lambda$.
Let $I_\lambda$ and $K_\lambda$ denote the image of $\lambda - g$ on $W_\lambda$ and the kernel of $\lambda-g$ respectively.
Because $V_k$ is $g$-stable, either $V_k\subset K_\lambda$ or $V_k\cap I_\lambda\neq \{0\}$.  In the latter case, $V_k$ is spanned
by a non-zero vector in $I_\lambda$ and a subspace of $W_\lambda$ of dimension $a_k-1$.  As $\dim I_\lambda \le k < N/3$, the number of possibilities for
the whole flag is less than $q^{(a_k-1)N-O(1)}$.

Finally, we consider the number of possibilities when $V_k\subset K_\lambda$.  As $g$ acts on $K_\lambda$ as scalar multiplication, all $\ua$-flags with
$V_k\subset K_\lambda$ are $g$-stable.  The total number is
$$F_{\ua}(\dim K_\lambda) = F_{\ua}(N-m) = q^{a_k m}(1-q^{-N+O(1)})F_{\ua}(N),$$
by \eqref{ratios}.
The lemma follows.
\end{proof}

The unipotent characters of $\GL_N(q)$ are indexed by partitions $\lambda\vdash N$, and we say $\chi=\chi_\lambda$ has \emph{level $N-\lambda_1$}, where the parts of $\lambda$ are arranged from largest to smallest, see \cite[\S3]{GLT}. Also recall that,
for $\lambda,\mu \vdash N$, the Kostka number $K_{\lambda\mu}$ is the number of semistandard Young tableaux of shape $\lambda$ and weight $\mu$.

\begin{lem}
\label{stable}
Assuming $\mu_1\ge N/2$, $K_{\lambda\mu}$ depends only the partitions $(\lambda_2,\lambda_3,\ldots)$ and $(\mu_2,\mu_3,\ldots)$, obtained by removing the largest parts $\lambda_1$ and $\mu_1$ from $\lambda$ and $\mu$, and not on 
the value of $N$.
\end{lem}

\begin{proof}
In any semistandard Young tableaux of shape $\lambda$ and weight $\mu$,
the first $\mu_1$ entries of the first row must have filled with value $1$, and the remaining boxes in the first row are all to the right of every box in the remaining rows.
Therefore, such a tableau is determined by choosing from the $\mu_2$ values $2$, the $\mu_3$ values $3$, and so on, an arbitrary weakly increasing sequence for the $\lambda_1-\mu_1$ remaining boxes in the first row,
and from the values that remain, a semistandard Young tableau of shape $(\lambda_2,\lambda_3,\ldots)$.  The number of such choices depends only on $(\lambda_2,\lambda_3,\ldots)$ and
$(\mu_2,\mu_3,\ldots)$, but not on $N$.
\end{proof}

\begin{prop}
\label{low a}
Let $m$ and $n$ be fixed positive integers.  If $N$ is a positive integer sufficiently large in terms of $m$ and $n$, $\chi$ is a unipotent character of $\GL_N(q)$ of level $n$, and $g\in \GL_N(q)$ has support $m$, then
\begin{equation}
\label{good-estimate}
\Bigm|\frac{q^{mn}\chi(g)}{\chi(1)}-1\Bigm| < q^{-N/3}.
\end{equation}
\end{prop}

\begin{proof}
As $K_{\lambda\mu}$ is the number of semistandard Young tableaux of shape $\lambda$ and weight $\mu$,
we have $K_{\lambda \lambda}=1$, and if $K_{\lambda \mu}\neq 0$, then $\lambda$ dominates $\mu$: $\lambda \succeq \mu$.   
In particular, this implies that $\mu_1\le \lambda_1$.

For each $\mu \vdash N$, we define the increasing sequence $\ua_\mu$ of positive integers such that the sequence $a_1 = a_1-a_0,\ldots,a_{k+1}-a_k = N-a_k$ gives the parts of $\mu$ in increasing order.

Let $\phi_\mu$ denote the permutation character of $\GL_N(q)$ acting on the set of $\ua_\mu$-flags in $\FF_q^N$.
Then by  \cite[Lemma~2.4]{AT},
$$\phi_\mu = \sum_{\lambda \succeq \mu} K_{\lambda\mu} \chi_\lambda.$$
If $N\ge 2n$ and $\mu_1\ge N-n$, then by Lemma~\ref{stable}, $K_{\lambda\mu}$ depends only on $(\lambda_2,\lambda_3,\ldots)$
and $(\mu_2,\mu_3,\ldots)$.

As every Kostka matrix $K$ (for partitions of $N$) is unitriangular, we can invert and write $\chi=\chi_\lambda$ as a linear combination of permutation characters associated to $\phi_\mu$, where $\mu\succeq \lambda$.
We can therefore express each unipotent character of level $n$, including $\chi_\lambda$, as a linear combination of permutation characters $\chi_{\ub,N}$ associated to flags $\ub$ with maximal dimension $\le n$,
with coefficients which are entries in the inverse Kostka matrix $K^{-1}$.  

Note that, for any fixed $n$, the set of partitions $\lambda \vdash N$ with $\lambda_1 \geq N-n$ depends only on $n$, but not on $N$, $m$, or $q$. The unitriangularity of $K$ implies that 
the submatrix of $K^{-1}$, truncated to only partitions of $N$ with the first part $\geq N-n$, is the inverse of the submatrix of $K$, 
truncated to the same set of partitions. Applying Lemma \ref{stable}, we see that all entries of this truncated submatrix of $K^{-1}$ 
are bounded by some constant $O(1)$ that depends only on $n$:
\begin{equation}\label{inv1}
  \abs{(K^{-1})_{\lambda\mu}} \leq O(1)
\end{equation}  
whenever $\lambda,\mu \geq N-n$. Moreover,
\begin{equation}\label{inv2}
  \sum_{\mu \succeq \lambda}(K^{-1})_{\lambda\mu}\phi_\mu(1) = \chi_\lambda(1).
\end{equation}

Define $\epsilon_\mu$ so that 
\begin{equation}
\label{epsilon-sub-mu}
\phi_\mu(g)= q^{(\mu_1-N)m}(1+\epsilon_\mu)\phi_\mu(1).
\end{equation}
By Lemma \ref{stable flags}, for fixed $m$ and $\mu$, $\abs{\epsilon_\mu} < q^{-N/2}$ if  $N$ is sufficiently large and $g$ is of support $m$.
When $\mu_1 = \lambda_1$, we have $q^{(\mu_1-N)m} = q^{-mn}$. On the other hand, if 
$\mu_1 > \lambda_1$, then by \cite[Theorem 1.2(i)]{GLT} we have
$$\phi_\mu(1) \leq q^{(n-1)N}.$$
Now, by \eqref{inv1}, \eqref{inv2}, and \eqref{epsilon-sub-mu},
\begin{equation}
\label{estimate-at-g}
\begin{split}
\chi_\lambda(g) &= \sum_{\mu\succeq \lambda}(K^{-1})_{\lambda\mu} \phi_\mu(g) \\
                           & = \sum_{\mu\succeq \lambda} (K^{-1})_{\lambda\mu} (1+\epsilon_\mu)q^{-mn}\phi_\mu(1) 
                                            + \sum_{\substack{\mu\succeq\lambda\\ \mu_1>\lambda_1}} (K^{-1})_{\lambda\mu} (1+\epsilon_\mu)(q^{-m(N-\mu_1)}-q^{-mn})\phi_\mu(1)\\
                          &=  q^{-mn}\sum_{\mu\succeq \lambda} (K^{-1})_{\lambda\mu} (1+\epsilon_\mu)\phi_\mu(1)  + q^{(n-1)N+O(1)} \\
                          &=  q^{-mn}\sum_{\mu\succeq \lambda} (K^{-1})_{\lambda\mu}\phi_\mu(1) + q^{-mn}\sum_{\mu\succeq \lambda} \epsilon_\mu\abs{(K^{-1})_{\lambda\mu}}\phi_\mu(1) + q^{(n-1)N+O(1)},\\
&=  (1+\epsilon)q^{-mn}\chi_\lambda(1) + q^{(n-1)N+O(1)},                          
\end{split}
\end{equation}
where $\abs{\epsilon} \le \max_\mu \abs{\epsilon_\mu} < q^{-N/2}$. On the other hand,
$\chi_\lambda(1) > q^{nN-O(1)}$ by \cite[Theorem 1.2(i)]{GLT}. So
for $N$ sufficiently large, \eqref{estimate-at-g} implies \eqref{good-estimate}.
\end{proof}

\begin{prop}
\label{cancel}
Let $p \geq 3$ be a prime, and let $\cX$ denote the set of cuspidal characters of $\GL_p(q)$, i.e. characters of the form $I_p^k[1]$ in the notation of \cite{Green}.
Let $T < \GL_p(q)$ denote the centralizer of a semisimple element with characteristic polynomial irreducible over $\FF_q$ and $T^1 := T\cap \SL_p(q)$.  Let $z$ be a central element of $\SL_p(q)$.  If $t$ is a generator of $T^1$, then 
$$\sum_{\chi\in \cX}\chi(t)^2\chi(z) = p(1-q)\chi(1).$$
\end{prop}

\begin{proof}
First we claim that for all non-negative integers $a,b,c$, if 
$$c\equiv q^a+q^b\pmod{\abs{T^1}}$$ 
then 
\begin{equation}\label{cusp1}
  \gcd\bigl( c(q-1),\abs{T} \bigr)=q-1.
\end{equation}  
It is clear that $q-1$ divides both factors. 
If a prime $\ell$ divides 
$$\gcd(c(q-1),\abs{T}) = \gcd((q^a+q^b)(q-1),q^p-1),$$
then it divides $\abs{T}$, so it cannot divide $q$.  It must also divide either $q^a+q^b$ or $q-1$ or both.   If it divides $q-1$ but not $q^a+q^b$, then the highest power of
$\ell$ dividing $(q^a+q^b)(q-1)$ is the same as the highest power dividing $q-1$ and therefore the same as the highest power dividing $\gcd(c(q-1),\abs{T})$.  If it divides both, then
$0\equiv q^a+q^b \equiv 2\pmod \ell$, so $\ell=2$.  However, $\abs{T^1}$ is odd (as $p$ is odd), so the highest power of $2$ dividing $\abs{T}$ is the highest power dividing $q-1$ and therefore the highest
dividing $\gcd(c(q-1),\abs{T})$.  If $\ell$ divides only $q^a+q^b$, then we may replace $a$ and $b$  by their remainders under division by $p$ (since $\ell|(q^p-1)$), and assume $\ell$ divides $q^a+q^b$ for $0\le a\le b<p$ and therefore divides $1+q^{b-a}$ with $0\le b-a<p$.  We have already seen that the highest power of $2$ dividing $q-1$ is the same as the highest
power dividing $\gcd(c(q-1),\abs{T})$, so we may assume $b-a>0$.  Thus, the order of $q$ (mod $\ell$) divides $2(b-a)$ as well as $p$, so $q\equiv 1\pmod \ell$, contrary to assumption.

Let $t_0$ be a generator of the cyclic group $T$ with $t_0^{q-1} = t$, and let $c$ be as above. If $\phi$ is a character of $T$ and $\phi(t)^c= 1$, then \eqref{cusp1} implies $\phi^{q-1}(t_0)=1$ and so $\phi^{q-1}=1_T$.
Therefore, if $\phi^{q-1}\neq 1_T$, we have
$$\sum_{i=0}^{\abs{T}-1} \phi(t^{ic}) = \sum_{i=0}^{\abs{T}-1} \phi(t_0^{ic(q-1)}) = \sum_{i=0}^{\abs{T}-1} (\phi(t^c))^i= 0.$$
If $I$ denotes the set of $i\in \{0,\ldots,\abs{T}-1\}$ such that $i$ is not divisible by $\abs{T^1}$, then
\begin{equation}\label{cusp2}
  \sum_{i\in I} \phi^i(t^c) =\sum_{i\in I} \phi(t^{ic}) = \sum_{i=0}^{\abs{T}-1} \phi(t^{ic}) - \sum_{j=0}^{q-2} 1 = 1-q.
\end{equation}  

Now consider the action of $\ZZ/p\ZZ$ on the character group of $T$ which is generated by the map $\psi\mapsto \psi^q$.
All orbits are of length $p$ except for the singletons $\{\psi\}$ for which $\psi^{q-1}=1_T$. 
By \cite[p.~431]{Green}, the restriction of any character $\chi\in \cX$ to $T$ is of the form 
$$\chi(t) = \phi(t) + \phi(t^q) + \cdots + \phi(t^{q^{p-1}})$$
for some length $p$ orbit $\{\phi,\phi^q,\ldots,\phi^{q^{p-1}}\}$, and moreover different $\chi \in \cX$ correspond to different orbits of
length $p$.
If $g = zu$ with $u$ unipotent, then $\chi(g) = \phi(z)\chi(u)$.  Note that this is well defined because $z\neq 1$ implies $q\equiv 1\pmod p$
and $z^p=1$, which implies 
$$\phi^q(z) = \phi(z^{q-1})\phi(z) = \phi(z).$$
In particular, we have 
$$\chi(z) = \phi(z)\chi(1) = \phi(t_0^k)\chi(1),$$
where $k$ is some integer divisible by $\abs{T^1}$.

Therefore $\chi(t)^2\chi(z)/\chi(1)$ is a sum of $p^2$ terms of the form $\phi(t^{q^a+q^b+k})$.
Denoting by $\phi_0$ a generator of the character group of $T$, we see that the set of 
$\phi \in \Irr(T)$ with $\phi^{q-1} \neq 1_T$ is precisely $\{ \phi_0^i \mid i \in I \}$. Now using \eqref{cusp2} we have
\begin{align*}
\sum_{\chi\in \cX} \frac{\chi(t)^2 \chi(z)}{\chi(1)}
&= \frac 1p \sum_{\{\phi\mid \phi^{q-1}\neq 1_T\}} \sum_{a=0}^{p-1}\sum_{b=0}^{p-1} \phi(t^{q^a+q^b+k})\\
&= \frac 1p  \sum_{a=0}^{p-1}\sum_{b=0}^{p-1} \sum_{\{\phi\mid \phi^{q-1}\neq 1_T\}} \phi^{q^a+q^b+k}(t)\\
&= \frac 1p  \sum_{a=0}^{p-1}\sum_{b=0}^{p-1} \sum_{i\in I} \phi_0^i(t^{q^a+q^b+k}) \\
&= \frac 1p  \sum_{a=0}^{p-1}\sum_{b=0}^{p-1} (1-q) = p(1-q).
\end{align*}
\end{proof}

\begin{thm}\label{sl-coxeter}
For all but finitely many ordered pairs $(p,q)$ where $p\ge 3$ is prime and $q$ is a prime power, the following statement holds. If 
$t$ is a generator of the norm-$1$ subgroup $T_1 \cong C_{(q^p-1)/(q-1)}$ of $\FF_{q^p}^\times$
then every non-central element of $\SL_p(q)$ is a product of two conjugates of $t$.
\end{thm}

\begin{proof}
Fixing an $\FF_q$-basis of $\FF_{q^p}$ we can identity $\FF_{q^p}^\times$ with the centralizer $T$ in $G:=\GL_p(q)$ of any generator of $\FF_{q^p}$.
As $p$ is prime, every non-central element of $T$ generates $\FF_{q^p}$ as $\FF_q$-algebra, so no such element
is  contained in a proper parabolic subgroup of $G$.  Therefore, every Harish-Chandra induced character of $G$ vanishes on every element of $T \smallsetminus \bZ(G)$; in particular at our element $t$.
By \cite[(12)]{Green}, a primary (i.e. not Harish-Chandra induced) character of $G$ can be non-zero at $t$ if and only if it is of the form $I_1^k[p]$ or of the form $I_p^k[1]$.
In the first case, it belongs to the set $\cX$ of Proposition~\ref{cancel}.  In the second case, it is the product of a unipotent character and a linear character. Since $\bC_G(t) = T$, the conjugate classes of $t$ in $G$ and in $\SL_p(q)$ are the same.

By Theorem~\ref{almost-thompson-SLU}, we may assume that our target element $g$ has bounded support.
By the Frobenius formula, $(q-1)^{-1}$ times the number of representations of $g$ as a product of two conjugates of $t$ in $G$ is
$$(q-1)^{-1}\frac{\abs{t^G}^2\abs{g^G}}{\abs{G}}\sum_{\chi \in \Irr(G)} \frac{\chi(t)^2\bar\chi(g)}{\chi(1)}.$$
We divide this sum into a sum over $\chi$ which are unipotent characters times linear characters and a sum over $\chi\in \cX$.

For the first, we note that $t$ and $g$ are both in $\SL_p(\FF_q)$, and all linear characters of $G$ are trivial on this subgroup. So we can simply sum over unipotent characters and omit the  factor $(q-1)^{-1}$.
The contribution of the trivial character to the sum
$$\sum_{\chi \in \Irr(G)} \frac{\chi(t)^2\bar\chi(g)}{\chi(1)}$$
is $1$.  For the other unipotent characters, by \cite[Theorem~12]{Green}, $\chi_\lambda(t)$ is given by the value at a $p$-cycle of the character of the symmetric group $\sS_p$ associated to the partition $\lambda\vdash p$,
and by the Murnaghan-Nakayama rule, this value is $\pm 1$ if $\lambda$ is of the form $1^n (p-n)^1$ and $0$ otherwise.  By the main theorem of \cite{LT}, since $g \notin \bZ(G)$, there exists an absolute constant $\epsilon > 0$ such that
\begin{equation}\label{cusp3}
  \abs{\chi(g)} \leq \chi(1)^{1-\epsilon/p}
\end{equation}
for all $\chi \in \Irr(G)$.
By the dimension formula for primary characters of $G$ \cite[Lemma~7.4]{Green}, 
$$\chi_\lambda(1) \ge q^{n(p-n/2-1/2)}\ge q^{np/3}.$$
Therefore,
$$\frac{\abs{\chi_{1^n (p-n)^1}(g)}}{\chi_{1^n (p-n)^1}(1)} \le 2^{-n\epsilon/3}.$$
Choosing $A$ to be a sufficiently large absolute constant, we have
$\sum_{n\ge A} 2^{-n\epsilon/3} < \frac 13$, which guarantees 
$$\Bigm|\sum_{\chi=\chi_{1^n(p-n)^1},\,n \geq A} \frac{\chi(t)^2\bar\chi(g)}{\chi(1)}\Bigm| \leq \frac13.$$
On the other hand, if $p$ is large enough compared to $A$, then
by Proposition~\ref{low a},  for $1 \leq n < A$ we have $\chi_{1^n(p-n)^1}(g) > 0$, and 
$$\sum_{1\le n<A} \frac{\chi_{1^n (p-n)^1}(g)}{\chi_{1^n (p-n)^1}(1)}$$
is the sum of a positive term $\gamma$ and an error term less than $1/3$ in absolute value.
Recalling $\chi(t)^2=1$ on all these $\chi=\chi_{1^n (p-n)^1}$, it follows that the total 
contribution to the sum from nontrivial unipotent characters is $\gamma$ plus an error term less that $2/3$ in absolute value. 
Hence the total contribution to the sum from all unipotent characters is at least $\gamma + 1/3$. If $p$ is bounded but 
$q$ is large enough, then by Gluck's bound \cite{Gl}, there is some absolute constant $C > 0$ such that 
$$\Bigm|\sum_{1\le n<A} \frac{\chi_{1^n (p-n)^1}(g)}{\chi_{1^n (p-n)^1}(1)}\Bigm| <  \frac{CA}{\sqrt{q}} \leq \frac 13,$$
and this ensures that the total contribution to the sum from all unipotent characters is at least $1/3$.

As $p$ is prime, for $\chi \in \cX$, $\chi(g) = 0$ unless $g=zu$, where $z$ is scalar and $u$ is unipotent, 
or $g$ is conjugate to an element of $T$.  In the former case, $\chi(g) = c_u \frac{\chi(z)}{\chi(1)}$, where $c_u\in\ZZ$ depends only on $u$
but not on the particular $\chi \in \cX$.
By Proposition~\ref{cancel}, 
$$\frac 1{q-1} \Bigm|\sum_{\chi\in \cX} \frac{\chi(t)^2\bar\chi(g)}{\chi(1)} \Bigm| = \frac{p\abs{c_u}}{\chi(1)}.$$
As $g$ is not scalar, $\abs{c_u} = \abs{\chi(g)} \le \chi(1)^{1-\epsilon/p}$ by \eqref{cusp3}.  Since $\chi(1) = \prod_{i=1}^p (q^i-1) > q^{p(p-1)/2}/4$,
we get a uniform bound $\abs{c_u}/\chi(1) \leq q^{-\beta p}$ for some $\beta>0$.  Thus, the contribution of the cuspidal characters is less
than $1/6$ when either $p$ or $q$ is sufficiently large, implying the theorem.
\end{proof}

\end{document}